\documentclass[10pt]{article}

\usepackage{bbm}
\usepackage{amsfonts}
\usepackage{amsmath}
\allowdisplaybreaks
\usepackage{amssymb}
\usepackage{fancyhdr}
\usepackage{latexsym}
\usepackage{mathrsfs}
\usepackage{amsthm}
\usepackage{cite}
\usepackage{color}
\usepackage{enumerate}
\usepackage{shorttoc}

\newtheorem{theorem}{Theorem}[section]
\newtheorem{definition}[theorem]{Definition}
\newtheorem{lemma}[theorem]{Lemma}
\newtheorem{proposition}[theorem]{Proposition}
\newtheorem{remark}[theorem]{Remark}

\newtheorem{corollary}[theorem]{Corollary}

\makeatletter 
\renewcommand\theequation{\thesection.\arabic{equation}}
\@addtoreset{equation}{section}
\makeatother

\numberwithin{equation}{section}

\begin{document}
\begin{center}
\textbf{Inertial manifolds for the incompressible Navier-Stokes equations}\\
\end{center}
\begin{center}
{\sc Xinhua Li$^\dag$ and Chunyou Sun$^\dag$$^*$}
\end{center}

\renewcommand{\theequation}{\arabic{section}.\arabic{equation}}
\numberwithin{equation}{section}

\vspace{-0.3cm}
\begin{abstract}
In this article, we devote to the existence of an $N$-dimensional inertial manifold for the incompressible Navier-Stokes equations in $\mathbb{T}^{d}$ ($d=2,3$).
\par
Our results can be summarized as two aspects: Firstly, we construct an $N$-dimensional inertial manifold for the Navier-Stokes equations in $\mathbb{T}^{2}$; Secondly, we extend slightly the spatial averaging method to the abstract case: $\partial_{t}u+A^{1+\alpha}u+A^{\alpha}F(u)=f$ (here $0<\alpha<1$, $A>0$ is a self-adjoint operator with compact inverse and $F$ is Lipschitz from a Hilbert space $\mathbb{H}$ to $\mathbb{H}$), and then verify the existence of an $N$-dimensional inertial manifold for the hyperviscous Navier-Stokes equation with the hyperviscous index $5/4$ in $\mathbb{T}^{3}$.

\textbf{Keywords:} Inertial manifolds; 2D Navier-Stokes equations;  3D hyperviscous Navier-Stokes equations; Spectral gap condition; Spatial averaging method
\footnote[0]{\hspace*{-7.4mm}
$^*$Corresponding author\\
$^\dag$School of Mathematics and Statistics, Lanzhou University,
Lanzhou, 730000, P.R. China\\
AMS Subject Classification: 35B33, 35B40, 35B42, 35Q30, 76F20
\\{E-mails: xhli2014@lzu.edu.cn; sunchy@lzu.edu.cn}\\
This work was supported by the NSFC
(Grants No. 11471148, 11522109, 11871169).
}
\end{abstract}
\tableofcontents
\section{Introduction}\label{sec1}
\noindent

We study the $d$-dimensional incompressible Navier-Stokes equations (NSEs)
\begin{equation}\label{1.1}
\begin{cases}
\partial_{t}u+\nu(-\Delta)^{\theta} u+(u\cdot\nabla)u+\nabla p=f(x),~(x,t)\in\mathbb{T}^{d}\times\mathbb{R}^{+},\\
\nabla\cdot u=0,\\
u(0)=u_{0},
\end{cases}
\end{equation}
where $d=2~\text{or}~3$, the viscous index $\theta=(d+2)/4$, the constant kinematic viscosity $\nu>0$; $u(x,t)$ and $p(x,t)$ are unknowns, the initial condition $u(x,0)=u_{0}(x)$ is given and $f(x)$ is a known forcing function which is assumed to be time independent; $u=(u^{1},\cdots,u^{d})$ is the fluid velocity field and scalar $p$ represents the pressure. The nonlocal operator $(-\Delta)^{\theta}$ is defined in this paper via the Fourier multiplier with symbol $|j|^{2\theta}$, i.e.
\begin{equation*}
(-\Delta)^{\theta}\varphi(x)=\sum_{j\in\mathbb{Z}^{d}}|j|^{2\theta}\widehat{\varphi}_{j}e^{ij\cdot x},~~~\mbox{where}~~
\varphi(x)=\sum_{j\in\mathbb{Z}^{d}}\widehat{\varphi}_{j}e^{ij\cdot x}.
\end{equation*}
\par
The long-time behaviors of dissipative partial differential equations (PDEs) can be described by the so-called attractors, which is a compact subset of infinite dimensional phase space that attracts the images of all bounded sets when time tends to infinity. In many cases these attractors have finite Hausdorff and box-counting dimension \cite{R01,T97,MZ08}. Then we may believe that many dissipative PDEs in bounded domains is essentially finite-dimensional. The theory of inertial manifolds (IMs) give an evidence that some dissipative PDEs can be reduced to an ordinary differential equations (ODEs) in terms of the long-time dynamical behaviors, see \cite{FST85,M-PS88,GC05,FST88,HGT15,Z14} and references therein.
\par
To the best of our knowledge, the IM was firstly proposed by Foias, Sell and Temam \cite{FST85} in 1985. An inertial manifold (IM) is a smooth finite-dimensional invariant manifold that contains the global attractor and that attracts all the orbits at an exponential rate \cite{FST85,FST88,Z14}. The idea was employed on a large class of dissipative equations \cite{FST88} (see also \cite{T97,Z14}). A number of dynamical systems possess inertial manifolds, such as certain nonlinear reaction-diffusion equations in two \cite{CFNT89,CFNT89S,FST88} and three \cite{M-PS88} dimensions, the Kuramoto-Sivashinsky equation \cite{FNST88}, the Cahn-Hilliard equation \cite{CFNT89,KZ15}, modified Navier-Stokes equations \cite{HGT15,K18,GG18,LS19} and the von K\'{a}rm\'{a}n plate equations \cite{CL02}. One may refer to \cite{CFNT89S,Z14} for the study of inertial manifolds for many dissipative PDEs.
\par
It is worth mentioning that an original motivation for the theory of inertial manifolds was treating the NSEs. Unfortunately, the problems of existence of inertial manifolds for the two or three-dimensional NSEs are still unsolved \cite{Z14,T89}. The attempts of constructing the IMs for the 2D NSEs have been done by using the so-called Kwak transform \cite{K92,TW93}.

This paper is devoted to prove the existence of an $N$-dimensional inertial manifold for the incompressible NSEs \eqref{1.1} in $\mathbb{T}^{d}$ ($d=2,3$).

The classical theory \cite{FST88} of IMs shows that if the abstract model
\begin{equation}\label{1.2}
\partial_{t}u+\mathcal{A} u=\mathcal{F}(u),~~u|_{t=0}=u_{0}\in \mathbb{H}~(\text{Hilbert~space})
\end{equation}
satisfies the spectral gap condition
\begin{equation}\label{1.3}
\lambda_{N+1}-\lambda_{N}>2L,
\end{equation}
then \eqref{1.2} has an $N$-dimensional IM, here $\{\lambda_{N}\}_{N\in\mathbb{N}}$ are the eigenvalues of $\mathcal{A}$ ($\mathcal{A}>0$, self-adjoint and $\mathcal{A}^{-1}$ is compact), and $L>0$ is the Lipschitz constant of $\mathcal{F}$ ($\mathcal{F}$ is globally Lipschitz continuous from $\mathbb{H}$ to $\mathbb{H}$).

As $d=2$, the theorem of number theory from Richards \cite{R82} (see Lemma \ref{lem3.11} below) shows that the eigenvalues $\{\lambda_{N}\}_{N\in\mathbb{N}}$ of Stokes operator $P_{\sigma}(-\Delta)$ in $\mathbb{T}^{2}$ satisfies
\begin{equation*}
\lambda_{N+1}-\lambda_{N}\geq c\ln\lambda_{N}
\end{equation*}
for infinitely many $N\in \mathbb{N}$ and $c>0$, thus the problem \eqref{1.1} (with $d=2$) will possesses an $N$-dimensional IM if we can prove that $\mathcal{F}(\cdot)$ (here $\mathcal{F}(u)=(u\cdot\nabla)u$) is globally Lipschitz continuous from $H$ to $H$ with some Lipschitz constant $L$ (the function spaces $H$ and $H^{s}$ ($s\in \mathbb{R}$) are defined in Section \ref{sec2}). Unfortunately, we can not verify this. The successful example of using Richards's theorem to construct the IMs, we refer to Hamed, Guo and Titi \cite{HGT15}, in which, for the simplified Baridina model (which was introduced by Bardina, Ferziger and Reynolds in \cite{BFR80} and simplified by Layton and Lewandowski \cite{LL06})
\begin{equation}\label{1.4}
\begin{cases}
\partial_{t}u-\nu\Delta u+(\bar{u}\cdot\nabla)\bar{u}+\nabla p=f,\\
\nabla\cdot u=0,~~u=\bar{u}-\alpha^{2}\Delta\bar{u},
\end{cases}
\end{equation}
the authors in \cite{HGT15} built an $N$-dimensional IM for (\ref{1.4}) in $\mathbb{T}^{2}$.

On the existence of IMs for 3D case, the situations are more complex. It is not only the Lipschitz continuity of nonlinear term, but also Richards's theorem does not work as $d=3$. Fortunately, Mallet-Paret and Sell \cite{M-PS88} introduced the principle of spatial averaging/spatial averaging method (PSA/SAM) and proved the existence of an IM for the following abstract model
\begin{equation}\label{1.5}
\begin{cases}
\partial_{t}u+\mathcal{A}u+\mathcal{F}(u)=g,~~(t,x)\in\mathbb{R}_{+}\times\mathbb{T}^{3},\\
u|_{t=0}=u_{0}(x)
\end{cases}
\end{equation}
and applied it to the scalar reaction diffusion equation in $\mathbb{T}^{3}$.

Since moving a stationary point does not affect the long time behavior of \eqref{1.1}, here we have done that, and we will put our attention to
\begin{equation}\label{1.6}
\begin{cases}
\nu (-\Delta)^{\theta}w+B(w,w)+B(v,w)+B(w,v)=-\partial_{t}u,\\
w(0)=u_{0}-v:=w_{0},
\end{cases}
\end{equation}
where $w=u-v$, $u$ solves the problem \eqref{1.1} and $v$ is a fixed stationary solution of \eqref{1.1}. This approach, the so-called asymptotic regularity method (see e.g., \cite{S10}), can effectively improve the regularity of solution under low regular external forces. Indeed, the solution $w$ of \eqref{1.6} will be at least in $H^{9/2}$ if $d=2$ with $f\in H^{1}$ (see Theorem \ref{thm3.4} below) and $d=3$ with $f\in H$ only (see Theorem \ref{thm5.4}).
\par
In this paper, the basic idea for constructing an IM is that although we cannot directly prove that the original equation \eqref{1.6} (or \eqref{1.1}) has an inertial manifold, we can define a properly smooth truncation operator $W: H\rightarrow H$
\begin{equation}\label{1.7}
W(w)=\sum_{\stackrel{j\in\mathbb{Z}^{d}}{j\neq0}}
\frac{\varrho_{d}}{|j|^{9/2}}P^{j}_{\sigma}\vec{\eta}\bigg(\frac{|j|^{9/2}\widehat{w}_{j}}{\varrho_{d}}\bigg)e^{ij\cdot x},
\end{equation}
where $\varrho_{d}$ is the radius of the absorbing ball $\mathscr{B}_{d}$ defined by \eqref{3.18} for $d=2$ (or \eqref{5.50} for $d=3$), $P^{j}_{\sigma}$ are the Leray projector matrices defined by (\ref{2.1}) (or \eqref{2.2} w.r.t. $d=3$) and  $\vec{\eta}\in C^{\infty}_{0}(\mathbb{C}^{d})$ is a smooth cut-off function (this ideal follows from \cite{K18}); By this definition, for $d=2$ we can prove that $W(w)=w$ whenever $w\in\mathscr{B}_{2}$, $W$ is globally Lipschitz from $H$ to $H$ and $W$ maps $H$ into $H^{\frac{7}{2}-\varepsilon}$ continuously (see Theorem \ref{thm3.10} below); For $d=3$ we can verify that $W(w)=w$ as $w\in\mathscr{B}_{3}$, $W$ is globally Lipschitz from $H$ to $H$ and $W$ maps $H$ into $H^{3-\varepsilon}$ continuously (see Theorem \ref{thm5.11} below).
\par
Furthermore, as that in \cite{K18},  we define the so-called ``prepared'' equation as follows:
\begin{equation}\label{1.8}
\partial_{t}w+\nu A^{\theta}w+A^{(\theta-1)}F_{d}(W(w))=0,
\end{equation}
where $A:=P_{\sigma}(-\Delta)$ is the Stokes operator and the modified nonlinearity is defined by the following formula
\begin{equation*}
F_{d}(W(w)):=A^{(1-\theta)}\left[B(W(w),W(w))+B(W(w),v)+B(v,W(w))\right].
\end{equation*}

\noindent\textbf{Main results}

At first, for the case $d=2,3$, we prove that the ``prepared'' equation \eqref{1.8} has the same long-time behavior as the original problem \eqref{1.1} or \eqref{1.6}.
\begin{theorem}\label{thm1.1}
Let $d=2~\text{or}~3$. Then the ``prepared'' equation \eqref{1.8} has the same long-time behavior as the problem \eqref{1.1} or \eqref{1.6}.
\end{theorem}

For the case $d=2$, the Eqs. \eqref{1.8} will be
\begin{equation}\label{1.9}
\partial_{t}w+\nu Aw+F_{2}(W(w))=0
\end{equation}
with
\begin{equation*}
F_{2}(W(w)):=B(W(w),W(w))+B(W(w),v)+B(v,W(w)).
\end{equation*}
Thanks to the truncation \eqref{1.7}, we can prove that the modified nonlinearity $F_{2}(W(\cdot)): H\rightarrow H$ is globally bounded, globally Lipschitz continuous with some Lipschitz constant $L$ (also in Theorem \ref{thm3.10}). Combining above analysis with Richards's theorem, we will see that the spectral gap condition \eqref{1.3} holds, which ensures us to build an IM for NSEs in $\mathbb{T}^{2}$.
\begin{theorem}\label{thm1.2}
Let $d=2$ with $\theta=1$. Assume that $f\in H^{1}$ and $u_{0}\in H$. Then the ``prepared'' equation \eqref{1.8} of the original problem \eqref{1.1} possesses an $N$-dimensional inertial manifold $\mathcal{M}$ in the sense of Definition \ref{def2.4} which contains the global attractor $\mathscr{A}$.
\end{theorem}

As $d=3$, the Eqs. \eqref{1.8} should be
\begin{equation}\label{1.10}
\partial_{t}w+\nu A^{5/4}w+A^{1/4}F_{3}(W(w))=0
\end{equation}
with
\begin{equation}\label{1.11}
F_{3}(W(w)):=A^{-1/4}\left[B(W(w),W(w))+B(W(w),v)+B(v,W(w))\right].
\end{equation}
Thanks to the truncation \eqref{1.7} again, we can verify that the modified nonlinearity $F_{3}(W(\cdot)): H\rightarrow H$ is globally bounded, globally Lipschitz continuous and Gateaux differential (see also Theorem \ref{thm5.11}). However, since $d=3$ and the structure of the Eqs. \eqref{1.10} is obviously different from \eqref{1.9}, we can not expect to construct an IM for \eqref{1.10} by adopting the Richards's theorem and the spectral gap theory of IMs. In the present work, we extend slightly SAM to the following abstract model
\begin{equation}\label{1.12}
\begin{cases}
\partial_{t}u+\mathcal{A}^{1+\alpha}u+\mathcal{A}^{\alpha}\mathcal{F}(u)=g,~~(t,x)\in\mathbb{R}_{+}\times\mathbb{T}^{3},\\
u|_{t=0}=u_{0},
\end{cases}
\end{equation}
and prove the existence of IMs for \eqref{1.12}, where $\alpha\in(0,1)$ (see Section \ref{sec4}; it is not only applicable to 3D hyperviscous NSEs but also is of independent interest):
\begin{theorem}\label{thm1.3}
For the abstract problem \eqref{1.12} defined in the Hilbert space $\mathbb{H}$, let $u_{1}$ and $u_{2}$ be two solutions of \eqref{1.12} in $\mathbb{H}$. Set $v=u_{1}-u_{2}$. Assume that $\mathcal{F}(\cdot): \mathbb{H}\rightarrow \mathbb{H}$ is Gateaux differentiable with $\|\mathcal{F}'(u)\|_{\mathcal{L}(\mathbb{H},\mathbb{H})}\leq L$ for any $u\in \mathbb{H}$, and for some $L\geq 1$. Suppose that there exist $N\in\mathbb{N}$ and $k\in[\hbar \log\lambda_{N},\frac{\lambda_{N}}{2})$ for some $\hbar\in(0,1/2]$ such that $\lambda_{N}\geq e^{60L^{2}/\hbar}$ with $1\leq\lambda_{N+1}-\lambda_{N}\leq 2L$, and the spatial averaging condition holds:
\begin{equation}\label{1.13}
\|R_{k,N}\mathcal{F}'(u)R_{k,N}v\|\leq\delta\|v\|,~~\mbox{for all}~~u\in \mathbb{H},
\end{equation}
for some $\delta\leq\frac{1}{30}$, here $R_{k,N}$ is the projector defined by \eqref{4.3}. Then the \eqref{1.12} possesses an $N$-dimensional inertial manifold $\mathcal{M}$ in the sense of Definition \ref{def2.4}.
\end{theorem}

However, note that the SAM will not work for vector-nonlinearity, because the PSA is based on a crucial property of the Schr\"{o}dinger operator $\Delta+\upsilon(x)$ that it can be well approximated by the constant coefficient problem $\Delta+\langle\upsilon\rangle$ over large segments of $H$ (see \cite{M-PS88}), where $\upsilon(x)=F'(u(x))$ and $\langle\upsilon\rangle=\frac{1}{(2\pi)^{3}}\int_{\mathbb{T}^{3}}\upsilon dx$, thus, it is impossible to find a scalar operator $\langle F'(u)\rangle$ to fulfil the property mentioned above for the vector equations. Fortunately, the SAM can successfully apply to our problem thanks to the specific form of $F_{3}(W(w))$ (see \eqref{1.11}) which allows us to treat $F'(W(w))$ as a finite sum of multiplication operators which possess zero mean of their generators (see Section \ref{sec5.3.3} below). Finally, owing to another result of number theory (see Proposition \ref{pro5.12}), we can verify that the SAC (\ref{5.60}) holds (see Theorem \ref{thm5.13} below). Based on the above analysis, we can construct an $N$-dimensional IM for NSEs \eqref{1.1} in the case $d=3$.
\begin{theorem}\label{thm1.4}
Let $d=3$ with $\theta=\frac{5}{4}$. Assume that $f\in H$ and $u_{0}\in H$. Then the ``prepared'' equation \eqref{1.8} of the problem \eqref{1.1} possesses an $N$-dimensional inertial manifold $\mathcal{M}$ in the sense of Definition \ref{def2.4} which contains the global attractor $\mathscr{A}$.
\end{theorem}

To our knowledge, the best known result for the 3D hyperviscous NSEs \eqref{1.1} was given in Gal and Guo \cite{GG18} with the assumption $\theta\geqslant \frac{3}{2}$.

\begin{remark}\label{rem1.5}
It follows from Theorem \ref{thm1.2} and Theorem \ref{thm1.4} that there exists a Lipschitz map $\Phi: H\rightarrow H$ such that the attractor $\mathscr{A}$ is entirely determined by a $N$-dimensional system of ODEs:
\begin{equation}\label{1.14}
\frac{dp}{dt}+\nu A^{\theta}p=P_{N}(A^{(1-\theta)}F(W(p+\Phi(p)))),
\end{equation}
where $P_{N}$ denotes the orthogonal projector defined by \eqref{2.8}, $p=P_{N}w$ and this ODEs is called `inertial form' of \eqref{1.8}. Moreover, we point out that the asymptotic behaviors of \eqref{1.1} is completely determined by \eqref{1.14} as well. Indeed, we know that $w=p+\Phi(p)$ when $w\in\mathcal{M}$ and that the attractor $\mathscr{A}$ is a subset of $\mathcal{M}$ and $\mathcal{M}$ is positive invariant, \eqref{1.14} must have a global attractor $\mathscr{A}_{p}=P_{N}\mathscr{A}$, and so the global attractor of \eqref{1.1} $\mathscr{A}_{0}=\mathscr{A}+v(x)$ (see Corollary \ref{cor3.8} for 2D and Corollary \ref{cor5.8} for 3D) is now completely determined by the dynamics on $\mathscr{A}_{p}$ because $v(x)$ is fixed.
\end{remark}
\par
Throughout this paper, we denote by $\mathbb{R}$, $\mathbb{Z}$ and $\mathbb{C}$ the set of real numbers, integer numbers and complex numbers, respectively. $C$, $c$ and $C_{i}$ without a subscript will stand for some constants that may change line to line.

\section{Preliminaries}\label{sec2}
\noindent

Denote the $d$ dimensional torus by $\mathbb{T}^{d}:=[-\pi,\pi]^{d}$ and set $\mathbb{Z}^{d}_{\ast}=\mathbb{Z}^{d}\setminus\{0\}$ (recall $d=2,3$). Let $H^{s}(\mathbb{T}^{d})$ be the classical Sobolev space on $\mathbb{T}^{d}$. Then each $u\in(L^{2}(\mathbb{T}^{d}))^{d}$ can be written as the Fourier series
\begin{equation*}
u(x)=\sum_{j\in\mathbb{Z}^{d}}\widehat{u}_{j}e^{ij\cdot x},~~\widehat{u}_{j}=\frac{1}{(2\pi)^{d}}\int_{\mathbb{T}^{d}}u(x)e^{ix\cdot j}\in\mathbb{C}^{d}.
\end{equation*}
Due to the Plancherel theorem,
\begin{equation*}
\|u\|^{2}_{H^{s}(\mathbb{T}^{d})}=|\widehat{u}_{0}|^{2}+\sum_{j\in\mathbb{Z}^{d}_{\ast}}|j|^{2s}|\widehat{u}_{j}|^{2},
\end{equation*}
where $|j|^{2}=j^{2}_{1}+\cdots+j^{2}_{d}$.
The Helmholtz-Leray orthoprojector $P_{\sigma}: (L^{2}(\mathbb{T}^{d}))^{d}\rightarrow H:=P_{\sigma}(L^{2}(\mathbb{T}^{d}))^{d}$ to divergent free vector fields with zero mean can be defined as follows
\begin{equation*}
P_{\sigma}u:=\sum_{j\in\mathbb{Z}^{d}_{\ast}}P^{j}_{\sigma}\widehat{u}_{j}e^{ij\cdot x},~~u=\sum_{j\in\mathbb{Z}^{d}}\widehat{u}_{j}e^{ij\cdot x}
\end{equation*}
and the $d\times d$-matrices $P^{j}_{\sigma}$ are defined by
\begin{equation}\label{2.1}
P^{j}_{\sigma}:=\frac{1}{|j|^{2}}\left(
                          \begin{array}{cc}
                            j^{2}_{2} & -j_{1}j_{2} \\
                            -j_{1}j_{2} & j^{2}_{1} \\
                          \end{array}
                        \right)~~\text{as}~d=2;
\end{equation}
\begin{equation}\label{2.2}
P^{j}_{\sigma}:=\frac{1}{|j|^{2}}\left(
                          \begin{array}{ccc}
                            j^{2}_{2}+j^{2}_{3} & -j_{1}j_{2} & -j_{1}j_{3} \\
                            -j_{1}j_{2} & j^{2}_{1}+j^{2}_{3} & -j_{2}j_{3} \\
                            -j_{1}j_{3} & -j_{2}j_{3} & j^{2}_{1}+j^{2}_{2} \\
                          \end{array}
                        \right)~~\text{as}~d=3.
\end{equation}
\par
We recall that the Stokes operator $A:=-P_{\sigma}\Delta$ can be viewed as the restriction of Laplacian to the divergence free vector fields and the domain $D(A)$ is given by
\begin{equation*}
D(A):=\left\{u\in (H^{2}(\mathbb{T}^{d}))^{d}: \nabla\cdot u=0,~\langle u\rangle=0\right\}.
\end{equation*}
Define the work space $H^{s}:=D(A^{s/2})$, $s\in\mathbb{R}$. Then
\begin{equation*}
Au=-P_{\sigma}\Delta u=-\Delta u,~\forall~u\in D(A);~~H^{s}=\left\{u\in (H^{s}(\mathbb{T}^{d}))^{d}: \nabla\cdot u=0,~\langle u\rangle=0\right\}
\end{equation*}
(see, e.g., \cite{T97}) and due to the Parseval equality \cite{G08},
\begin{equation*}
\|u\|^{2}_{H^{s}}=\sum_{j\in\mathbb{Z}^{d}_{\ast}}|j|^{2s}|\widehat{u}_{j}|^{2},~~u\in H^{s}.
\end{equation*}
Denote the standard bilinear form associated with the NSEs by
\begin{equation*}
B(u,v):=P_{\sigma}((u\cdot\nabla)v),~~\forall~u,v\in H^{1}.
\end{equation*}
Applying the Helmholtz-Leray orthogonal projection $P_{\sigma}$ to 2D NSEs, that is, the Eqs. \eqref{1.1} in the case $d=2$, we obtain
\begin{equation}\label{2.3}
\begin{cases}
\partial_{t}u+\nu A u+B(u,u)=f,\\
u(0)=u_{0};
\end{cases}
\end{equation}
and applying the Helmholtz-Leray orthogonal projection $P_{\sigma}$ to Eqs. (\ref{1.1}) in the case $d=3$ gives that
\begin{equation}\label{2.4}
\begin{cases}
\partial_{t}u+\nu A^{5/4} u+B(u,u)=f,\\
u(0)=u_{0}.
\end{cases}
\end{equation}
\begin{lemma}[\cite{R01,T97}]\label{lem2.1}
For $d=2,3$, the bilinear form $B:H^{1}\times H^{1}\rightarrow H^{-1}$ is continuous and satisfies the following estimates:
\begin{equation}\label{2.5}
|(B(u,v),w)|\leq c\|u\|_{H^{1}}\|v\|^{1/2}_{H^{1}}\|Av\|^{1/2}_{H}\|w\|_{H},~\forall~u\in H^{1},v\in H^{2},w\in H;
\end{equation}
\begin{equation}\label{2.6}
|(B(u,v),w)|\leq c\|u\|_{H^{1/2}}\|v\|_{H^{1}}\|w\|_{H^{1}},~\forall~u\in H^{1/2},v,w\in H^{1};
\end{equation}
\begin{equation}\label{2.7}
(B(u,v),v)=0,~\forall~u,v\in H^{1}.
\end{equation}
\end{lemma}

We recall the following interpolation inequality, e.g., see \cite{CD00} for proof.
\begin{lemma}[Nirenberg-Gagliardo inequality]\label{lem2.2}
Let $1<p,~p_{1},~p_{2}<\infty$, $0\leq s_{2}<s_{1}<\infty$, $r\in (0,1)$ and $s$, $p$ satisfy $s=s_{1}+(1-r)s_{2}$, $\frac{1}{p}=\frac{r}{p_{1}}+\frac{1-r}{p_{2}}$. Then the following estimate holds:
\begin{equation*}
\|u\|_{W^{s,p}(\mathbb{T}^{d})}
\leq c_{r}\|u\|^{r}_{W^{s_{1},p_{1}}(\mathbb{T}^{d})}\|u\|^{1-r}_{W^{s_{2},p_{2}}(\mathbb{T}^{d})},~~\forall~u\in W^{s_{1},p_{1}}(\mathbb{T}^{d})\cap W^{s_{2},p_{2}}(\mathbb{T}^{d}),
\end{equation*}
where $W^{k,\wp}(\mathbb{T}^{d})$ $(k=s,s_{1},s_{2}, \wp=p,p_{1},p_{2})$ denote the classical Sobolev spaces.
\end{lemma}
\par
We recall the definition of a global attractor.
\begin{definition}[\cite{T97}]\label{def2.3}
Let $\mathbb{H}$ is a Hilbert space. A set $\mathscr{A}$ is said to be a global attractor for the semigroup $S(t)$ in $\mathbb{H}$ if
\begin{enumerate}[(i)]
  \item $\mathscr{A}$ is compact in $\mathbb{H}$;
  \item $\mathscr{A}$ is invariant, i.e. $S(t)\mathscr{A}=\mathscr{A}$ for all $t\geq 0$;
  \item $\mathscr{A}$ attracts each bounded subset $\mathbb{H}$, that is, for any bounded set $B\subset \mathbb{H}$, $dist(S(t)B,\mathscr{A})\rightarrow 0$ as $t\rightarrow +\infty$, where $dist(A,C)$ denotes the Hausdorrf semidistance between $A$ and $C$ in $\mathbb{H}$, defined as
      \begin{equation*}
       dist(A,C)=\sup_{a\in A}\inf_{c\in C}d_{H}(a,c).
      \end{equation*}
\end{enumerate}
\end{definition}

Now, we introduce some notions on the inertial manifolds. To do this, we fix an integer $N$ and denote $P_{N}$ the orthogonal projector onto the space spanned by the first $N$ eigenvectors of $\mathcal{A}$ in $\mathbb{H}$ (the corresponding eigenvector $\{e_{k}\}_{k=1}^{\infty}$ forms orthonormal basis of $\mathbb{H}$ ), that is,
\begin{equation}\label{2.8}
P_{N}u=\sum_{k\leq N}u_{k}e_{k},~~u_{k}=(u,e_{k}),~~u\in \mathbb{H}.
\end{equation}
Let $Q_{N}=I-P_{N}$, then
\begin{equation*}
Q_{N}u=\sum_{k> N}u_{k}e_{k}
\end{equation*}
and, for every $u\in H$ then $u=p+q$ where $p=P_{N}u$ and $q=Q_{N}u$.
\par
We introduce the following definitions, one can see \cite{FNST88,FST88,T97} and references therein.
\begin{definition}\label{def2.4}
Let $\mathbb{H}$ be a Hilbert space. A subset $\mathcal{M}\subset \mathbb{H}$ is called an inertial manifold for a dynamical system in $\mathbb{H}$ associated with the semigroup $S(t)$ if the following conditions are satisfied:
\begin{enumerate}[(i)]
\item $\mathcal{M}$ is a finite-dimensional Lipschitz manifold, i.e., there exists a Lipschitz continuous function $\Phi: P_{N}\mathbb{H}\rightarrow Q_{N}\mathbb{H}$ such that
   \begin{equation*}
   \mathcal{M}=\{p+\Phi(p): p\in P_{N}\mathbb{H}\};
   \end{equation*}
\item $\mathcal{M}$ is invariant with respect to $S(t)$, i.e., $S(t)\mathcal{M}\subseteq\mathcal{M}$ for all $t\geq 0$;
\item $\mathcal{M}$ is exponentially attracting, i.e., there exist positive constants $\omega$ and $C$ such that for every $u_{0}\in \mathbb{H}$ there exists a $v_{0}\in \mathcal{M}$ such that
    \begin{equation*}
    \|S(t)u_{0}-S(t)v_{0}\|_{\mathbb{H}}\leq Ce^{-\omega t}\|u_{0}-v_{0}\|_{\mathbb{H}},~ for~all~t\geq 0.
    \end{equation*}
\end{enumerate}
\end{definition}
To verify the existence of the IM, we will use the notation of strong cone condition introduced by Kostianko and Zelik in \cite{KZ15} (one see also \cite{Z14}). Firstly, we introduce the following quadratic form in $H$:
\begin{equation*}
V(\xi)=V_{N}(\xi):=\|Q_{N}\xi\|_{\mathbb{H}}^{2}-\|P_{N}\xi\|_{\mathbb{H}}^{2},~~\xi\in \mathbb{H},
\end{equation*}
and set $K^{+}:=\{\xi\in \mathbb{H}, V(\xi)\leq 0\}$ to be the associated cone.
\begin{definition}[Strong cone condition]\label{def2.6}
Let $u_{1},u_{2}$ be two solutions of \eqref{1.1} with initial data $u_{1,0},u_{2,0}\in \mathbb{H}$ respectively. If the following estimate hold:
\begin{equation}\label{2.9}
\frac{1}{2}\frac{d}{dt}V(u_{1}(t)-u_{2}(t))+\gamma V(u_{1}(t)-u_{2}(t))
\leq-\mu\|u_{1}(t)-u_{2}(t)\|_{\mathbb{H}}^{2},
\end{equation}
where $\gamma$, $\mu>0$. Then, we say that \eqref{1.1} satisfies the strong cone condition.
\end{definition}
\begin{definition}\label{def2.5}
 $(i)$ We say that the semigroup $S(t)$ generated by the equation \eqref{1.1} possesses the cone invariance if for any $\xi_{1}$, $\xi_{2} \in \mathbb{H}$ such that
\begin{align}\label{2.10}
&\xi_{1}-\xi_{2}\in K^{+}\nonumber\\
\Longrightarrow~~&S(t)\xi_{1}-S(t)\xi_{2}\in K^{+}~for~all~t\geq 0.
\end{align}
$(ii)$ The solution semigroup $S(t)$ possesses the squeezing property if there exists positive $\vartheta$ and $C$ such that
\begin{align}\label{2.11}
&S(T)\xi_{1}-S(T)\xi_{2}\notin K^{+}\nonumber\\
\Longrightarrow~~&\|S(t)\xi_{1}-S(t)\xi_{2}\|_{\mathbb{H}}\leq Ce^{-\vartheta t}\|\xi_{1}-\xi_{2}\|_{\mathbb{H}},~t\in~[0,T].
\end{align}
\end{definition}
\par
We recall the following formulation of the uniform Gronwall lemma, it's proof can be found in \cite{T97}.
\begin{lemma}\label{lem2.7}
Let $\Psi$, $h$ and $Y$ be nonnegative locally integrable functions on $(t_{0},+\infty)$ such that
\begin{equation*}
\frac{dY}{dt}\leq \Psi Y+h~\mbox{for}~t\geq t_{0}.
\end{equation*}
Then, it holds
\begin{equation*}
Y(t+1)\leq
\exp\left(\int^{t+1}_{t}\Psi(\tau)d\tau\right)\left(\int^{t+1}_{t}Y(\tau)d\tau+\int^{t+1}_{t}h(\tau)d\tau\right),
~\mbox{for~all}~t\geq t_{0}.
\end{equation*}
\end{lemma}
\par
We conclude this section by the following lemma on compactness of sequences. The proof of this lemma can be founded in \cite{A63,D65,L69,C15}.
\begin{lemma}[Aubin-Dubinskii-Lions]\label{lem2.8}
Assume that $X\subset Y\subset Z$ is a triple of Banach spaces such that $X$ is compactly embedded in $Y$. Let $\{u_{n}\}_{n\in\mathbb{Z}}$ be a bounded sequence in $L^{p}([a,b];X)$ for some $1\leq p<\infty$ such that the set $\partial_{t}\{u_{n}\}_{n\in\mathbb{N}}:=\{\partial_{t}u_{n}:~n\in \mathbb{N}\}$ is bounded in $L^{q}([a,b];Z)$ for some $q\geq 1$. Here $\partial_{t}u_{n}$ is the derivative in the distributional sense. Then $\{u_{n}\}_{n\in\mathbb{Z}}$ is relatively compact in $L^{p}([a,b];Y)$. Moreover, if $\{u_{n}\}_{n\in\mathbb{Z}}$ is a bounded set in $L^{\infty}([a,b];X)$ and $\partial_{t}\{u_{n}\}_{n\in\mathbb{Z}}$ is bounded in $L^{r}([a,b];Z)$ for some $r>1$, then $\{u_{n}\}_{n\in\mathbb{Z}}$ is relatively compact in $C([a,b];Y)$. In particular, the above statements are valid whenever $p=q=r=2$.
\end{lemma}
\section{IM for the 2D NSEs}\label{sec3}

The main purpose of this section is to establish the existence of IMs for the 2D NSEs on $\mathbb{T}^{2}$.
\subsection{A priori estimates}\label{sec3.1}
\noindent

In this subsection, we give some a priori estimates, including $H^{3}$-estimate on the solution $v$ of stationary equation of Eqs. \eqref{2.3}, $H^{2}$-estimate on the solution $u$ of Eqs. \eqref{2.3} and $H^{9/2}$-estimate on difference between $u$ and $v$ (see \eqref{3.3}).
\subsubsection{$H^{3}$-estimate on the solution of stationary equation}\label{sec3.1.1}
\noindent

We consider the stationary problem of Eqs. \eqref{2.3}:
\begin{equation}\label{3.1}
\nu Av+B(v,v)=f.
\end{equation}
We intend to prove the $H^{3}$-regularity of the solution of the auxiliary Eqs. \eqref{3.1}.
\begin{lemma}\label{lem3.1}
Let $f\in H^{1}$. Then, there exists at least one weak solution $v\in H^{1}$ of the stationary problem \eqref{3.1}. Moreover, the solution $v$ of \eqref{3.1} belongs to $H^{3}$ such that there exist $\rho_{v}>0$  and a constant $c$ (depends only on $\nu$) satisfies the following estimate:
\begin{equation*}
\|v\|_{H^{3}}\leq \rho_{v},
\end{equation*}
where
\begin{equation*}
\rho_{v}=c\left(\|f\|^{2}_{H}+\|f\|_{H^{1}}\right).
\end{equation*}
\end{lemma}
\begin{proof}
Multiplying \eqref{3.1} through by $Av$ and integrate over $\mathbb{T}^{2}$, we have
\begin{equation*}
\nu\|v\|^{2}_{H^{2}}+(B(v,v),Av)=(f,Av)
\leq \frac{2}{\nu}\|f\|^{2}_{H}+\frac{\nu}{2}\|v\|^{2}_{H^{2}}.
\end{equation*}
Then, thanks to $(B(v,v),Av)=0$ in $\mathbb{T}^{2}$ (see, e.g. \cite{R01}), we obtain
\begin{equation*}
\|v\|^{2}_{H^{2}}\leq\frac{4}{\nu^{2}}\|f\|^{2}_{H}.
\end{equation*}
Thus, combining this with the equality \eqref{3.1}, we estimate
\begin{equation*}
\nu\|Av\|_{H^{1}}\leq\|B(v,v)\|_{H^{1}}+\|f\|_{H^{1}}
\end{equation*}
and
\begin{equation*}
\|B(v,v)\|_{H^{1}}\leq \bar{c}\|v\|_{L^{\infty}}\|\nabla \cdot v\|_{H^{1}}\leq C\|v\|^{2}_{H^{2}}
\leq C_{\nu}\|f\|^{2}_{H},
\end{equation*}
where $\bar{c}$ is the embedding constant of $H^{2}(\mathbb{T}^{2})\hookrightarrow L^{\infty}(\mathbb{T}^{2})$. This implies that
\begin{equation*}
\|v\|_{H^{3}}\leq C_{\nu}(\|f\|^{2}_{H}+\|f\|_{H^{1}}).
\end{equation*}
\end{proof}
\begin{remark}
What we have done in the proof of this lemma is a formal calculation, which can be made rigorous using a Galerkin truncation, e.g., see \cite{T95,R01}. We will perform a formal calculation in the rest parts of this paper as well.
\end{remark}
\subsubsection{$H^{2}$-estimate on the solution of Eqs. (\ref{2.3})}\label{sec3.1.2}
\begin{theorem}\label{thm3.2}
Let $f\in H$ and $u_{0}\in H$. If $u(t)$ is a weak solution of \eqref{1.1}, then there exists  $t_{0}:=t_{0}(\|u_{0}\|_{H})$  such that
\begin{equation*}
\|u(t)\|_{H^{s}}\leq \rho_{s},~~\forall~t\geq t_{0},~~s=1,2,
\end{equation*}
where $\rho_{s}$ depend only on $\nu$ and $\|f\|_{H}$.
\end{theorem}
\begin{proof}
This proof can be found in the page 313 of \cite{R01}.
\end{proof}
\begin{remark}\label{rem3.3}
It follows from Theorem \ref{thm3.2} that there exists $\bar{\rho}_{0}>0$ such that
\begin{equation}\label{3.2}
\|\partial_{t}u(t)\|_{H}\leq \bar{\rho}_{0},
~~~~\int^{t+1}_{t}\|\partial_{t}u(s)\|^{2}_{H^{1}}ds\leq C_{\nu}\rho^{2}_{2}\bar{\rho}^{2}_{0},~~\forall t\geq t_{0}+1.
\end{equation}
Indeed, differentiating Eqs. (\ref{2.3}) with respect to $t$  and taking the inner product with $\partial_{t}u$, we obtain
\begin{equation*}
\frac{1}{2}\frac{d}{dt}\|\partial_{t}u\|^{2}_{H}+\nu\|\partial_{t}u\|^{2}_{H^{1}}
+\left(B(\partial_{t}u,u),\partial_{t}u\right)+\left(B(u,\partial_{t}u),\partial_{t}u\right)=0.
\end{equation*}
Since $\left(B(u,\partial_{t}u),\partial_{t}u\right)=0$ and thanks to the Young's inequality
\begin{align*}
\left|\left(B(\partial_{t}u,u),\partial_{t}u\right)\right|
\leq& \frac{\nu}{2\rho^{2}_{2}}\|B(\partial_{t}u,u)\|^{2}_{H}+\frac{2\rho^{2}_{2}}{\nu}\|\partial_{t}u\|^{2}_{H}\nonumber\\
\leq& \frac{\nu}{2\rho^{2}_{2}}\|u\|^{2}_{L^{\infty}}\|\partial_{t}u\|^{2}_{H^{1}}
+\frac{2\rho^{2}_{2}}{\nu}\|\partial_{t}u\|^{2}_{H}\nonumber\\
\leq& \frac{\nu}{2}\|\partial_{t}u\|^{2}_{H^{1}}+\frac{2\rho^{2}_{2}}{\nu}\|\partial_{t}u\|^{2}_{H},~~\forall~t\geq t_{0}.
\end{align*}
Then
\begin{equation*}
\frac{d}{dt}\|\partial_{t}u\|^{2}_{H}+\nu\|\partial_{t}u\|^{2}_{H^{1}}
\leq \frac{2\rho^{2}_{2}}{\nu}\|\partial_{t}u\|^{2}_{H},~~\forall t\geq t_{0},
\end{equation*}
which gives the desired result thanks to the uniform Gronwall Lemma \ref{lem2.7}.
\end{remark}
\subsubsection{Asymptotic regularity: $H^{9/2}$-estimate}\label{sec3.1.3}
\noindent

In order to construct an IM for \eqref{2.3} in the subsequent subsections, here we need to prove the existence of an absorbing set in $H^{9/2}$ which is extremely crucial for our technique of construction---we modify the nonlinearity outside the absorbing ball and demonstrate that, in particular,  the modified nonlinearity satisfies the global Lipschitz continuity with some Lipschitz constant (see Theorem \ref{thm3.10} below). Proverbially, the regularity of the solution $u(t)$ is restricted by the regularity of the external force $f$. We emphasize that despite the fact that we cannot expect more regularity than $H^{3}$ because we just assume that the external force $f$ belongs to $H^{1}$, in this paper we can employ the so-called asymptotic regularity method (see, e.g., \cite{S10}) to deal with the difficulty and obtain the desired regularity of solution.  Specifically, let us consider equation on
\begin{equation}\label{3.3}
w:=u-v,
\end{equation}
where $u$ solves problem \eqref{1.1} and $v$ solves the stationary problem \eqref{3.1}, that is, $v$ is the stationary point of \eqref{1.1} (there may be several stationary points of this problem, we just fix one of them). Combining the Eqs. \eqref{2.3} with \eqref{3.1}, we have
\begin{equation}\label{3.4}
\begin{cases}
\nu Aw+B(w,w)+B(v,w)+B(w,v)=-\partial_{t}u,\\
w(0)=w_{0}=u_{0}-v.
\end{cases}
\end{equation}
Since removing a stationary point does not affect its long time behavior, we will study the solution $w$ and it's attractor of the equation \eqref{3.4} instead of the equation \eqref{2.3}. The advantage of this approach is that the solution of \eqref{3.4} is more regular. Indeed, as we will see from the following theorem, the solution $w$ will be at least in $H^{9/2}$ even if $f \in H^{1}$ only. We point out that the difficulty for employing this method is that we need to deal with more complicated nonlinear terms when we construct an IM for \eqref{3.4} in Subsection \ref{sec3.2.2}.
\begin{theorem}\label{thm3.4}
Assume that $f\in H^{1}$. Let $w(t)$ be a solution of \eqref{3.4} with $u_{0}\in H$. Then there exist $t'_{0}=t'_{0}\left(\|u_{0}\|_{H}\right)\geq t_{0}$ and a constant $k$ (depends only on $\nu$) such that
\begin{equation*}
\|\partial_{t}w(t)\|_{H^{s}}=\|\partial_{t}u(t)\|_{H^{s}}\leq \bar{\rho}_{s},~~\forall~t\geq t'_{0}+1,~~s=1,2~\text{or}~5/2,
\end{equation*}
\begin{equation}\label{3.5}
\|w(t)\|_{H^{9/2}}\leq\varrho_{2},~~\forall~t\geq t'_{0}+1,
\end{equation}
where $\bar{\rho}_{s}$ and $\varrho_{2}$ are given by the following formula:
\begin{equation*}
\bar{\rho}^{2}_{1}=k\exp\left\{\rho^{4}_{1}+\rho^{2}_{2}\right\}
\cdot\left\{\rho^{2}_{2}\bar{\rho}^{2}_{0}\right\},
\end{equation*}
\begin{equation*}
\bar{\rho}^{2}_{2}=k\exp\left\{\rho^{4}_{1}+\rho^{2}_{2}\right\}
\left\{\left(1+\rho^{4}_{1}+\rho^{2}_{2}\right)\bar{\rho}^{2}_{1}+\bar{\rho}^{2}_{1}\rho_{2}\rho_{3}\right\},
\end{equation*}
\begin{equation*}
\bar{\rho}^{2}_{5/2}=k\exp\left\{\rho^{4}_{1}+\rho^{2}_{2}\right\}
\left\{\left(1+\rho^{4}_{1}+\rho^{2}_{2}\right)\bar{\rho}^{2}_{2}
+\bar{\rho}^{2}_{1}\rho_{2}\rho_{3}+(\bar{\rho}_{1}+\bar{\rho}_{2})\bar{\rho}_{2}\rho^{2}_{3}\right\},
\end{equation*}
\begin{equation*}
\varrho_{2}=k\left[\left(\rho_{v}+\rho_{3}\right)^{2}+\bar{\rho}_{5/2}\right],
\end{equation*}
where $\rho_{3}=k\left(\bar{\rho}_{1}+\rho^{2}_{2}+\|f\|_{H^{1}}\right)$; $\rho_{v},\rho_{1},\rho_{2}$ were given in Lemma \ref{lem3.1} and Theorem \ref{thm3.2}, respectively.
\end{theorem}
\begin{proof}
We will continue to estimate the uniform boundedness in more regular space until we get $H^{9/2}$-estimate. However, multiplying (\ref{2.3}) by $A^{k}u$ ($k>3$) can not obtain $H^{k}$ estimate since we assume that $f\in H^{1}$ only. To achieve the desired result, as that in \cite{R01,T95,T97,CHOT05} we set $\psi=\partial_{t}u$ and differentiate Eqs. (\ref{2.3}) with respect to $t$ to leads to the following equation.
\begin{equation}\label{3.6}
\partial_{t}\psi+\nu A\psi+B(\psi,u)+B(u,\psi)=0.
\end{equation}
\par
\textbf{Firstly, we give the $H^{1}$-estimate of $\partial_{t}u$.} Taking the inner product of (\ref{3.6}) with $A\psi$, we obtain
\begin{equation}\label{3.7}
\frac{1}{2}\frac{d}{dt}\|\psi\|^{2}_{H^{1}}+\nu\|\psi\|^{2}_{H^{2}}+(B(\psi,u),A\psi)+(B(u,\psi),A\psi)=0.
\end{equation}
Thanks to \eqref{2.5}, we can estimate
\begin{align}\label{3.8}
\left|(B(\psi,u),A\psi)\right|
\leq  c\|\psi\|_{H^{1}}\|u\|^{1/2}_{H^{1}}\|Au\|^{1/2}_{H}\|A\psi\|_{H}
\leq C_{\nu}\|u\|_{H^{1}}\|u\|_{H^{2}}\|\psi\|^{2}_{H^{1}}+\frac{\nu}{4}\|\psi\|^{2}_{H^{2}}.
\end{align}
Analogously,
\begin{align}\label{3.9}
\left|(B(u,\psi),A\psi)\right|
\leq  c\|u\|_{H^{1}}\|\psi\|^{1/2}_{H^{1}}\|A\psi\|^{3/2}_{H}
\leq C_{\nu}\|u\|^{4}_{H^{1}}\|\psi\|^{2}_{H^{1}}+\frac{\nu}{4}\|A\psi\|^{2}_{H}.
\end{align}
Substituting (\ref{3.8}) and (\ref{3.9}) into (\ref{3.7}), we deduce
\begin{equation}\label{3.10}
\frac{d}{dt}\|\psi\|^{2}_{H^{1}}+\nu\|\psi\|^{2}_{H^{2}}
\leq C_{\nu}(\|u\|^{4}_{H^{1}}+\|u\|^{2}_{H^{2}})\|\psi\|^{2}_{H^{1}}.
\end{equation}
On the other hand, recalling the estimate \eqref{3.2} that
\begin{align*}
\int^{t+1}_{t}\|\psi(s)\|^{2}_{H^{1}}ds
\leq C_{\nu}\rho^{2}_{2}\bar{\rho}^{2}_{0},~~\forall~t\geq t_{0}+1.
\end{align*}

Substituting above estimates into (\ref{3.10}) and applying Gronwall-type Lemma \ref{lem2.7}, we know that there exists $t'_{0}=t'_{0}\left(\|u_{0}\|_{H}\right)\geq t_{0}$ such that
\begin{align}\label{3.11}
\|\psi\|^{2}_{H^{1}}
\leq& \exp\left\{C_{\nu}(\rho^{4}_{1}+\rho^{2}_{2})\right\}
\cdot\left\{C_{\nu}\rho^{2}_{2}\bar{\rho}^{2}_{0}\right\}:=\bar{\rho}^{2}_{1},~~\forall~t\geq t'_{0}+1.
\end{align}
\par
As a result, applying the assumption $f\in H^{1}$, the Eqs. \eqref{2.3} and the estimate \eqref{3.11}, we obtain
\begin{align}\label{3.12}
\|u\|_{H^{3}}=\|Au\|_{H^{1}}
\leq& \frac{1}{\nu}\left(\|\partial_{t}u\|_{H^{1}}+\|B(u,u)\|_{H^{1}}+\|f\|_{H^{1}}\right)\nonumber\\
\leq& C_{\nu}\left(\bar{\rho}_{1}+\rho^{2}_{2}+\|f\|_{H^{1}}\right)
:=\rho_{3},~~\forall~t\geq t'_{0}+1.
\end{align}
\par
\textbf{Now we ready to estimate $\|\psi\|_{H^{2}}$.} To do this, multiplying (\ref{3.6}) by $A^{2}\psi$ and integrating over $x\in\mathbb{T}^{2}$, we get
\begin{equation}\label{3.13}
\frac{1}{2}\frac{d}{dt}\|\psi\|^{2}_{H^{2}}+\nu\|\psi\|^{2}_{H^{3}}+(B(\psi,u),A^{2}\psi)
+(B(u,\psi),A^{2}\psi)=0.
\end{equation}
We can estimate
\begin{align}\label{3.14}
\left|(B(\psi,u),A^{2}\psi)\right|
\leq &|(B(A^{1/2}\psi,u),A^{3/2}\psi)|+|(B(\psi,A^{1/2}u),A^{3/2}\psi)|\nonumber\\
\leq &c\left(\|A^{1/2}\psi\|_{H^{1}}\|u\|^{1/2}_{H^{1}}\|Au\|^{1/2}_{H}
+\|\psi\|_{H^{1}}\|A^{1/2}u\|^{1/2}_{H^{1}}\|A^{3/2}u\|^{1/2}_{H}\right)\|A^{3/2}\psi\|_{H}\nonumber\\
\leq &C_{\nu}\left(\|\psi\|^{2}_{H^{2}}\|u\|_{H^{1}}\|u\|_{H^{2}}+\|\psi\|^{2}_{H^{1}}\|u\|_{H^{2}}\|u\|_{H^{3}}\right)
+\frac{\nu}{4}\|\psi\|^{2}_{H^{3}}
\end{align}
and, similarly,
\begin{align}\label{3.15}
\left|(B(u,\psi),A^{2}\psi)\right|
\leq &|(B(A^{1/2}u,\psi),A^{3/2}\psi)|+|(B(u,A^{1/2}\psi),A^{3/2}\psi)|\nonumber\\
\leq &c\left(\|A^{1/2}u\|_{H^{1}}\|\psi\|^{1/2}_{H^{1}}\|A\psi\|^{1/2}_{H}
+\|u\|_{H^{1}}\|A^{1/2}\psi\|^{1/2}_{H^{1}}\|A^{3/2}\psi\|^{1/2}_{H}\right)\|A^{3/2}\psi\|_{H}\nonumber\\
\leq &C_{\nu}\left(\|u\|^{2}_{H^{2}}+\|u\|^{4}_{H^{1}}\right)\|\psi\|^{2}_{H^{2}}
+\frac{\nu}{4}\|\psi\|^{2}_{H^{3}},
\end{align}
here we have used \eqref{2.5} and H\"{o}lder inequality with exponent $(\frac{1}{4},\frac{3}{4})$. Substituting (\ref{3.14}) and (\ref{3.15}) into (\ref{3.13}) leads to the following differential inequality
\begin{align}\label{3.16}
\frac{d}{dt}\|\psi\|^{2}_{H^{2}}+\nu\|\psi\|^{2}_{H^{3}}
\leq &C_{\nu}\left(\|u\|_{H^{1}}\|u\|_{H^{2}}+\|u\|^{2}_{H^{2}}+\|u\|^{4}_{H^{1}}\right)\|\psi\|^{2}_{H^{2}}
+\|\psi\|^{2}_{H^{1}}\|u\|_{H^{2}}\|u\|_{H^{3}}\nonumber\\
\leq &C_{\nu}(\rho^{4}_{1}+\rho^{2}_{2})\|\psi\|^{2}_{H^{2}}+\bar{\rho}^{2}_{1}\rho_{2}\rho_{3},~~\forall~t\geq t'_{0}+1,
\end{align}
here we have used the result \eqref{3.11} and (\ref{3.12}). Besides, by (\ref{3.10}) and (\ref{3.11}), we know that
\begin{align}\label{3.17}
\int^{t+1}_{t}\|\psi(s)\|^{2}_{H^{2}}ds
\leq C_{\nu}\left(1+\rho^{4}_{1}+\rho^{2}_{2}\right)\bar{\rho}^{2}_{1},~~\forall~t\geq t'_{0}+1.
\end{align}
Combining (\ref{3.16}) with \eqref{3.17} and applying Gronwall-type Lemma \ref{lem2.7}, we deduce
\begin{align*}
\|\psi\|^{2}_{H^{2}}
\leq C_{\nu}\exp\left\{C_{\nu}(\rho^{4}_{1}+\rho^{2}_{2})\right\}
\left[\left(1+\rho^{4}_{1}+\rho^{2}_{2}\right)\bar{\rho}^{2}_{1}+\bar{\rho}^{2}_{1}\rho_{2}\rho_{3}\right]
:=\bar{\rho}^{2}_{2},~~\forall~t\geq t'_{0}+1.
\end{align*}
\par
\textbf{Furthermore, we give the estimate $\|\psi\|_{H^{5/2}}$.} Thus, we multiply (\ref{3.6}) by $A^{5/2}\psi$ and integrating over $x\in\mathbb{T}^{2}$ to obtain
\begin{equation*}
\frac{1}{2}\frac{d}{dt}\|\psi\|^{2}_{H^{5/2}}+\nu\|\psi\|^{2}_{H^{7/2}}+(B(\psi,u),A^{5/2}\psi)
+(B(u,\psi),A^{5/2}\psi)=0.
\end{equation*}
We deal with the nonlinear terms as follows:
\begin{align*}
\left|(B(\psi,u),A^{5/2}\psi)+(B(u,\psi),A^{5/2}\psi)\right|
\leq &|(B(A^{3/4}\psi,u),A^{7/4}\psi)|+|(B(\psi,A^{3/4}u),A^{7/4}\psi)|\nonumber\\
&+|(B(A^{3/4}u,\psi),A^{7/4}\psi)|+|(B(u,A^{3/4}\psi),A^{7/4}\psi)|.
\end{align*}
In addition, \eqref{2.5} and H\"{o}lder inequality with exponent $(\frac{1}{4},\frac{3}{4})$ give us that
\begin{align*}
&|(B(A^{3/4}\psi,u),A^{7/4}\psi)|
+|(B(A^{3/4}u,\psi),A^{7/4}\psi)|+|(B(u,A^{3/4}\psi),A^{7/4}\psi)|\nonumber\\
\leq &c\left(\|A^{3/4}\psi\|_{H^{1}}\|u\|^{1/2}_{H^{1}}\|Au\|^{1/2}_{H}
+\|A^{3/4}u\|_{H^{1}}\|\psi\|^{1/2}_{H^{1}}\|A\psi\|^{1/2}_{H}\right)\|A^{7/4}\psi\|_{H}\nonumber\\
&+c\left(
\|u\|_{H^{1}}\|A^{3/4}\psi\|^{1/2}_{H^{1}}\|A^{7/4}\psi\|^{1/2}_{H}\right)\|A^{7/4}\psi\|_{H}\nonumber\\
\leq &C_{\nu}\left[\left(\|u\|_{H^{1}}\|u\|_{H^{2}}+\|u\|^{4}_{H^{1}}\right)\|\psi\|^{2}_{H^{5/2}}
+\|\psi\|_{H^{1}}\|\psi\|_{H^{2}}\|u\|^{2}_{H^{5/2}}\right]+\frac{\nu}{4}\|\psi\|^{2}_{H^{7/2}}.
\end{align*}
Moreover, we can estimate
\begin{align*}
|(B(\psi,A^{3/4}u),A^{7/4}\psi)|
\leq& c\|B(\psi,A^{3/4}u)\|_{H}\|A^{7/4}\psi\|_{H}
\leq C\|\psi\|_{L^{\infty}}\|A^{3/4}(\nabla\cdot u)\|_{H}\|A^{7/4}\psi\|_{H}\nonumber\\
\leq& c\|\psi\|_{H^{2}}\|A^{3/4}u\|_{H^{1}}\|A^{7/4}\psi\|_{H}
\leq C_{\nu}\|\psi\|^{2}_{H^{2}}\|u\|^{2}_{H^{5/2}}+\frac{\nu}{4}\|\psi\|^{2}_{H^{7/2}}.
\end{align*}
Then
\begin{align*}
&\frac{d}{dt}\|\psi\|^{2}_{H^{5/2}}+\nu\|\psi\|^{2}_{H^{7/2}}\nonumber\\
\leq &C_{\nu}\left[\left(\|u\|_{H^{1}}\|u\|_{H^{2}}+\|u\|^{4}_{H^{1}}\right)\|\psi\|^{2}_{H^{5/2}}
+(\|\psi\|_{H^{1}}+\|\psi\|_{H^{2}})\|\psi\|_{H^{2}}\|u\|^{2}_{H^{5/2}}\right]\nonumber\\
\leq &C_{\nu}\left[(\rho^{4}_{1}+\rho^{2}_{2})\|\psi\|^{2}_{H^{5/2}}
+(\bar{\rho}_{1}+\bar{\rho}_{2})\bar{\rho}_{2}\rho^{2}_{5/2}\right].
\end{align*}

By \eqref{3.16} and the uniform estimate on $\|\psi\|_{H^{2}}$, we see that
\begin{align*}
\int^{t+1}_{t}\|\psi(s)\|^{2}_{H^{5/2}}ds\leq\int^{t+1}_{t}\|\psi(s)\|^{2}_{H^{3}}ds
\leq C_{\nu}\left(1+\rho^{4}_{1}+\rho^{2}_{2}\right)\bar{\rho}^{2}_{2}
+\bar{\rho}^{2}_{1}\rho_{2}\rho_{3},~~\forall~t\geq t'_{0}+1.
\end{align*}
Consequently, the Gronwall lemma gives us that
\begin{align*}
\|\psi\|^{2}_{H^{5/2}}
\leq C_{\nu}\exp\left\{C_{\nu}(\rho^{4}_{1}+\rho^{2}_{2})\right\}
\left[\left(1+\rho^{4}_{1}+\rho^{2}_{2}\right)\bar{\rho}^{2}_{2}
+\bar{\rho}^{2}_{1}\rho_{2}\rho_{3}+(\bar{\rho}_{1}+\bar{\rho}_{2})\bar{\rho}_{2}\rho^{2}_{3}\right]
:=\bar{\rho}^{2}_{5/2}
\end{align*}
for any $t\geq t'_{0}+1$.

\textbf{Finally, we proof the desired result, that is, the uniform boundedness of $w$ in the $H^{\frac{9}{2}}$ norm.} Indeed, by the Lemma \ref{lem3.1} and the estimate \eqref{3.12}, we see that
\begin{equation*}
\|w\|_{L^{\infty}(t,\infty;H^{3})}\leq \|v\|_{L^{\infty}(t,\infty;H^{3})}+\|v\|_{H^{3}}\leq \rho_{3}+\rho_{v},~~\forall t\geq t'_{0}+1.
\end{equation*}
Consequently, by the first equation in \eqref{3.4}, we can estimate
\begin{align*}
\|w\|_{H^{9/2}}=&\|Aw\|_{H^{5/2}}\leq \frac{1}{\nu}\left(\|B(w,w)\|_{H^{5/2}}+\|B(v,w)\|_{H^{5/2}}+\|B(w,v)\|_{H^{5/2}}
+\|\partial_{t}u\|_{H^{5/2}}\right)\nonumber\\
\leq& C_{\nu}\left(\|w\|_{H^{5/2}}\|\nabla \cdot w\|_{L^{\infty}}+2\|w\|_{H^{5/2}}\|\nabla \cdot v\|_{L^{\infty}}+\|\partial_{t}u\|_{H^{5/2}}\right)\nonumber\\
\leq& C_{\nu}\left[\|w\|_{H^{5/2}}\left(2\|v\|_{H^{3}}+\|w\|_{H^{3}}\right)+\|\partial_{t}u\|_{H^{5/2}}\right]\nonumber\\
\leq& C_{\nu}\left[(\rho_{3}+\rho_{v})\left(3\rho_{v}+\rho_{3}\right)+\bar{\rho}_{5/2}\right]
:=\varrho_{2},
\end{align*}
for any $t\geq t'_{0}+1$. Thus we complete the proof.
\end{proof}
\subsubsection{Well-posedness and global attractors}\label{sec3.1.4}
\noindent

Our basic goal in this subsection is to prove that problems \eqref{2.3} and \eqref{3.4} possess a global attractor in $H$ by using some a priori estimates obtained in Subsection \ref{sec3.1.3}. Firstly, we recall the following result of well-posedness of \eqref{2.3} which can be found in \cite{T95}.
\begin{theorem}\label{thm3.5}
Let $f\in H$ and $u_{0}\in H$. Then for any $T>0$ there exists a unique solution $u$ to Eqs. \eqref{2.3} that satisfies $u\in C([0,T];H)\cap L^{2}([0,T];H^{1})$ and $\partial_{t}u\in L^{2}([0,T];H^{-1})$ such that
\begin{equation*}
(\partial_{t}u+\nu Au+B(u,u)-f,\varphi)_{H^{-1},H^{1}}=0
\end{equation*}
for every $\varphi \in H^{1}$ and almost every $t\in (0,T)$. Moreover, this solution depends continuously on the initial data $u_{0}$ in the $H$ norm.
\end{theorem}
\begin{corollary}\label{cor3.6}
Let $f\in H$ and $u_{0}\in H$. Then the solution of \eqref{3.4} admits a continuous semigroup $S(t)$, and thus $(S(t),H)$ is a continuous dynamical system generated by the solution of \eqref{3.4}.
\end{corollary}
\begin{proof}
For the initial value problem \eqref{3.4}, we may easily note, as a corollary of Theorem \ref{thm3.5}, that it is uniquely solvable for any $f\in H$ and $w_{0}:=u_{0}-v\in H$. Indeed, Since $u(t)$ solves Eqs. \eqref{2.3} uniquely by Theorem \ref{thm3.5} and $v$ solves Eqs. \eqref{3.1} (which is fixed), thus $w(t):=u(t)-v$ solves Eqs. \eqref{3.4} uniquely. Thus, the solution operator $S(t) : w_{0} \rightarrow w(t)$ yields a nonlinear semigroup
of operator which enjoys the properties
\begin{equation*}
S(t+\tau)=S(t) \cdot S(\tau), \forall \tau, t \geq 0,
\end{equation*}
\begin{equation*}
S(0)=I
\end{equation*}
with
$S(t)$ a continuous nonlinear operator from $H$ into itself, for any $\nu>0, \forall t \geq 0$, where $w(t)$ is the solution of \eqref{3.4} and $w_{0}=u_{0}-v$ is the corresponding initial datum.
\end{proof}
Next, we consider the global attractor for the solution semigroup $S(t)$ associated with equation \eqref{3.4}. In fact, from Corollary \ref{cor3.6} and some a priori estimates in Subsection \ref{sec3.1.3}, we can obtain a $H$-attractor for the equation $\eqref{3.4}$.
\begin{theorem}\label{thm3.7}
Assume that $f\in H^{1}$ and $u_{0}\in H$. Let $(S(t),H)$ be a dynamical system generated by the solution of \eqref{3.4}. Then $(S(t),H)$ possesses a global attractor $\mathscr{A}$ in the phase space $H$. Moreover, it is bounded in the high regular space $H^{9/2}$.
\end{theorem}
\begin{proof}
It follows from the dissipative estimate \eqref{3.5} that the ball
\begin{equation}\label{3.18}
\mathscr{B}_{2}:=\left\{w\in H^{9/2}:~\|w\|_{H^{9/2}}\leq \varrho_{2}\right\}
\end{equation}
in $H^{9/2}$ is an bounded absorbing set in $H$. This infers that $(S(t),H)$ possesses a bounded absorbing set $\mathscr{B}$ and $S(t)$ is asymptotic compact in $H$ by the fact that the embedding $H^{9/2}\hookrightarrow H$ is compact. By the general theory of attractor (see, e.g.,\cite{T97}), we know that $(S(t),H)$ possesses a global attractor $\mathscr{A}$ in the phase space $H$, which is bounded in $H^{9/2}$.
\end{proof}
An immediate consequence of Theorem \ref{thm3.7} is the following result which will be pivotal not only for the analysis of global attractor but for our discussion of inertial manifolds in Section \ref{sec3.2} as well:
\begin{corollary}\label{cor3.8}
Let $f\in H^{1}$ and $u_{0}\in H$. Then, the problem \eqref{2.1} has a global attractor $\mathscr{A}_{0}$. Moreover, we can decompose $\mathscr{A}_{0}=\mathscr{A}+v(x)$, where $v(x)$ is a stationary solution (equilibrium point) of \eqref{2.1} and $\mathscr{A}$ is bounded in $H^{9/2}$.
\end{corollary}
\subsection{Existence of the inertial manifold}\label{sec3.2}
\noindent

In this section,  we construct an $N$-dimensional IM for the initial-value problem \eqref{3.4} in $H$, namely, prove Theorem \ref{thm1.2}.
\subsubsection{An abstract scheme for IMs}\label{sec3.2.1}
\noindent

We recall, in this subsection, the sequence of supporting results that will lead to the existence of IM for the abstract models, including the cone invariance, squeezing property and strong cone condition. One can refer to \cite{M-PS88,Z14} and the references therein. Let $\mathbb{H}$ be a separable Hilbert space and consider the following abstract model in $\mathbb{H}$,
\begin{equation}\label{3.19}
\partial_{t}w+\mathcal{A}w+\mathcal{F}(w)=0,~~w|_{t=0}=w_{0},
\end{equation}
on a bounded domain $\Omega\subset\mathbb{R}^{d}$ ($d\leq 3$) with periodic boundary condition or Dirichlet boundary condition, where $\mathcal{F}$ is an operator mapping from $\mathbb{H}$ to $\mathbb{H}$, and $\mathcal{A}: D(\mathcal{A})\rightarrow \mathbb{H}$ is a positive definite operator with discrete spectrum in $\mathbb{H}$:
\begin{equation*}
0<\lambda_{1}\leq\lambda_{2}\leq\cdot\cdot\cdot~~\mbox{and}~~\lambda_{j}\rightarrow\infty~~\mbox{as}~~j\rightarrow\infty
\end{equation*}
and the corresponding eigenvector $\{e_{j}\}_{j=1}^{\infty}$ forms orthonormal basis of $\mathbb{H}$ and such that
\begin{align*}
\mathcal{A}e_{j}=\lambda_{j}e_{j}.
\end{align*}
Every $w\in \mathbb{H}$ can thus be presented in the form
\begin{equation*}
w=\sum^{\infty}_{j=1}w_{j}e_{j},~~w_{j}=(w,e_{j}),
\end{equation*}
and, due to the Parseval equality,
\begin{equation*}
\|w\|^{2}_{\mathbb{H}}=\sum^{\infty}_{j=1}w_{j}^{2}.
\end{equation*}
\begin{theorem}[\cite{Z14}]\label{thm3.9}
Consider the abstract equation \eqref{3.19}. Assume that $\mathcal{F}: \mathbb{H}\rightarrow \mathbb{H}$ is globally Lipschitz continuous with Lipschitz constant $L>0$ and $\|\mathcal{F}(w)\|_{\mathbb{H}}\leq C$ for some $C>0$. Suppose that the spectral gap condition
\begin{equation*}
\lambda_{N+1}-\lambda_{N}>2L
\end{equation*}
holds for some $N\in\mathbb{N}$. Then
\begin{enumerate}[(i)]
  \item the strong cone condition \eqref{2.9} holds with some $\gamma$, $\mu>0$, thus, the cone invariance \eqref{2.10} and squeezing property \eqref{2.11} hold;
  \item moreover, the problem \eqref{3.19} possesses an $N$--dimensional inertial manifold $\mathcal{M}$ in sense of Definition \ref{def2.4}.
\end{enumerate}
\end{theorem}
\subsubsection{The proof of Theorem \ref{thm1.2}}\label{sec3.2.2}
\noindent

In this part, we set $\mathbb{H}:=H,~\mathcal{A}:=A=-\nu P_{\sigma}\Delta$, $\Omega=\mathbb{T}^{2}$ and $\mathcal{F}(w)=B(w,w)+B(v,w)+B(w,v)$. We will verify the existence of IMs of \eqref{3.4} by using the abstract results stated in Subsection \ref{sec3.2.1}. To this end, firstly, we deal with the nonlinear term to satisfy the conditions in Theorem \ref{thm3.9}.
\par
From the analysis in Section \ref{sec3.1.4}, it follows that there exist absorbing balls in $H^{9/2}$
which attract all the orbits of \eqref{3.4}. We will, in the subsequent discussion, restrict our attention to the dynamics inside an absorbing ball $\mathscr{B}\subseteq H^{9/2}$ established in \eqref{3.18}. Now, we define a vector-valued cut-off function as
\begin{equation*}
\vec{\eta}=(\eta(\zeta_{1}),\eta(\zeta_{2}))\in\mathbb{C}^{2},~~\mbox{for}~\zeta=(\zeta_{1},\zeta_{2}),
\end{equation*}
where $\eta\in C^{\infty}_{0}(\mathbb{C})$ is a smooth cut-off function satisfying
\begin{equation*}
\eta(\zeta)=\zeta~~\mbox{for}~~|\zeta|\leq 1~~\mbox{and}~~|\eta(\zeta)|\leq 2~~\mbox{for all}~~\zeta\in\mathbb{C}.
\end{equation*}
Next, similar to Kostianko \cite{K18}, we shall introduce the truncation operator $W: H \rightarrow H$ defined as
\begin{equation*}
W(w)=\sum_{j\in\mathbb{Z}^{2}_{\ast}}
\frac{\varrho_{2}}{|j|^{9/2}}P^{j}_{\sigma}\vec{\eta}\left(\frac{|j|^{9/2}\widehat{w}_{j}}{\varrho_{2}}\right)e^{ij\cdot x},
\end{equation*}
where $\varrho_{2}$ is the radius of the absorbing ball $\mathscr{B}_{2}$ that given in \eqref{3.18}, $P^{j}_{\sigma}$ are the Leray projector matrices defined by \eqref{2.1} and
\begin{equation*}
w=\sum_{j\in\mathbb{Z}^{2}_{\ast}}\widehat{w}_{j}e^{ij\cdot x},
\end{equation*}
here we recall that $w=(w^{1},w^{2})$ is a vector variable and each coefficient $\widehat{w}_{j}=(\widehat{w}^{1}_{j},\widehat{w}^{2}_{j})\in \mathbb{C}^{2}$ is also a vector.
\par
Now, we define the modified nonlinearity by the following formula
\begin{equation}\label{3.20}
F_{2}(W(w)):=\mathcal{F}(W(w))=B(W(w),W(w))+B(W(w),v)+B(v,W(w)).
\end{equation}
The so-called ``prepared'' equation of \eqref{3.4} is written as follows:
\begin{equation}\label{3.21}
\partial_{t}w+Aw+F_{2}(W(w))=0.
\end{equation}
\par
Then the ``prepared'' equation \eqref{3.21} has the same long-time behavior as the original problem \eqref{1.1} or \eqref{3.4}. In fact,
\begin{proof}[\textbf{Proof of Theorem \ref{thm1.1} in the case $d=2$}]
It is sufficient to show that the truncation procedure does not change the equation on the attractor $\mathscr{A}$ or the absorbing set $\mathscr{B}_{2}$ at least, that is, $W(w)=w$ whenever $w$ lies in the absorbing set $\mathscr{B}_{2}$. Indeed, let $w\in \mathscr{B}_{2}$, i.e. $w\in H^{9/2}$ such that
\begin{equation*}
\|w\|^{2}_{H^{9/2}}=\sum_{j\in\mathbb{Z}^{2}_{\ast}}|j|^{9}|\widehat{w}_{j}|^{2}\leq \varrho_{2}^{2},
\end{equation*}
it follows that
\begin{equation*}
\frac{|j|^{9/2}\widehat{w}_{j}}{\varrho_{2}}\leq 1,~~\forall~j\in\mathbb{Z}^{2},j\neq 0.
\end{equation*}
Thus, by the definition of $\vec{\eta}$, we know that $\vec{\eta}\left(\frac{|j|^{9/2}\widehat{w}_{j}}{\varrho_{2}}\right)=\frac{|j|^{9/2}\widehat{w}_{j}}{\varrho_{2}}$ for all $j\in\mathbb{Z}^{2},j\neq 0$. This shows that $W(w)=w$ whenever $w\in \mathscr{B}_{2}$ according to the definition of $W(w)$.
\end{proof}
\begin{theorem}\label{thm3.10}
The modified nonlinearity $F_{2}(W(\cdot)): H\rightarrow H$, given in \eqref{3.20}, is globally bounded and globally Lipschitz continuous with some Lipschitz constant $L>0$.
\end{theorem}
\begin{proof}
\textbf{Step 1.} To verify $F_{2}(W(\cdot))$ is globally bounded from $H$ to $H$, firstly, we give the \textbf{claim 1} that $W$ is a regularization operator: for any $\varepsilon>0$, $W$ maps $H$ into $H^{7/2-\varepsilon}$ continuously and there exists a constant $M_{\varepsilon}$ such that
\begin{equation}\label{3.22}
\|W(w)\|_{H^{7/2-\varepsilon}}\leq M_{\varepsilon},~~\mbox{for all}~~w\in H.
\end{equation}
Then we can estimate
\begin{align*}
\|F_{2}(W(w))\|_{H}\leq& \|B(W(w),W(w))\|_{H}+\|B(W(w),v)\|_{H}+\|B(v,W(w))\|_{H}\nonumber\\
\leq& C\left(\|W(w)\|_{H^{1}}\|W(w)\|_{H^{2}}+\|W(w)\|_{H^{1}}\|v\|_{H^{2}}+\|v\|_{H^{2}}\|W(w)\|_{H^{1}}\right)\nonumber\\
\leq& C\|W(w)\|_{H^{7/2-\varepsilon}}\left(\|W(w)\|_{H^{7/2-\varepsilon}}+2\|v\|_{H^{2}}\right)\nonumber\\
\leq& CM_{\varepsilon}\left(M_{\varepsilon}+2\|v\|_{H^{2}}\right),~~(0<\varepsilon<\frac{1}{2}),~~\forall~w\in~H,
\end{align*}
where we have used the regularity of the cut-off operator $W$ as in \eqref{3.22} and the following preliminary results:
\begin{equation*}
\|B(X,Y)\|_{H}\leq C\|\nabla X\|_{H}\|Y\|_{L^{\infty}}\leq C'\|X\|_{H^{1}}\|Y\|_{H^{2}},
\end{equation*}
for any $X\in H^{1}$ and $Y\in H^{2}$, where we have used Sobolev embedding $H^{2}\hookrightarrow L^{\infty}$ in $\mathbb{T}^{2}$.
\par
It remains to prove the \textbf{claim 1} holds. Indeed, let $w\in H$. Since $|\eta(\xi)|\leq 2$ for all $\xi\in \mathbb{C}$, then by the definition of $W(w)$, we estimate
\begin{align*}
\|W(w)\|^{2}_{H^{7/2-\varepsilon}}
=&\sum_{j\in\mathbb{Z}^{2}_{\ast}}|j|^{7-2\varepsilon}\cdot
\frac{\varrho^{2}_{2}}{|j|^{9}}\left|P^{j}_{\sigma}\vec{\eta}\left(\frac{|j|^{9/2}\widehat{w}_{j}}{\varrho_{2}}\right)\right|^{2}\nonumber\\
\leq& C\sum_{j\in\mathbb{Z}^{2}_{\ast}}|j|^{7-2\varepsilon}\cdot\frac{\varrho^{2}_{2}}{|j|^{9}}
\leq C\varrho^{2}_{2}\sum_{j\in\mathbb{Z}^{2}_{\ast}}\frac{1}{|j|^{2+2\varepsilon}}\leq M_{\varepsilon}.
\end{align*}
Next, we will show that $W: H\rightarrow H^{7/2-\varepsilon}$ is continuous. Indeed, given $\delta>0$, $w$, $\varpi\in H$ satisfies $\|w-\varpi\|_{H}\leq\sigma$. Then
\begin{align*}
\|W(w)-W(\varpi)\|^{2}_{H^{7/2-\varepsilon}}
=&\sum_{j\in\mathbb{Z}^{2}_{\ast}}
\frac{\varrho^{2}_{2}}{|j|^{2+2\varepsilon}}\left|\vec{\eta}\left(\frac{|j|^{9/2}\widehat{w}_{j}}{\varrho_{2}}\right)
-\vec{\eta}\left(\frac{|j|^{9/2}\widehat{\varpi}_{j}}{\varrho_{2}}\right)\right|^{2}\nonumber\\
\leq& \sum_{\stackrel{j\in\mathbb{Z}^{2}}{1\leq|j|\leq N}}\frac{\varrho^{2}_{2}}{|j|^{2+2\varepsilon}}\left|\vec{\eta}\left(\frac{|j|^{9/2}\widehat{w}_{j}}{\varrho_{2}}\right)
-\vec{\eta}\left(\frac{|j|^{9/2}\widehat{\varpi}_{j}}{\varrho_{2}}\right)\right|^{2}
+C\sum_{\stackrel{j\in\mathbb{Z}^{2}}{|j|\geq N}}\frac{\varrho^{2}_{2}}{|j|^{2+2\varepsilon}},
\end{align*}
where we have used the fact that $|\vec{\eta}|^{2}\leq \frac{C}{2}$ due to $|\eta(\xi)|<2$ for all $\xi\in\mathbb{C}$. Since the series $\sum_{j\in\mathbb{Z}^{2}_{\ast}}\frac{1}{|j|^{2+2\varepsilon}}$ is convergent, thus
\begin{align*}
 C\sum_{\stackrel{j\in\mathbb{Z}^{2}}{|j|\geq N}}\frac{\varrho^{2}_{2}}{|j|^{2+2\varepsilon}}\leq \frac{\delta}{2},
\end{align*}
when we choose $N$ large enough. Moreover, since $\|w-\varpi\|_{H}\leq\sigma$, we have $|\widehat{w}_{j}-\widehat{\varpi}_{j}|\leq\sigma$ for every $j\in\mathbb{Z}^{2},j\neq 0$, then we can select $\sigma$ sufficiently small such that
\begin{align*}
\sum_{\stackrel{j\in\mathbb{Z}^{2}}{1\leq|j|\leq N}}
\frac{\varrho^{2}_{2}}{|j|^{2+2\varepsilon}}\left|\vec{\eta}\left(\frac{|j|^{9/2}\widehat{w}_{j}}{\varrho_{2}}\right)
-\vec{\eta}\left(\frac{|j|^{9/2}\widehat{\varpi}_{j}}{\varrho_{2}}\right)\right|^{2}
\leq& C\sum_{\stackrel{j\in\mathbb{Z}^{2}}{1\leq|j|\leq N}}|j|^{9}|\widehat{w}_{j}-\widehat{\varpi}_{j}|^{2}
\leq \frac{\delta}{2},
\end{align*}
according to the fact that $\vec{\eta}$ is uniformly continuous and $N$ is fixed whenever it is chose above.
\par
\textbf{Step 2.} In order to prove $F_{2}(W(\cdot))$ is globally Lipschitz from $H$ to $H$, we give the \textbf{claim 2} that the operator $W$ is globally Lipschitz from $H$ to $H$, i.e. there exists constant $L_{1}>0$ such that
\begin{equation}\label{3.23}
\|W(w)-W(\varpi)\|_{H}\leq L_{1}\|w-\varpi\|_{H},~~\forall~w,\varpi\in H.
\end{equation}
Indeed, let $w$, $\varpi\in H$. Since $\vec{\eta}'$ is globally bounded, we calculate
\begin{align*}
\|W(w)-W(\varpi)\|^{2}_{H}
\leq& \sum_{j\in\mathbb{Z}^{2}_{\ast}}
\frac{\varrho^{2}_{2}}{|j|^{9}}\left|\vec{\eta}\left(\frac{|j|^{9/2}\widehat{w}_{j}}{\varrho_{2}}\right)
-\vec{\eta}\left(\frac{|j|^{9/2}\widehat{\varpi}_{j}}{\varrho_{2}}\right)\right|^{2}\nonumber\\
\leq& \sum_{j\in\mathbb{Z}^{2}_{\ast}}|\vec{\eta}'(\xi)|^{2}|\widehat{w}_{j}-\widehat{\varpi}_{j}|^{2}
\leq L^{2}_{1}\|w-\varpi\|^{2}_{H}.
\end{align*}
\par
As a result, let $w_{i}\in H$, $i=1,2$ and set $W_{i}=W(w_{i})$, we write
\begin{align*}
F_{2}(W_{1})-F_{2}(W_{2})=& B(W_{1},W_{1})+B(W_{1},v)+B(v,W_{1})-B(W_{2},W_{2})+B(W_{2},v)+B(v,W_{2})\nonumber\\
=& B(W_{1}-W_{2},W_{1})+B(W_{2},W_{1}-W_{2})+B(v,W_{1}-W_{2})+B(W_{1}-W_{2},v).
\end{align*}
Note that the fact that $W: H\rightarrow H$ is globally Lipschitz with Lipschitz constant $L_{1}$ (see \eqref{3.23}), we can estimate
\begin{align*}
\|B(W_{1}&-W_{2},W_{1})\|_{H}+\|B(W_{1}-W_{2},v)\|_{H}
\leq C\|W_{1}-W_{2}\|\left(\|W_{1}\|_{H^{3}}+\|v\|_{H^{3}}\right)\nonumber\\
\leq& C L_{1}\|w_{1}-w_{2}\|_{H}(\|W_{1}\|_{H^{3}}+\|v\|_{H^{3}})
\leq \frac{L}{2}\|w_{1}-w_{2}\|_{H},
\end{align*}
where the constant $L$ depends on $L_{1}$ and $\|v\|_{H^{3}}$, here we have used \eqref{3.23} and the fact that
\begin{equation}\label{3.24}
\|B(X,Y)\|_{H}\leq C\|X\|_{H}\|\nabla Y\|_{L^{\infty}}\leq \tilde{C}\|X\|_{H}\|Y\|_{H^{3}},
\end{equation}
for any $X\in H$ and $Y\in H^{3}$, where we have used Sobolev embedding $H^{2}\hookrightarrow L^{\infty}$ in $\mathbb{T}^{2}$. Analogously,
\begin{align*}
\|B(W_{2},W_{1}-W_{2})\|_{H}+\|B(v,W_{1}-W_{2})\|_{H}
\leq C(\|W_{2}\|_{H^{3}}+\|v\|_{H^{3}})\|W_{1}-W_{2}\|_{H}
\leq \frac{L}{2}\|w_{1}-w_{2}\|_{H},
\end{align*}
where we have used \eqref{3.23} and \eqref{3.24} again. As a result, we conclude that
\begin{align*}
\|F_{2}(W(w_{1}))-F_{2}(W(w_{2}))\|_{H}\leq L\|w_{1}-w_{2}\|_{H},~~\forall w_{1},w_{2}\in H.
\end{align*}
\end{proof}
\par
Notice that for the Stokes operator $A$,  an arbitrarily large spectral gap can be obtained due to that its eigenvalues in the periodic domain $\mathbb{T}^{2}$ are of the form $(2\pi)^{-2}(k^{2}_{1}+k^{2}_{2})$, i.e.,there are some $N\in\mathbb{N}$ such that the eigenvalues of the Stokes operator $A$ satisfy the spectral gap condition:
\begin{equation}\label{3.25}
\limsup_{N\rightarrow\infty}(\lambda_{N+1}-\lambda_{N})=\infty,
\end{equation}
which benefits from a classical number theory result introduced by Richards \cite{R82}, given below.
\begin{lemma}[Richards \cite{R82}]\label{lem3.11}
The sequence
\begin{equation*}
\left\{\omega_{n}=k^{2}_{1}+k^{2}_{2}: k_{1}, k_{2}\in\mathbb{Z}~\text{and}~\omega_{n+1}\geq\omega_{n}\right\}
\end{equation*}
has a subsequence $\{\omega_{n_{j}}\}$ such that $\omega_{n_{j}+1}-\omega_{n_{j}}\geq c\log(\omega_{n_{j}})$ for some $c>0$.
\end{lemma}
Now, we are ready to prove Theorem \ref{thm1.2} in the following.
\begin{proof}[\textbf{Proof of Theorem \ref{thm1.2}}]
According to Theorem \ref{thm3.10}, we know that $F_{2}(W(\cdot)):H\rightarrow H$ is globally bounded and globally Lipschitz continuous with the Lipschitz constant $L$. In addition, we have clarified that the eigenvalues of the Stokes operator $A$ satisfies the spectral gap condition \eqref{3.25}, which gives that there exists infinitely many $N$ such that
\begin{equation*}
\lambda_{N+1}-\lambda_{N}>2L.
\end{equation*}
As a consequence, applying Theorem \ref{thm3.4}, we can obtain an $N$-dimensional IM for the ``prepared'' equation \eqref{3.21}. This concludes the proof of Theorem \ref{thm1.2}.
\end{proof}

\section{IMs for the abstract model (\ref{1.12})}\label{sec4}
\noindent

Let $\mathbb{H}$ be a separable Hilbert space. We consider the following abstract model
\begin{equation}\label{4.1}
\partial_{t}u+\mathcal{A}^{1+\alpha}u+\mathcal{A}^{\alpha}\mathcal{F}(u)=g,~~u|_{t=0}=u_{0}\in\mathbb{H},
\end{equation}
with $0<\alpha<1$. Here $\mathcal{F}$ is an operator mapping from $\mathbb{H}$ to $\mathbb{H}$, and $\mathcal{A}: D(\mathcal{A})\rightarrow \mathbb{H}$ is a positive definite and self-adjoint operator with compact inverse in the Hilbert space $\mathbb{H}$ (with the norm $\|\cdot\|$). Then there exists discrete spectrum $\{\lambda_{k}\}_{k\in\mathbb{N}}$ such that
\begin{equation*}
0<\lambda_{1}\leq\lambda_{2}\leq\cdot\cdot\cdot~~\mbox{and}~~\lambda_{k}\rightarrow\infty~~\mbox{as}~~k\rightarrow\infty,
\end{equation*}
and the corresponding eigenvector $\{e_{k}\}_{k=1}^{\infty}$ forms an orthonormal basis of $\mathbb{H}$ with
\begin{align*}
\mathcal{A}e_{k}=\lambda_{k}e_{k}.
\end{align*}
Every $u\in \mathbb{H}$ can thus be presented in the form
\begin{equation*}
u=\sum^{\infty}_{k=1}u_{k}e_{k},~~u_{k}=(u,e_{k}),
\end{equation*}
and, due to the Parseval equality,
\begin{equation*}
\|u\|^{2}=\sum^{\infty}_{k=1}u^{2}_{k}.
\end{equation*}
\par
As the usual way, we can define the space $\mathbb{H}^{s}$ as the domain $D(\mathcal{A}^{s})$ endowed with the norm $\|\cdot\|_{\mathbb{H}^{s}}=\|\mathcal{A}^{s}\cdot\|$, in particular $D(\mathcal{A}^{0})=\mathbb{H}$.
\par
We fix an integer $N$ and denote $P_{N}$ the orthogonal projector onto the space spanned by the first $N$ eigenvectors of $\mathcal{A}$, that is,
$$P_{N}u=\sum_{k\leq N}u_{k}e_{k},~u\in \mathbb{H}.$$
Let $Q_{N}=I-P_{N}$, then
\begin{equation*}
Q_{N}u=\sum_{k> N}u_{k}e_{k}
\end{equation*}
and, for every $u\in \mathbb{H}$ then $u=p+q$ where $p=P_{N}u$ and $q=Q_{N}u$.
\subsection{A spectral gap condition}\label{sec4.1}
\begin{theorem}\label{thm4.1}
Assume that there exists $N\in \mathbb{N}$ such that the spectral gap condition
\begin{align}\label{4.2}
\frac{\lambda^{1+\alpha}_{N+1}-\lambda^{1+\alpha}_{N}}{\lambda^{\alpha}_{N+1}+\lambda^{\alpha}_{N}}> L
\end{align}
holds. Then \eqref{4.1} satisfies the strong cone condition \eqref{2.9} in $H^{-\alpha}$, that is,
\begin{equation*}
\frac{1}{2}\frac{d}{dt}\mathcal{V}(u_{1}(t)-u_{2}(t))+\gamma \mathcal{V}(u_{1}(t)-u_{2}(t))
\leq-\mu\|u_{1}(t)-u_{2}(t)\|_{H^{-\alpha}}^{2},
\end{equation*}
with $\gamma=\frac{\lambda_{N}^{\alpha}\lambda_{N+1}^{\alpha}(\lambda_{N+1}+\lambda_{N})}{\lambda^{\alpha}_{N+1}+\lambda^{\alpha}_{N}}$ and $\mu=\frac{\lambda^{1+\alpha}_{N+1}-\lambda^{1+\alpha}_{N}}{\lambda^{\alpha}_{N+1}+\lambda^{\alpha}_{N}}-L$, where $\mathcal{V}(u):=\mathcal{V}_{N}(u)=\|Q_{N}u\|^{2}_{H^{-\alpha}}-\|P_{N}u\|^{2}_{H^{-\alpha}}$. Then \eqref{4.1} possesses an $N$-dimensional IM in the sense of Definition \ref{def2.4}.
\end{theorem}
\par
The sharp spectral gap condition \eqref{4.2} means that there exists globally Lipschitz nonlinearity $\mathcal{F}$ such that \eqref{4.1} does not have an IM if \eqref{4.2} is invalid (see, e.g. \cite{Z14}). One can refers to \cite{M91,R94} for more detail on the sharp spectral gap condition.
The proof of Theorem \ref{thm4.1} is similar to the argument given in \cite{Z14,GG18} (one also see \cite{M91,R94} and reference therein). For the sake of completeness, we will prove it in the Appendix. We mainly focus on the following situations.
\subsection{The spatial averaging scheme}\label{sec4.2}
\noindent

To construct an IM for the 3D hyperviscous NSEs in next section, we extend slightly the SAM to the abstract model \eqref{4.1}. In this subsection, the specific goal is to give the proof of Theorem \ref{thm1.3}, that is, we prove that ($\ref{4.1}$) possesses a $N$-dimensional IM  when the SAC \eqref{1.13} is assumed. To state the precise result, we define the following projectors. Let $N\in\mathbb{N}$ and $k>0$ be such that $k<\lambda_{N}$. Define
\begin{equation}\label{4.3}
\left\{
  \begin{aligned}
    &P_{k,N}u:=\sum_{\lambda_{j}<\lambda_{N}-k}(u,e_{j})e_{j},\\
    &R_{k,N}u:=\sum_{\lambda_{N}-k\leq\lambda_{j}\leq\lambda_{N}+k}(u,e_{j})e_{j},\\
    &Q_{k,N}u:=\sum_{\lambda_{j}>\lambda_{N}+k}(u,e_{j})e_{j},
  \end{aligned}
\right.
\end{equation}
for any $u\in \mathbb{H}$.
\par
The idea of the proof is that if the spectral gap condition (\ref{4.2}) is not fulfilled, but the SAC (\ref{1.13}) is valid, then a strong cone condition can be obtained.
\begin{theorem}\label{thm4.2}
For the abstract problem \eqref{4.1} in the Hilbert space $\mathbb{H}$. Let the assumption of Theorem \ref{thm1.3} be valid. Then the following strong cone condition is valid,
\begin{equation*}
\frac{d}{dt}V(t)+\left(\lambda^{1+\alpha}_{N+1}+\lambda^{1+\alpha}_{N}\right)V(t)
\leq-\frac{(1+\alpha)\lambda^{\alpha}_{N}}{4}\|v(t)\|^{2},~~\mbox{for all}~~t\geq 0.
\end{equation*}
\end{theorem}
\par
The proof of Theorem \ref{thm4.2} is similar to the argument in \cite{GG18,Z14}, for the reader's convenience, in Appendix \ref{sec6.1} we prove a strong cone condition for our problem provided that the SAC \eqref{1.13} is satisfied.
\begin{remark}\label{rem4.3}
As proved by Zelik \cite{Z14}, the result of Theorem \ref{thm4.2} can be extend to the more general SAC
\begin{equation*}
\|R_{k,N}\mathcal{F}'(u)R_{k,N}v-a(u)R_{k,N}v\|\leq\delta\|v\|,~~\mbox{for all}~~u,v\in \mathbb{H},
\end{equation*}
for some $\delta\leq\frac{1}{30}$, where $a(u)\in \mathbb{R}$ is a scalar depending continuously on $u\in \mathbb{H}$. The precise statement and proof are shown in the Appendix (see Section \ref{sec6.2}).
\end{remark}
\begin{proposition}\label{pro4.4}
Consider the abstract equation \eqref{4.1}and let $g\in \mathbb{H}$. Assume that $\mathcal{F}: \mathbb{H}\rightarrow \mathbb{H}$ is globally Lipschitz continuous with Lipschitz constant $L$. Suppose that the strong cone condition \eqref{2.9} holds for some $\gamma,~\mu>0$. Then, the cone invariance $(\ref{2.10})$ and squeezing property $(\ref{2.11})$ are valid.
\end{proposition}
\par
The ideal of proof follows Zelik \cite{Z14}, we write a detail proof in the Appendix (see Section \ref{sec6.1}). In the following, we will give an crucial result about the existence of IMs for problem \eqref{4.1}, which will be used to the problem \eqref{2.4}.
\begin{theorem}\label{thm4.5}
Let $g\in \mathbb{H}$. Assume that $\mathcal{F}: \mathbb{H}\rightarrow \mathbb{H}$ is globally Lipschitz continuous with Lipschitz constant $L>0$ and $\|A^{-3/8}\mathcal{F}(u)\|\leq C$ for any $u\in \mathbb{H}$. If the solution semigroup $S(t)$ satisfies the cone invariance $(\ref{2.10})$ and squeezing property $(\ref{2.11})$, then problem \eqref{4.1} possesses an $N$-dimensional inertial manifold $\mathcal{M}$ in sense of Definition \ref{def2.4}.
\end{theorem}
\par
The proof that the cone invariance $(\ref{2.10})$ and squeezing property $(\ref{2.11})$ leads to an IM for the problem \eqref{2.4} is similar to the classical case, we employ the method of Zelik \cite{KZ15}, a detail construction can be found in the Appendix (see Section \ref{sec6.1}) for our problem.
\begin{proof}[\textbf{Proof of Theorem \ref{thm1.3}}]
Combining Theorem \ref{thm4.2} and Proposition \ref{pro4.4} with Theorem \ref{thm4.5}, we can conclude that Theorem \ref{thm1.3} is valid.
\end{proof}
\begin{remark}\label{rem4.6}
When \eqref{4.2} is invalid for any $N\in\mathbb{N}$, then
\begin{equation*}
\frac{\lambda^{1+\alpha}_{N+1}-\lambda^{1+\alpha}_{N}}{\lambda^{\alpha}_{N+1}+\lambda^{\alpha}_{N}}\leq L,~~N\in\mathbb{N},
\end{equation*}
which implies from the elementary inequality that
\begin{equation*}
\lambda_{N+1}-\lambda_{N}\leq\frac{\lambda^{\alpha}_{N+1}+\lambda^{\alpha}_{N}}{(1+\alpha)\lambda^{\alpha}_{N}}\cdot L
\rightarrow \frac{2L}{1+\alpha}~(<2L),~~\text{as}~\lambda_{N}\rightarrow\infty.
\end{equation*}
Thus there exists $N\in\mathbb{N}$ such that $\lambda_{N+1}-\lambda_{N}\leq 2L$. Based the above analysis, the assumption $\lambda_{N+1}-\lambda_{N}\leq\frac{\lambda^{\alpha}_{N+1}+\lambda^{\alpha}_{N}}{(1+\alpha)\lambda^{\alpha}_{N}}\cdot L$ is replaced by $\lambda_{N+1}-\lambda_{N}\leq 2L$ in Theorem \ref{thm1.3}.
\end{remark}

\section{IMs for the 3D hyperviscous NSEs}\label{sec5}
\noindent

In this section, we devote to prove Theorem \ref{thm1.4}. Specifically, apply the abstract results established in Section \ref{sec4}, we consider the existence of IMs for the 3D incompressible hyperviscous NSEs
\begin{equation}\label{5.1}
\begin{cases}
\partial_{t}u+\nu(-\Delta)^{5/4} u+(u\cdot\nabla)u+\nabla p=f(x),~(t,x)\in\mathbb{R}_{+}\times\mathbb{T}^{3},\\
\nabla\cdot u=0,\\
u(0,x)=u_{0}(x).
\end{cases}
\end{equation}
\subsection{A priori estimates}\label{sec5.1}
\noindent

In this subsection, as that in subsection \ref{sec3.1} for $d=2$, we give some a priori estimates, including $H^{5/2}$-estimate on the stationary  solution $v$ of Eqs. (\ref{2.4}), $H^{\frac{5}{2}}$-estimate on the solution $u$ of Eqs. (\ref{2.4}), and $H^{9/2}$-estimate on the solution of the Eqs. \eqref{1.6} with $\theta=5/4$ (or (\ref{5.33}) below) that is achieved by subtracting between $u(t)$ and $v$.
\subsubsection{$H^{5/2}$-estimate on the solution of stationary equation}\label{sec5.1.1}
\noindent

We consider the stationary problem of Eqs. (\ref{2.4}):
\begin{equation}\label{5.2}
\nu A^{5/4}v+B(v,v)=f.
\end{equation}
Now, we intend to prove the $H^{5/2}$-regularity of the solutions of the auxiliary Eqs. (\ref{5.2}).
\begin{lemma}\label{lem5.1}
Let $f\in H$. Then, there exists at least one weak solution $v\in H^{5/4}$ of the problem \eqref{5.2}. Moreover, any solution $v$ of \eqref{5.2} belongs to $H^{5/2}$ and the following estimate holds
\begin{equation*}
\|v(t)\|_{H^{5/2}}\leq r_{v},
\end{equation*}
where
\begin{equation*}
r_{v}=\tilde{c}\left(\|f\|^{4}_{H^{-5/4}}+\|f\|^{10}_{H^{-5/4}}+\|f\|^{2}_{H^{-1}}\right)\|f\|^{2}_{H^{-5/4}}
+\frac{2}{\nu^{2}}\|f\|^{2}_{H},
\end{equation*}
here the constant $\tilde{c}$ depends only on $\nu$ and is independent of $v$ and $f$.
\end{lemma}
\begin{proof}
Firstly, in order to obtain the $H^{5/4}$-estimate on $v$, we multiply (\ref{5.2}) by $v$ and integrate over $\mathbb{T}^{3}$. Then, we get
\begin{equation}\label{5.3}
\nu\|v\|^{2}_{H^{5/4}}+(B(v,v),v)=(f,v).
\end{equation}
By the property of bilinear form (\ref{2.7}), we have
\begin{equation}\label{5.4}
(B(v,v),v)=0.
\end{equation}
The Young's inequality gives that
\begin{equation}\label{5.5}
|(f,v)|\leq\frac{1}{2\nu}\|f\|^{2}_{H^{-5/4}}+\frac{\nu}{2}\|v\|^{2}_{H^{5/4}}.
\end{equation}
Substituting (\ref{5.4}) and (\ref{5.5}) into (\ref{5.3}), we obtain
\begin{equation*}
\|v\|^{2}_{H^{5/4}}\leq\frac{1}{\nu^{2}}\|f\|^{2}_{H^{-5/4}}.
\end{equation*}
\par
Secondly, on the $H^{3/2}$-estimate of $v$, we multiply (\ref{5.2}) by $A^{1/4}v$ and integrate over $x\in\mathbb{T}^{3}$ to obtain that
\begin{equation}\label{5.6}
\nu\|v\|^{2}_{H^{3/2}}+(B(v,v),A^{1/4}v)=(f,A^{1/4}v).
\end{equation}
Since
\begin{equation}\label{5.7}
|(f,A^{1/4}v)|\leq\frac{1}{\nu}\|f\|^{2}_{H^{-1}}+\frac{\nu}{4}\|v\|^{2}_{H^{3/2}},
\end{equation}
\begin{align}\label{5.8}
|(B(v,v),A^{1/4}v)|
\leq \frac{1}{\nu}\|B(v,v)\|^{2}_{H^{-1}}+\frac{\nu}{4}\|v\|^{2}_{H^{3/2}}
\end{align}
and
\begin{align}\label{5.9}
\|B(v,v)\|^{2}_{H^{-1}}
\leq c\|v\|^{2}_{H^{1}}\|v\|^{2}_{H^{1}}\leq C\|v\|^{4}_{H^{5/4}},
\end{align}
here we have used the Young's inequality. Then, substituting (\ref{5.7})-(\ref{5.9}) into (\ref{5.6}) leads to the fact that
\begin{equation*}
\|v\|^{2}_{H^{3/2}}\leq C_{\nu}\|f\|^{4}_{H^{-5/4}}+\frac{2}{\nu^{2}}\|f\|^{2}_{H^{-1}},
\end{equation*}
where the constant $C_{\nu}$ depends only on $\nu$.
\par
Finally, in order to estimate the $H^{5/2}$-norm on $v$, we multiply (\ref{5.2}) by $A^{5/4}v$ and integrate over $x\in\mathbb{T}^{3}$. Then
\begin{equation}\label{5.10}
\nu\|v\|^{2}_{H^{5/2}}+(B(v,v),A^{5/4}v)=(f,A^{5/4}v).
\end{equation}
We estimate
\begin{align}\label{5.11}
|(B(v,v),A^{5/4}v)|
\leq &|(B(A^{1/2}v,v),A^{3/4}v)|+|(B(v,A^{1/2}v),A^{3/4}v)|\nonumber\\
\leq &c\left(\|A^{1/2}v\|_{H^{1/2}}\|v\|_{H^{1}}+\|v\|_{H^{1/2}}\|A^{1/2}v\|_{H^{1}}\right)\|A^{3/4}v\|_{H^{1}}\nonumber\\
\leq &c\left(\|v\|_{H^{3/2}}\|v\|_{H^{1}}+\|v\|_{H^{1/2}}\|v\|_{H^{2}}\right)\|v\|_{H^{5/2}}\nonumber\\
\leq &C\left(\|v\|_{H^{3/2}}\|v\|_{H^{1}}+\|v\|_{H^{1/2}}\|v\|^{1/5}_{H}\|v\|^{9/5}_{H^{5/2}}\right)\nonumber\\
\leq &C_{\nu}\left(\|v\|^{2}_{H^{3/2}}\|v\|^{2}_{H^{1}}+\|v\|^{10}_{H^{1/2}}\|v\|^{2}_{H}\right)+\frac{\nu}{4}\|v\|^{2}_{H^{5/2}}
\end{align}
here we have used the inequality (\ref{2.6}), the interpolation inequality  $\|v\|_{H^{2}}\leq\|v\|^{1/5}_{H}\|v\|^{4/5}_{H^{5/2}}$ in $\mathbb{T}^{3}$ and Young's inequality with exponent $(\frac{1}{10},\frac{9}{10})$. Note that
\begin{equation}\label{5.12}
|(f,A^{5/4}v)|\leq\frac{1}{\nu}\|f\|^{2}_{H}+\frac{\nu}{4}\|v\|^{2}_{H^{5/2}}.
\end{equation}
Then, substituting (\ref{5.11}) and (\ref{5.12}) into (\ref{5.10}), we deduce
\begin{equation*}
\|v\|^{2}_{H^{5/2}}\leq C_{\nu}\left(\|f\|^{4}_{H^{-5/4}}+\|f\|^{10}_{H^{-5/4}}+\|f\|^{2}_{H^{-1}}\right)\|f\|^{2}_{H^{-5/4}}
+\frac{2}{\nu^{2}}\|f\|^{2}_{H},
\end{equation*}
where the constant $C_{\nu}$ depends only on $\nu$.
\end{proof}
\subsubsection{$H^{5/2}$-estimate on the solution of Eqs. (\ref{2.4})}\label{sec5.1.2}
\noindent

We are ready to give some a priori estimates of the equation (\ref{2.4}).
\begin{theorem}\label{thm5.2}
Let $f\in H$ and $u_{0}\in H$. Then there exist a constant $C>0$ (depends only on $\nu$) such that for any solution $u(t)$ of \eqref{2.4} there is $t_{0}:=t_{0}(\|u_{0}\|_{H},\|f\|_{H^{-5/4}})>0$ satisfies that
\begin{equation*}
\|u(t)\|_{H^{s}}\leq r_{s},~~\forall~t\geq t_{0}+1,~~s=0,1,5/4,3/2~\mbox{or}~5/2,
\end{equation*}
where $r_{s}$ are given by the following formula:
\begin{equation*}
r^{2}_{0}=\frac{2}{\nu}\|f\|^{2}_{H^{-5/4}},
\end{equation*}
\begin{equation*}
r^{2}_{1/2}
=C\exp\left\{r^{2}_{0}+\|f\|^{2}_{H^{-5/4}}\right\}
\cdot\left(r^{2}_{0}+\|f\|^{2}_{H^{-5/4}}\right),
\end{equation*}
\begin{equation*}
r^{2}_{1}
=C\exp\left\{r^{2}_{1/2}\left(r^{2}_{0}+\|f\|^{2}_{H^{-5/4}}\right)\right\}\cdot
\left(r^{2}_{0}+\|f\|^{2}_{H^{-5/4}}+\|f\|_{H^{-1/4}}^{2}\right),
\end{equation*}
\begin{equation*}
r^{2}_{5/4}=C\exp\left\{r^{2}_{0}r^{8}_{1}\right\}
\left(r^{2}_{0}+\|f\|^{2}_{H^{-5/4}}+\|f\|^{2}_{H}\right),
\end{equation*}
\begin{align*}
r_{3/2}=C\left(\bar{r}^{2}_{0}+r^{2}_{1}+\|f\|_{H^{-1}}\right),
\end{align*}
\begin{align*}
r_{2}=C\left(\bar{r}^{2}_{0}+r_{1}r_{3/2}+\|f\|_{H^{-1/2}}\right),
\end{align*}
\begin{align*}
r_{5/2}=C\left(\bar{r}^{2}_{0}+r_{1}r_{2}+\|f\|_{H}\right),
\end{align*}
here the $\bar{r}^{2}_{0}$ denotes the uniform upper bound of $\|\partial_{t}u(t)\|^{2}_{H}$ when $t\geq t_{0}+1$ given by $(\ref{5.31})$ and $t_{0}$ is chosen by $(\ref{5.16})$.
\end{theorem}
\begin{proof}
\textbf{Step 1.} ($H$-estimate on $u$). Take the scalar product of (\ref{2.4}) with $u$. Then, we get
\begin{equation*}
\frac{1}{2}\frac{d}{dt}\|u\|^{2}_{H}+\nu\|u\|^{2}_{H^{5/4}}+(B(u,u),u)=(f,u).
\end{equation*}
Since $f\in H$, due to the Young's inequality, we have
\begin{equation}\label{5.13}
|(f,u)|\leq \frac{1}{2\nu}\|f\|^{2}_{H^{-5/4}}+\frac{\nu}{2}\|u\|^{2}_{H^{5/4}}.
\end{equation}
Combine \eqref{5.13} with the property of the bilinear form $B$: $(B(u,u),u)=0$ to give that
\begin{equation}\label{5.14}
\frac{d}{dt}\|u\|^{2}_{H}+\nu\|u\|^{2}_{H^{5/4}}\leq \frac{1}{\nu}\|f\|^{2}_{H^{-5/4}}.
\end{equation}
By the Poincar\'{e} and Gronwall's inequalities, we know that there exists $\gamma>0$ such that
\begin{equation}\label{5.15}
\|u(t)\|^{2}_{H}\leq e^{-\gamma t}\|u(0)\|^{2}_{H}
+\frac{1}{\nu}\|f\|^{2}_{H^{-5/4}}.
\end{equation}
\par
Consequently, it follows from (\ref{5.15}) that
\begin{equation*}
\limsup_{t\rightarrow +\infty}\|u(t)\|^{2}_{H}\leq\frac{1}{\nu}\|f\|^{2}_{H^{-5/4}},
\end{equation*}
which implies that there exists a $t_{0}=t_{0}(\|u_{0}\|_{H})>0$ such that
\begin{equation}\label{5.16}
\|u(t)\|^{2}_{H}\leq\frac{2}{\nu}\|f\|^{2}_{H^{-5/4}}:=r^{2}_{0},~~\mbox{for}~~t\geq t_{0}.
\end{equation}
\par
\textbf{Step 2.} ($H^{\frac{1}{2}}$-estimate on $u$). Taking the scalar product of (\ref{2.4}) with $A^{1/2}u$ to obtain that
\begin{equation*}
\frac{1}{2}\frac{d}{dt}\|u\|_{H^{\frac{1}{2}}}^{2}+\nu\|u\|_{H^{7/4}}^{2}+(B(u,u),A^{1/2}u)=(f,A^{1/2}u)
\leq \frac{1}{\nu}\|f\|^{2}_{H^{-3/4}}+\frac{\nu}{2}\|u\|_{H^{7/4}}^{2}.
\end{equation*}
We estimate from \eqref{2.6} that
\begin{align*}
\left|(B(u,u),A^{1/2}u)\right|
\leq &\left|(A^{1/8}B(u,u),A^{3/8}u)\right|
\leq c\|u\|_{H^{1/2}}\|A^{1/8}u\|_{H^{1}}\|A^{3/8}u\|_{H^{1}}\nonumber\\
\leq &C\|u\|_{H^{1/2}}\|u\|_{H^{5/4}}\|u\|_{H^{7/4}}
\leq C_{\nu}\|u\|^{2}_{H^{1/2}}\|u\|^{2}_{H^{5/4}}+\frac{\nu}{2}\|u\|_{H^{7/4}}^{2}.
\end{align*}
Thus
\begin{equation}\label{5.17}
\frac{d}{dt}\|u\|_{H^{\frac{1}{2}}}^{2}+\nu\|u\|_{H^{7/4}}^{2}\leq C_{\nu}\|u\|^{2}_{H^{1/2}}\|u\|^{2}_{H^{5/4}},
\end{equation}
which shows from the uniform Gronwall lemma that
\begin{align*}
\|u(t+1)\|^{2}_{H^{1/2}}\leq C_{\nu}\exp\left\{\int^{t+1}_{t}\|u(s)\|^{2}_{H^{5/4}}ds\right\}\cdot\left(\int^{t+1}_{t}\|u(s)\|^{2}_{H^{1/2}}ds\right),
~~\forall~t\geq 0.
\end{align*}
On the other hand, by \eqref{5.14} and \eqref{5.16}, we have
\begin{align}\label{5.18}
\int^{t+1}_{t}\|u(s)\|^{2}_{H^{1/2}}ds\leq\int^{t+1}_{t}\|u(s)\|^{2}_{H^{5/4}}ds
\leq r^{2}_{0}+\frac{1}{\nu}\|f\|^{2}_{H^{-5/4}},~~\forall~t\geq t_{0}.
\end{align}
Then
\begin{align}\label{5.19}
\|u(t)\|^{2}_{H^{1/2}}\leq C_{\nu}\exp\left\{r^{2}_{0}+\frac{1}{\nu}\|f\|^{2}_{H^{-5/4}}\right\}
\cdot\left(r^{2}_{0}+\frac{1}{\nu}\|f\|^{2}_{H^{-5/4}}\right):=r^{2}_{1/2},~~\forall~t\geq t_{0}+1.
\end{align}
\par
\textbf{Step 3.} ($H^{1}$-estimate on $u$). Taking the scalar product of (\ref{2.4}) with $Au$ gives us that
\begin{equation}\label{5.20}
\frac{1}{2}\frac{d}{dt}\|u\|_{H^{1}}^{2}+\nu\|u\|_{H^{9/4}}^{2}+(B(u,u),Au)=(f,Au).
\end{equation}
Note that
\begin{equation}\label{5.21}
|(f,Au)|\leq \frac{1}{\nu}\|f\|_{H^{-1/4}}^{2}+\frac{\nu}{4}\|Au\|^{2}_{H^{1/4}}
=\frac{1}{\nu}\|f\|_{H^{-1/4}}^{2}+\frac{\nu}{4}\|u\|^{2}_{H^{9/4}}
\end{equation}
and
\begin{align}\label{5.22}
|(B(u,u),Au)|
\leq \|B(u,u)\|_{H^{-1/4}}\|Au\|_{H^{1/4}}
\leq \|u\|_{H^{7/4}}\|u\|_{H^{1}}\|u\|_{H^{9/4}}
\leq C_{\nu}\|u\|^{2}_{H^{7/4}}\|u\|^{2}_{H^{1}}+\frac{\nu}{4}\|u\|^{2}_{H^{9/4}},
\end{align}
here we have used Young's inequality. Substituting (\ref{5.21}) and (\ref{5.22}) into (\ref{5.20}), we obtain
\begin{align}\label{5.23}
\frac{d}{dt}\|u\|_{H^{1}}^{2}+\nu\|u\|_{H^{9/4}}^{2}
\leq C_{\nu}\|u\|^{2}_{H^{7/4}}\|u\|^{2}_{H^{1}}+\frac{1}{\nu}\|f\|_{H^{-1/4}}^{2}.
\end{align}
On the other hand, from (\ref{5.17})-(\ref{5.19}), we know that
\begin{align*}
\|u\|^{2}_{L^{2}\left(t,t+1;H^{7/4}\right)}
\leq C_{\nu}r^{2}_{1/2}\left(r^{2}_{0}+\|f\|^{2}_{H^{-5/4}}\right),~~\forall~~t\geq t_{0}+1,
\end{align*}
where the constant $C_{\nu}$ is independent of time. Combining this with (\ref{5.23}) and applying the uniform Gronwall-type lemma \ref{lem2.7}, we obtain
\begin{align*}
\|u(t)\|_{H^{1}}^{2}
\leq& \exp\left\{C_{\nu}\|u\|^{2}_{L^{2}\left(t,t+1;H^{7/4}\right)}\right\}\cdot
\left(\|u\|^{2}_{L^{2}(t,t+1;H^{1})}+\frac{1}{\nu}\|f\|_{H^{-1/4}}^{2}\right)\nonumber\\
\leq& \exp\left\{C_{\nu}r^{2}_{1/2}\left(r^{2}_{0}+\|f\|^{2}_{H^{-5/4}}\right)\right\}\cdot
\left(r^{2}_{0}+\frac{1}{\nu}\|f\|^{2}_{H^{-5/4}}+\frac{1}{\nu}\|f\|_{H^{-1/4}}^{2}\right)
:=r^{2}_{1},~~\forall~t\geq t_{0}+1.
\end{align*}

\textbf{Step 4.} ($H^{5/4}$-estimate on $u$.) Take the scalar product of (\ref{2.4}) with $A^{5/4}u$. Then, we have
\begin{equation*}
\frac{1}{2}\frac{d}{dt}\|u\|_{H^{5/4}}^{2}+\nu\|u\|_{H^{5/2}}^{2}+(B(u,u),A^{5/4}u)
=(f,A^{5/4}u).
\end{equation*}
It is easy to see from the Young's inequality and $f\in H$ that
\begin{equation*}
|(f,A^{5/4}u)|
\leq \frac{1}{\nu}\|f\|^{2}_{H}+\frac{\nu}{4}\|u\|^{2}_{H^{5/2}}.
\end{equation*}
By the definition $B(u,u)$, we estimate
\begin{align*}
|(B(u,u),A^{5/4}u)|
\leq& \|B(u,u)\|_{H}\|A^{5/4}u\|_{H}
\leq c\|u\|_{H^{2}}\|u\|_{H^{1}}\|u\|_{H^{5/2}}\nonumber\\
\leq& C\|u\|^{\frac{1}{5}}_{H}\|u\|^{\frac{4}{5}}_{H^{5/2}}\|u\|_{H^{1}}\|u\|_{H^{5/2}}
\leq C_{\nu}\|u\|^{2}_{H}\|u\|^{10}_{H^{1}}+\frac{\nu}{4}\|u\|^{2}_{H^{5/2}},
\end{align*}
where we have used Nirenberg-Gagliardo inequality
$\|u\|_{H^{2}}\leq \|u\|^{\frac{1}{5}}_{H}\|u\|^{\frac{4}{5}}_{H^{5/2}}$
and the Young's inequality $ab\leq 4\nu^{-1}a^{10/9}+C_{\nu}b^{10},~\forall~a,b>0$. This implies that
\begin{equation}\label{5.24}
\frac{d}{dt}\|u\|_{H^{5/4}}^{2}+\nu\|u\|_{H^{5/2}}^{2}
\leq C_{\nu}\|u\|^{2}_{H}\|u\|^{10}_{H^{1}}+\frac{1}{\nu}\|f\|^{2}_{H}
\leq \tilde{C}_{\nu}\|u\|^{2}_{H}\|u\|^{8}_{H^{1}}\|u\|^{2}_{H^{5/4}}+\frac{1}{\nu}\|f\|^{2}_{H}.
\end{equation}
Let us drop $\nu\|u\|_{H^{5/2}}^{2}$ on the left-hand side of \eqref{5.24} and use the uniform Gronwall lemma and the estimate \eqref{5.18} to obtain that
\begin{align*}
\|u(t)\|^{2}_{H^{5/4}}
\leq& \frac{1}{\nu}\exp\left\{\tilde{C}_{\nu}r^{2}_{0}r^{8}_{1}\right\}
\left(r^{2}_{0}+\nu^{-1}\|f\|^{2}_{H^{-5/4}}+\|f\|^{2}_{H}\right):=r^{2}_{5/4},~~\forall~t\geq t_{0}+1,
\end{align*}
which, in addition, implies that
\begin{equation}\label{5.25}
\int^{t+1}_{t}\|u(s)\|_{H^{5/2}}^{2} ds
\leq C_{\nu}\left(r^{2}_{5/4}\left(1+r^{2}_{0}r^{8}_{1}\right)+\|f\|^{2}_{H}\right),
~~\forall~t\geq t_{0}+1.
\end{equation}
\par
Next, we will continue to estimate the uniform boundedness in more regular space until we get $H^{5/2}$-estimate. Similarly, multiplying (\ref{2.4}) by $A^{k}u$ ($k>5/4$) can not obtain $H^{k}$ estimate since we assume that $f\in H$ only. To achieve the desired result, we set $\psi=\partial_{t}u$ and differentiate Eqs. (\ref{2.4}) with respect to $t$ to consider the following equation.
\begin{equation}\label{5.26}
\partial_{t}\psi+\nu A^{5/4}\psi+B(\psi,u)+B(u,\psi)=0.
\end{equation}
Then we can obtain the desired result step by step.
\par
\textbf{Step 5.} ($H$-estimate on $\partial_{t} u$.) Taking the inner product of (\ref{5.26}) with $\psi$, we can find that
\begin{equation*}
\frac{1}{2}\frac{d}{dt}\|\psi\|^{2}_{H}+\nu\|\psi\|^{2}_{H^{5/4}}+(B(\psi,u),\psi)+(B(u,\psi),\psi)=0.
\end{equation*}
At the same time, considering $(B(u,\psi),\psi)=0$, \eqref{2.7} and
\begin{align*}
\left|(B(\psi,u),\psi)\right|
\leq& C\|\psi\|_{H^{1/2}}\|u\|_{H^{1}}\|\psi\|_{H^{1}}
\leq  C\|\psi\|^{1/2}_{H}\|\psi\|^{1/2}_{H^{1}}\|u\|_{H^{1}}\|\psi\|_{H^{1}}\nonumber\\
\leq& C\|\psi\|^{1/2}_{H}\|u\|_{H^{1}}\|\psi\|^{3/2}_{H^{5/4}}
\leq C_{\nu}\|\psi\|^{2}_{H}\|u\|^{4}_{H^{1}}+\frac{\nu}{2}\|\psi\|^{2}_{H^{5/4}},
\end{align*}
which gives us that
\begin{equation}\label{5.27}
\frac{d}{dt}\|\psi\|^{2}_{H}+\nu\|\psi\|^{2}_{H^{5/4}}
\leq C_{\nu}\|\psi\|^{2}_{H}\|u\|^{4}_{H^{1}}.
\end{equation}
An application of the uniform Gronwall lemma leads to that
\begin{equation}\label{5.28}
\|\psi(t)\|^{2}_{H}
\leq \exp\left\{C_{\nu}r^{4}_{1}\right\}\left\{\|\psi\|_{L^{2}(t,t+1;H)}\right\},~~\forall~t\geq t_{0}+1.
\end{equation}
\par
On the other hand, by the Eqs. \eqref{2.4} and the assumption that $f\in H$, we have
\begin{align}\label{5.29}
\|\psi\|_{H}=&\|\partial_{t}u\|_{H}
\leq \nu\|A^{5/4}u\|_{H}+\|B(u,u)\|_{H}+\|f\|_{H}\nonumber\\
\leq& \nu\|u\|_{H^{5/2}}+C\|u\|_{H^{2}}\|u\|_{H^{1}}+\|f\|_{H}\nonumber\\
\leq& \nu\|u\|_{H^{5/2}}+Cr_{1}\|u\|_{H^{2}}+\|f\|_{H},~~\forall~t\geq t_{0}+1.
\end{align}
So, combining \eqref{5.29} with \eqref{5.25}, we deduce that
\begin{align}\label{5.30}
\|\psi\|^{2}_{L^{2}(t,t+1;H)}
\leq& C_{\nu}\left(1+r^{2}_{1}\right)\int^{t+1}_{t}\|u(s)\|^{2}_{H^{5/2}} ds+\|f\|^{2}_{H}\nonumber\\
\leq& C_{\nu}r^{2}_{5/4}\left(1+r^{2}_{1}\right)\left(1+r^{2}_{0}r^{3}_{1}+\|f\|^{2}_{H}\right),
~~\forall~t\geq t_{0}+1.
\end{align}
\par
Therefore, we can conclude from \eqref{5.28} and \eqref{5.30} that
\begin{equation}\label{5.31}
\|\psi(t)\|^{2}_{H}
\leq \exp\left\{C_{\nu}r^{2}_{1}\right\}
\left\{C_{\nu}r^{2}_{5/4}\left(1+r^{2}_{1}\right)\left(1+r^{2}_{0}r^{8}_{1}+\|f\|^{2}_{H}\right)\right\}
:=\bar{r}^{2}_{0},~~\forall~t\geq t_{0}+1.
\end{equation}
\par
\textbf{Step 6.} ($H^{5/2}$-estimate on $u$.) Firstly, we can obtain the $H^{3/2}$-estimate on $u$. Indeed, by Eqs. \eqref{2.4}, we find that
\begin{align*}
\nu\|A^{5/4}u\|_{H^{-1}}
\leq \|\partial_{t}u\|_{H^{-1}}+\|B(u,u)\|_{H^{-1}}+\|f\|_{H^{-1}}
\leq \|\partial_{t}u\|_{H^{-1}}+C\|u\|^{2}_{H^{1}}+\|f\|_{H^{-1}}.
\end{align*}
This means that
\begin{align*}
\|u(t)\|_{H^{3/2}}
\leq C_{\nu}\left(\bar{r}^{2}_{0}+r^{2}_{1}+\|f\|_{H^{-1}}\right):=r_{3/2}
,~~\forall~t\geq t_{0}+1.
\end{align*}
Then, from the Eqs. \eqref{2.4}, one has
\begin{align*}
\|u(t)\|_{H^{2}}=&\|A^{5/4}u\|_{H^{-1/2}}
\leq \nu^{-1}\left(\|\partial_{t}u\|_{H^{-1/2}}+\|B(u,u)\|_{H^{-1/2}}+\|f\|_{H^{-1/2}}\right)\nonumber\\
\leq& C_{\nu}\left(\|\partial_{t}u\|_{H}+\|u\|^{2}_{H^{1}}\|u\|^{2}_{H^{3/2}}+\|f\|_{H^{-1/2}}\right)
\leq C_{\nu}\left(\bar{r}^{2}_{0}+r_{1}r_{3/2}+\|f\|_{H^{-1/2}}\right):=r_{2},
\end{align*}
for any $t\geq t_{0}+1$. Now, using the Eqs. \eqref{2.4} again, then we can easily to see that
\begin{align*}
\nu\|A^{5/4}u\|_{H}
\leq \|\partial_{t}u\|_{H}+\|B(u,u)\|_{H}+\|f\|_{H}
\leq \|\partial_{t}u\|_{H}+C\|u\|_{H^{2}}\|u\|_{H^{1}}+\|f\|_{H}.
\end{align*}
It follows that
\begin{align}\label{5.32}
\|u(t)\|_{H^{5/2}}
\leq C_{\nu}\left(\bar{r}^{2}_{0}+r_{1}r_{2}+\|f\|_{H}\right):=r_{5/2},~~\forall~t\geq t_{0}+1.
\end{align}
\end{proof}
\subsubsection{Asymptotic regularity: $H^{9/2}$-estimate}\label{sec5.1.3}
\noindent

The existence of an absorbing set in $H^{9/2}$ is extremely crucial for the construction of an IM for (\ref{2.4}) in our frame as we modify the nonlinearity outside the absorbing ball and show that the modified nonlinearity satisfies the SAC in the subsequent sections. Similar to the 2D case, we consider equations on $w:=u-v$, where $u$ solves problem (\ref{5.1}) and $v\in H^{5/2}$ solves the stationary problem (\ref{5.2}), that is, $v$ is the stationary point of (\ref{5.1}) (there may be several stationary points of this problem, we just fix one of them). After applying Helmholtz-Leray projection we get
\begin{equation}\label{5.33}
\begin{cases}
\nu A^{5/4}w+B(w,w)+B(v,w)+B(w,v)=-\partial_{t}u,\\
w(0)=u_{0}-v:=w_{0}.
\end{cases}
\end{equation}
Again, since we remove a stationary point of system (\ref{2.4}) that does not affect its long time behavior, we will study the solutions $w$ and the attractor of the Eqs. (\ref{5.33}) instead of solutions $u$ of Eqs. (\ref{2.4}). The advantage of this approach is that the solutions of (\ref{5.33}) are more regular.
\begin{lemma}\label{lem5.3}
Let $f\in H$ and $u_{0}\in H$ then for any solution $u(t)$ of problem \eqref{2.4} there exists $\bar{r}_{1},\bar{r}_{1}>0$ and $\tilde{C}>0$ (depends only on $\nu$) such that
\begin{equation*}
\|\partial_{t}u\|_{H^{s}}\leq\bar{r}_{s},~~\forall~t\geq t_{0}+1,~~s=1,2
\end{equation*}
with
\begin{equation*}
\bar{r}^{2}_{1}
\leq \tilde{C}\exp\left\{(r_{2})^{9/4}\right\}
\cdot\left\{\left(1+r^{4}_{1}\right)\bar{r}^{2}_{0}\right\}
\end{equation*}
and
\begin{equation*}
\bar{r}^{2}_{2}
\leq \tilde{C}\exp\{r^{2}_{2}\}\left(1+r^{2}_{2}\right)\bar{r}^{2}_{1},
\end{equation*}
where $t_{0}:=t_{0}(\|u_{0}\|_{H},\|f\|_{H^{-5/4}})>0$ is the same as in Theorem \ref{thm5.2}.
\end{lemma}
\begin{proof}
Firstly, we give the $H^{1}$-estimate of $\partial_{t}u$. Taking the inner product of (\ref{5.26}) with $A\psi$, we obtain
\begin{equation}\label{5.34}
\frac{1}{2}\frac{d}{dt}\|\psi\|^{2}_{H^{1}}+\nu\|\psi\|^{2}_{H^{9/4}}+(B(\psi,u),A\psi)+(B(u,\psi),A\psi)=0.
\end{equation}
Note that
\begin{align}\label{5.35}
\left|(B(\psi,u),A\psi)\right|
\leq& \left|(B(A^{3/8}\psi,u),A^{5/8}\psi)\right|+\left|(B(\psi,A^{3/8}u),A^{5/8}\psi)\right|\nonumber\\
\leq& C\left(\|A^{3/8}\psi\|_{H^{1/2}}\|u\|_{H^{1}}\|A^{5/8}\psi\|_{H^{1}}
+\|\psi\|_{H^{1/2}}\|A^{3/8}u\|_{H^{1}}\|A^{5/8}\psi\|_{H^{1}}\right)\nonumber\\
\leq& C\left(\|\psi\|_{H^{5/4}}\|u\|_{H^{1}}\|\psi\|_{H^{9/4}}
+\|\psi\|_{H^{1/2}}\|u\|_{H^{7/4}}\|\psi\|_{H^{9/4}}\right)\nonumber\\
\leq& C\left(\|\psi\|_{H^{5/4}}\|u\|_{H^{1}}\|\psi\|_{H^{9/4}}
+\|\psi\|_{H^{5/4}}\|u\|_{H^{7/4}}\|\psi\|_{H^{9/4}}\right)\nonumber\\
\leq& C\left(\|u\|_{H^{1}}+\|u\|_{H^{7/4}}\right)\|\psi\|^{8/9}_{H^{1}}\|\psi\|^{10/9}_{H^{9/4}}\nonumber\\
\leq& C_{\nu}\left(\|u\|^{9/4}_{H^{1}}+\|u\|^{9/4}_{H^{7/4}}\right)\|\psi\|^{2}_{H^{1}}+\frac{\nu}{4}\|\psi\|^{2}_{H^{9/4}},
\end{align}
where we have used the property \eqref{2.6}, the interpolation inequality $\|\psi\|_{H^{5/4}}\leq c\|\psi\|^{8/9}_{H^{1}}\|\psi\|^{1/9}_{H^{9/4}}$ and the Young's inequality with exponent $(9/4,9/5)$. Analogously,
\begin{align}\label{5.36}
|(B(u,\psi),A\psi)|
\leq C_{\nu}\left(\|u\|^{9/4}_{H^{1}}+\|u\|^{9/4}_{H^{7/4}}\right)\|\psi\|^{2}_{H^{1}}+\frac{\nu}{4}\|\psi\|^{2}_{H^{9/4}}.
\end{align}
Substituting (\ref{5.35}) and (\ref{5.36}) into (\ref{5.34}), we deduce
\begin{equation}\label{5.37}
\frac{d}{dt}\|\psi\|^{2}_{H^{1}}+\nu\|\psi\|^{2}_{H^{9/4}}
\leq C_{\nu}\left(\|u\|^{9/4}_{H^{1}}+\|u\|^{9/4}_{H^{7/4}}\right)\|\psi\|^{2}_{H^{1}}
\leq C_{\nu}\|u\|^{9/4}_{H^{2}}\|\psi\|^{2}_{H^{1}}.
\end{equation}
On the other hand, by (\ref{5.27}) and \eqref{5.31}, we know that
\begin{align*}
\int^{t+1}_{t}\|\psi(s)\|^{2}_{H^{1}}ds
\leq \int^{t+1}_{t}\|\psi(s)\|^{2}_{H^{5/4}}ds
\leq C_{\nu}\left(1+r^{4}_{1}\right)\bar{r}^{2}_{0},
~~\forall~t\geq t_{0}+1.
\end{align*}
Substituting above estimate into (\ref{5.37}) and applying Gronwall-type Lemma \ref{lem2.7} give that
\begin{align}\label{5.38}
\|\psi\|^{2}_{H^{1}}
\leq& \exp\left\{C_{\nu}(r_{2})^{9/4}\right\}
\cdot\left\{C_{\nu}\left(1+r^{4}_{1}\right)\bar{r}^{2}_{0}\right\}:=\bar{r}^{2}_{1},~~\forall~t\geq t_{0}+1.
\end{align}
\par
Now we ready to estimate $\|\psi\|_{H^{2}}$. To do this, multiplying (\ref{5.26}) by $A^{2}\psi$ and integrating over $x\in\mathbb{T}^{3}$, we obtain
\begin{equation}\label{5.39}
\frac{1}{2}\frac{d}{dt}\|\psi\|^{2}_{H^{2}}+\nu\|\psi\|^{2}_{H^{13/4}}+(B(\psi,u),A^{2}\psi)
+(B(u,\psi),A^{2}\psi)=0.
\end{equation}
We can estimate
\begin{align}\label{5.40}
\left|(B(\psi,u),A^{2}\psi)\right|
\leq& \left|(B(A^{7/8}u,\psi),A^{9/8}\psi)\right|
+\left|(B(u,A^{7/8}\psi),A^{9/8}\psi)\right|\nonumber\\
\leq& \left(\|B(A^{7/8}u,\psi)\|_{H^{-1}}+\|B(u,A^{7/8}\psi)\|_{H^{-1}}\right)
\|A^{9/8}\psi\|_{H^{1}}\nonumber\\
\leq& \left(\|A^{7/8}u\|_{H}\|\psi\|_{L^{\infty}}
+\|u\|_{L^{\infty}}\|A^{7/8}\psi\|_{H}\right)\|\psi\|^{2}_{H^{13/4}}\nonumber\\
\leq& C_{\nu}\left(\|u\|_{H^{7/4}}\|\psi\|_{H^{2}}
+\|u\|_{H^{2}}\|\psi\|_{H^{7/4}}\right)\|\psi\|^{2}_{H^{13/4}}\nonumber\\
\leq& C_{\nu}\|\psi\|^{2}_{H^{2}}\|u\|^{2}_{H^{2}}+\frac{\nu}{4}\|\psi\|^{2}_{H^{13/4}},
\end{align}
here we have used the Sobolev embedding $H^{2}\hookrightarrow L^{\infty}$ in $\mathbb{T}^{3}$. Similarly,
\begin{align}\label{5.41}
\left|(B(u,\psi),A^{2}\psi)\right|
\leq C_{\nu}\|\psi\|^{2}_{H^{2}}\|u\|^{2}_{H^{2}}+\frac{\nu}{4}\|\psi\|^{2}_{H^{13/4}}.
\end{align}
Substituting (\ref{5.40}) and (\ref{5.41}) into (\ref{5.39}) leads to the following differential inequality
\begin{align}\label{5.42}
\frac{d}{dt}\|\psi\|^{2}_{H^{2}}+\nu\|\psi\|^{2}_{H^{13/4}}
\leq C_{\nu}\|u\|^{2}_{H^{2}}\|\psi\|^{2}_{H^{2}}
\leq C_{\nu}r^{2}_{2}\|\psi\|^{2}_{H^{2}},~~\forall~t\geq t_{0}+1,
\end{align}
here we have used the result that established in \ref{thm5.2}. Besides, by (\ref{5.37}) and (\ref{5.38}), we know that
\begin{align}\label{5.43}
\int^{t+1}_{t}\|\psi(s)\|^{2}_{H^{2}}ds\leq\int^{t+1}_{t}\|\psi(s)\|^{2}_{H^{9/4}}ds
\leq C_{\nu}\left(1+r^{2}_{2}\right)\bar{r}^{2}_{1},~~\forall~t\geq t_{0}+1.
\end{align}
Combining (\ref{5.42}) with (\ref{5.43}) and applying the uniform Gronwall-type Lemma \ref{lem2.7}, we deduce
\begin{align*}
\|\psi\|^{2}_{H^{2}}
\leq \exp\{C_{\nu}r^{2}_{2}\}\times\|\psi\|_{L^{2}(t,t+1;H^{2})}
\leq C_{\nu}\exp\{C_{\nu}r^{2}_{2}\}\left(1+r^{2}_{2}\right)\bar{r}^{2}_{1}:=\bar{r}^{2}_{2},
\end{align*}
for any $t\geq t_{0}+1$. Thus we complete the proof.
\end{proof}
\begin{theorem}\label{thm5.4}
Assume that $f\in H$. Let $w(t)$ be a solution of Eqs. \eqref{5.33} with $w(0)\in H$. Then there exists $t'_{0}=t'_{0}(\|w(0)\|_{H},\|f\|_{H^{-5/4}})>0$ such that
\begin{equation}\label{5.44}
\|w(t)\|_{H^{9/2}}\leq\varrho_{3},~~\forall~t\geq t'_{0}+1,
\end{equation}
where $\varrho_{3}$ is given by the following formula
\begin{equation*}
\varrho^{2}_{3}=K\left(r^{2}_{v}+r^{2}_{5/2}+\bar{r}^{2}_{2}\right),
\end{equation*}
here $r^{2}_{5/2}$ denotes the uniform upper bound of $\|u(t)\|^{2}_{H^{5/2}}$ in the large time given in $(\ref{5.32})$ and the constant $K$ depends only on $\nu$.
\end{theorem}
\begin{proof}
By Theorem \ref{thm5.2}, we know that
\begin{equation*}
\|u\|^{2}_{L^{\infty}\left(t,+\infty;H^{5/2}\right)}\leq r^{2}_{5/2},~~\forall~ t\geq t_{0}+1,
\end{equation*}
for some $t_{0}>0$ which depends on $\|u_{0}\|_{H}$ and $\|f\|_{H^{-5/4}}$. Then there exists $t'_{0}=t'_{0}(\|w(0)\|_{H})\geq t_{0}$ such that
\begin{equation}\label{5.45}
\|w\|^{2}_{L^{\infty}\left(t,+\infty;H^{5/2}\right)}
\leq 2\left(\|u\|^{2}_{L^{\infty}(t,+\infty;H^{2})}+\|v\|^{2}_{H^{5/2}}\right)
\leq 2\left(r^{2}_{5/2}+r^{2}_{v}\right),~~\forall~t\geq~t'_{0}+1,
\end{equation}
thanks to the equality $w=u-v$ and the fact that $v\in H^{5/2}$ obtained in Lemma \ref{lem5.1}. Due to Theorem \ref{lem5.3}, we see that
\begin{equation}\label{5.46}
\|\partial_{t}u\|^{2}_{L^{\infty}(t,+\infty;H^{2})}\leq \bar{r}^{2}_{2},
\end{equation}
for any $t\geq t'_{0}+1$. Then, one can estimate
\begin{align}\label{5.47}
&\left\|\left(-B(w,w)-B(v,w)-B(w,v)-\partial_{t}u\right)\right\|_{H^{1/2}}\nonumber\\
\leq& c\left(\|w\|_{L^{\infty}}\|\nabla\cdot w\|_{H^{1/2}}+2\|v\|_{L^{\infty}}\|\nabla\cdot w\|_{H^{1/2}}\right)
+\|\partial_{t}u\|_{H^{1/2}}\nonumber\\
\leq& c\left(\|w\|_{H^{2}}\|w\|_{H^{3/2}}+2\|v\|_{H^{2}}\|w\|_{H^{3/2}}\right)+\|\partial_{t}u\|_{H^{2}},
\end{align}
here we have used the Sobolev embedding $H^{2}(\mathbb{T}^{3})\hookrightarrow L^{\infty}(\mathbb{T}^{3})$ again. Therefore, by (\ref{5.45}) and (\ref{5.46}), we deduce
\begin{align*}
\left(-B(w,w)-B(v,w)-B(w,v)-\partial_{t}u\right)\in L^{\infty}(t,+\infty;H^{1/2}),
\end{align*}
for any $t\geq t'_{0}+1$. Moreover, by (\ref{5.45}) and (\ref{5.47}), we can give an uniform upper bound on the $H^{1/2}$ norm in the large time, that is,
\begin{align}\label{5.48}
&\left\|\left(-B(w,w)-B(v,w)-B(w,v)-\partial_{t}u\right)\right\|_{L^{\infty}(t,+\infty;H^{1/2})}\nonumber\\
\leq& C\left(\|w\|^{2}_{L^{\infty}\left(t,+\infty;H^{5/2}\right)}
+2\|v\|_{H^{5/2}}\|w\|_{L^{\infty}\left(t,+\infty;H^{5/2}\right)}\right)
+\left\|\partial_{t}u\right\|_{L^{\infty}\left(t,+\infty;H^{1/2}\right)}\nonumber\\
\leq& C\left(r^{2}_{v}+r^{2}_{5/2}+\bar{r}_{2}\right),~~\mbox{for any}~t\geq t_{0}+1.
\end{align}
Consequently, combining (\ref{5.48}) with the first equation in (\ref{5.34}) and due to $\partial_{t}u=\partial_{t}w$, we know that
\begin{align}\label{5.49}
\|w\|_{L^{\infty}(t,+\infty;H^{3})}=\|A^{5/4}w\|_{L^{\infty}(t,+\infty;H^{1/2})}
\leq \nu^{-1}\left\{C\left(r^{2}_{v}+r^{2}_{5/2}
+\bar{r}^{2}_{2}\right)\right\},
\end{align}
for any $t\geq t'_{0}+1$, which indicates that
\begin{align*}
\|w(t)\|_{H^{9/2}}=&\|A^{5/4}w\|_{L^{\infty}(t,+\infty;H^{2})}\nonumber\\
\leq& \nu^{-1}\left\|\left(-B(w,w)-B(v,w)-B(w,v)-\partial_{t}u\right)\right\|_{L^{\infty}(t,+\infty;H^{2})}\nonumber\\
\leq& C_{\nu}\left(\|w\|^{2}_{L^{\infty}\left(t,+\infty;H^{3}\right)}
+2\|v\|_{H^{5/2}}\|w\|_{L^{\infty}\left(t,+\infty;H^{3}\right)}\right)
+\left\|\partial_{t}u\right\|_{L^{\infty}\left(t,+\infty;H^{2}\right)}\nonumber\\
\leq& C_{\nu}\left[\left(r^{2}_{v}+r^{2}_{5/2}\right)
+\bar{r}_{2}\right]:=\varrho^{2}_{3},
\end{align*}
for any $t\geq t'_{0}+1$, where we have used \eqref{5.49} and the Sobolev embedding $H^{2}(\mathbb{T}^{3})\hookrightarrow L^{\infty}(\mathbb{T}^{3})$ as well.
\end{proof}
\subsection{Well-posedness and attractors}\label{sec5.2}
\noindent

In this subsection, by using the a priori estimate achieved in Subsection \ref{sec5.1}, we prove that problems \eqref{2.4} and \eqref{5.33} possesses a global attractor in $H$. First of all, we recall that,
\begin{theorem}[\cite{GG18,HLT10,ILT06,CD18}]\label{thm5.5}
Let $f\in H^{-5/4}$ and $u_{0}\in H$. Then for any $T>0$ there exists a unique solution $(u,p)$ to Eqs. (\ref{5.1}) that satisfies $u\in C([0,T];H)\cap L^{2}([0,T];H^{5/4})$ and $\partial_{t}u\in L^{2}([0,T];H^{-5/4})$ such that
\begin{equation*}
(\partial_{t}u+\nu A^{5/4}u+B(u,u)-f,\varphi)_{H^{-5/4},H^{5/4}}=0
\end{equation*}
for every $\varphi \in H$ and almost every $t\in (0,T)$. Moreover, this solution depends continuously on the initial data $u_{0}$ in the $H$ norm.
\end{theorem}
\begin{proposition}\label{pro5.6}
The solution of \eqref{5.33} admits a continuous dynamical system $(S(t),H)$ and $S(t)$ is the corresponding nonlinear continuous semigroup.
\end{proposition}
\begin{proof}
It follows from Theorem \ref{thm5.5} that the problem (\ref{5.33}) is uniquely solvable for any $w_{0}:=u_{0}-v\in H$, because $u(t)$ solves Eqs. (\ref{2.4}) uniquely by Theorem \ref{thm5.5} and $v$ is fixed. As a consequence, we can define a continuous semigroup $S(t): H\rightarrow H$ by
\begin{equation*}
S(t)w_{0}=w(t),~~\forall~~t\in\mathbb{R}^{+},
\end{equation*}
where $w(t)$ is the solution of (\ref{5.33}) and $w_{0}=u_{0}-v$ is the corresponding initial datum.
\end{proof}
Next, we consider the existence of a global attractor for the problem (\ref{5.33}). Indeed, from Theorem \ref{thm5.4} and Proposition \ref{pro5.6}, we can obtain the $H$-attractor for the problem \eqref{2.4} and $\eqref{5.33}$ respectively.
\begin{theorem}\label{thm5.7}
Assume that $f\in H$ and $u_{0}\in H$. Let $(S(t),H)$ be a dynamical system generated by the solution of \eqref{5.33}. Then $(S(t),H)$ possesses a global attractor $\mathscr{A}$ in the phase space $H$. Moreover, it is bounded in $H^{9/2}$.
\end{theorem}
\begin{proof}
We easily see from the dissipative estimate (\ref{5.44}) that the ball
\begin{equation}\label{5.50}
\mathscr{B}_{3}:=\left\{w\in H:~\|w\|^{2}_{H^{9/2}}\leq \varrho_{3}\right\}
\end{equation}
is a bounded absorbing set of $(S(t),H)$. In addition, $S(t)$ is asymptotic compact in $H$ due to the fact that the embedding $H^{9/2}\hookrightarrow H$ is compact. By the global attractor's theory (see, e.g.,\cite{R01}), we know that $(S(t),H)$ possesses a global attractor $\mathscr{A}$ in $H$, which is a bounded set in $H^{9/2}$.
\end{proof}
\begin{corollary}\label{cor5.8}
Under the conditions stated in Theorem \ref{thm5.7}, the problem \eqref{2.4} has a global attractor $\mathscr{A}_{0}$. Moreover, we can decompose $\mathscr{A}_{0}=\mathscr{A}+v(x)$, where $v(x)$ is a stationary solution (equilibrium point) of \eqref{2.4} and $\mathscr{A}$ is bounded in $H^{9/2}$.
\end{corollary}
\subsection{Existence of the inertial manifold}\label{sec5.3}
\noindent

In this subsection, we prove the existence of IMs for the problem (\ref{5.33}). Since the exact value of $\nu$ is not essential for the existence of an IM, we assume that without loss of generality that $\nu=1$ in what follows. Firstly, we modify the nonlinearity outside the absorbing ball $\mathscr{B}_{3}$ in the next part.
\subsubsection{Modification of the nonlinearity outside the absorbing ball}\label{sec5.3.1}
\noindent

Although, we can write the hyperviscous NSEs as an abstract model as follows:
\begin{equation}\label{5.51}
\partial_{t}w+A^{5/4}w+A^{1/4}\mathfrak{F}(w)=0,
\end{equation}
where $A$ is the Stokes operator, $\mathfrak{F}(w)=A^{-1/4}\left(B(w,w)+B(v,w)+B(w,v)\right)$ and $w=u-v$,
the nonlinearity $\mathfrak{F}(w)$  still does not satisfy condition of the abstract results established in Section \ref{sec3}. However, we are mainly concerned with the long-time behaviors of a solution of (\ref{5.33}) and thus we may freely modify the nonlinearity outside the absorbing ball in $H^{9/2}$. In this part we adopt again the method introduced by Kostianko in \cite{K18} to truncate the nonlinearity of the Eqs. (\ref{5.33}). To do this, we define a vector-valued cut-off function as
\begin{equation*}
\vec{\eta}=(\eta(\xi_{1}),\eta(\xi_{2}),\eta(\xi_{3}))\in\mathbb{C}^{3},~~\mbox{for}~\xi=(\xi_{1},\xi_{2},\xi_{3}),
\end{equation*}
where $\eta(\cdot)\in C^{\infty}_{0}(\mathbb{C})$ is a smooth cut-off function satisfying
\begin{equation*}
\eta(\xi)=\xi~~\mbox{for}~~|\xi|\leq 1~~\mbox{and}~~|\eta(\xi)|\leq 2~~\mbox{for all}~~\xi\in\mathbb{C}.
\end{equation*}
Now, similar to Kostianko \cite{K18}, we introduce the operator $W: H\rightarrow H$ defined as
\begin{equation*}
W(w)=\sum_{j\in\mathbb{Z}^{3}_{\ast}}
\frac{\varrho_{3}}{|j|^{9/2}}P^{j}_{\sigma}\vec{\eta}\left(\frac{|j|^{9/2}\widehat{w}_{j}}{\varrho_{3}}\right)e^{ij\cdot x},
\end{equation*}
where $\varrho_{3}$ is the radius of the absorbing ball $\mathscr{B}$ defined by (\ref{5.50}), $P^{j}_{\sigma}$ are the Leray projector matrices defined by (\ref{2.2}) and
\begin{equation*}
w=\sum_{j\in\mathbb{Z}^{3}_{\ast}}\widehat{w}_{j}e^{ij\cdot x},
\end{equation*}
here we recall that $w=(w^{1},w^{2},w^{3})$ is a vector variable and each coefficient $\widehat{w}_{j}=(\widehat{w}^{1}_{j},\widehat{w}^{2}_{j},\widehat{w}^{3}_{j})\in \mathbb{C}^{3}$ is also a vector.
\par
Then, we modify $\mathfrak{F}(w)$ outside the absorbing ball in $H^{9/2}$ via replacing $w$ by $W(w)$. Namely, we define $F_{3}(W(\cdot))$ by
\begin{equation}\label{5.52}
F_{3}(W(w)):=\mathfrak{F}(W(w))=A^{-1/4}\left(B(W(w),W(w))+B(v,W(w))+B(W(w),v)\right).
\end{equation}
Hence, we consider the following ``prepared'' equation, which is a modification of \eqref{5.51}:
\begin{equation}\label{5.53}
\partial_{t}w+A^{5/4}w+A^{1/4}F_{3}(W(w))=0.
\end{equation}
Then, the ``prepared'' equation \eqref{5.53} has the same long-time behavior as the original problem \eqref{1.1} or \eqref{5.51}. Indeed,
\begin{proof}[\textbf{Proof of Theorem \ref{thm1.1} in the case $d=3$}]
It is sufficient to show that the truncation procedure does not change the equation on the attractor $\mathscr{A}$ or the absorbing set $\mathscr{B}_{3}$ at least, that is, $W(w)=w$ whenever $w$ lies in the absorbing set $\mathscr{B}_{3}$. Indeed, let $w\in \mathscr{B}_{3}$, i.e. $w\in H^{9/2}$ such that
\begin{equation*}
\|w\|_{H^{9/2}}=\sum_{j\in\mathbb{Z}^{3}_{\ast}}|j|^{9}|\widehat{w}_{j}|^{2}\leq \varrho_{3}^{2},
\end{equation*}
it follows that
\begin{equation*}
\frac{|j|^{9/2}\widehat{w}_{j}}{\varrho_{3}}\leq 1,~~\forall~j\in\mathbb{Z}^{3},j\neq 0.
\end{equation*}
Thus, by the definition of $\vec{\eta}$, we know that $\vec{\eta}\left(\frac{|j|^{9/2}\widehat{w}_{j}}{\varrho_{3}}\right)=\frac{|j|^{9/2}\widehat{w}_{j}}{\varrho_{3}}$ for all $j\in\mathbb{Z}^{3},j\neq 0$. This shows that $W(w)=w$ whenever $w\in \mathscr{B}_{3}$ according to the definition of $W(w)$.
\end{proof}
\subsubsection{Global boundedness and differentiability}\label{sec5.3.2}
\noindent

In this part, we show the truncation nonlinearity $F_{3}(W(w))$ possesses the properties of global boundedness, Lipschitz continuity and differentiability.
\begin{theorem}\label{thm5.11}
The modified nonlinearity $F_{3}(W(\cdot)): H\rightarrow H$, given by \eqref{5.52}, is globally bounded, globally Lipschitz continuous and Gateaux differentiable.
\end{theorem}
\begin{proof}
\textbf{Step 1.} To verify $F_{3}(W(\cdot))$ is globally bounded from $H$ to $H$, as that in Section \ref{sec3}, we give firstly the \textbf{claim A} that $W$ is a regularization operator: for any $\epsilon>0$, $W$ maps $H$ into $H^{3-\epsilon}$ continuously and there exists a constant $C_{\epsilon}$ such that
\begin{equation}\label{5.54}
\|W(w)\|_{H^{3-\epsilon}}\leq C_{\epsilon},~~\mbox{for all}~~w\in H.
\end{equation}
Then we can estimate
\begin{align*}
\|F_{3}(W(w))\|_{H}=& \|A^{-1/4}\left(B(W(w),W(w))+B(W(w),v)+B(v,W(w))\right)\|_{H}\nonumber\\
\leq& C\left(\|W(w)\|_{H^{1/2}}\|W(w)\|_{H^{2}}+\|W(w)\|_{H^{2}}\|v\|_{H^{1/2}}
+\|v\|_{H^{2}}\|W(w)\|_{H^{1/2}}\right)\nonumber\\
\leq& C\|W(w)\|_{H^{3-\epsilon}}\left(\|W(w)\|_{H^{3-\epsilon}}+2\|v\|_{H^{5/2}}\right)\nonumber\\
\leq& C_{\epsilon}\left(C_{\epsilon}+2\|v\|_{H^{5/2}}\right),~~(0<\epsilon<\frac{1}{2}),~~\forall~w\in~H,
\end{align*}
where we have used the regularity of the cut-off operator $W$ as in (\ref{5.54}) and the following preliminary results:
\begin{equation}\label{5.55}
\|B(X,Y)\|_{H}\leq C\|\nabla X\|_{H}\| Y\|_{L^{\infty}}\leq C'\|X\|_{H^{1}}\|Y\|_{H^{2}},
\end{equation}
for any $X\in H^{1}$ and $Y\in H^{2}$, where we have used Sobolev embedding $H^{2}\hookrightarrow L^{\infty}$. It remains to prove the \textbf{claim A} holds. Indeed, let $w\in H$. Since $|\eta(\xi)|\leq 2$ for all $\xi\in \mathbb{C}$, then by the definition of $W(w)$, we estimate
\begin{align*}
  \|W(w)\|^{2}_{H^{3-\epsilon}}
=&\sum_{j\in\mathbb{Z}^{3}_{\ast}}|j|^{6-2\epsilon}\cdot
\frac{\varrho_{3}^{2}}{|j|^{9}}\left|P^{j}_{\sigma}\vec{\eta}\left(\frac{|j|^{9/2}\widehat{w}_{j}}{\varrho_{3}}\right)\right|^{2}\nonumber\\
\leq& C\sum_{j\in\mathbb{Z}^{3}_{\ast}}|j|^{6-2\epsilon}\cdot\frac{\varrho_{3}^{2}}{|j|^{9}}
\leq C\varrho_{3}^{2}\sum_{j\in\mathbb{Z}^{3}_{\ast}}\frac{1}{|j|^{3+2\epsilon}}\leq C_{\epsilon}.
\end{align*}
It remains to show that $W: H\rightarrow H^{3-\epsilon}$ is continuous. Indeed, given $\delta>0$, let $w$, $\varpi\in H$ satisfy $\|w-\varpi\|_{H}\leq\sigma$. Then
\begin{align*}
  \|W(w)-W(\varpi)\|^{2}_{H^{3-\epsilon}}
=&\sum_{j\in\mathbb{Z}^{3}_{\ast}}
\frac{\varrho_{3}^{2}}{|j|^{3+2\epsilon}}\left|\vec{\eta}\left(\frac{|j|^{9/2}\widehat{w}_{j}}{\varrho_{3}}\right)
-\vec{\eta}\left(\frac{|j|^{9/2}\widehat{\varpi}_{j}}{\varrho_{3}}\right)\right|^{2}\nonumber\\
\leq& \sum_{\stackrel{j\in\mathbb{Z}^{3}}{1\leq|j|\leq N}}\frac{\varrho_{3}^{2}}{|j|^{3+2\epsilon}}\left|\vec{\eta}\left(\frac{|j|^{9/2}\widehat{w}_{j}}{\varrho_{3}}\right)
-\vec{\eta}\left(\frac{|j|^{9/2}\widehat{\varpi}_{j}}{\varrho_{3}}\right)\right|^{2}
+C\sum_{\stackrel{j\in\mathbb{Z}^{3}}{|j|\geq N}}\frac{\varrho_{3}^{2}}{|j|^{3+2\epsilon}},
\end{align*}
where we have used the fact that $|\vec{\eta}|^{2}\leq \frac{C}{2}$ due to $|\eta(\xi)|<2$ for all $\xi\in\mathbb{C}$. Since the series $\sum_{j\in\mathbb{Z}^{3}_{\ast}}\frac{1}{|j|^{3+2\epsilon}}$ is convergent, thus
\begin{align*}
 C\sum_{\stackrel{j\in\mathbb{Z}^{3}}{|j|\geq N}}\frac{\varrho_{3}^{2}}{|j|^{3+2\epsilon}}\leq \frac{\delta}{2},
\end{align*}
when we choose $N$ large enough. Moreover, since $\|w-\varpi\|_{H}\leq\sigma$, we have $|\widehat{w}_{j}-\widehat{\varpi}_{j}|\leq\sigma$ for every $j\in\mathbb{Z}^{3},j\neq 0$, then we can select $\sigma$ sufficiently small such that
\begin{align*}
\sum_{\stackrel{j\in\mathbb{Z}^{3}}{1\leq|j|\leq N}}
\frac{\varrho_{3}^{2}}{|j|^{3+2\epsilon}}\left|\vec{\eta}\left(\frac{|j|^{9/2}\widehat{w}_{j}}{\varrho_{3}}\right)
-\vec{\eta}\left(\frac{|j|^{9/2}\widehat{\varpi}_{j}}{\varrho_{3}}\right)\right|^{2}
\leq& C\sum_{\stackrel{j\in\mathbb{Z}^{3}}{1\leq|j|\leq N}}|j|^{9}|\widehat{w}_{j}-\widehat{\varpi}_{j}|^{2}
\leq \frac{\delta}{2},
\end{align*}
according to the fact that $\vec{\eta}$ is uniformly continuous and $N$ is fixed whenever it is chose above.
\par
\textbf{Step 2.} In order to prove $F_{3}(W(\cdot))$ is globally Lipschitz from $H$ to $H$, we give the \textbf{claim B} that the operator $W$ is globally Lipschitz from $H$ to $H$, i.e. there exists constant $L_{1}>0$ such that
\begin{equation}\label{5.56}
\|W(w)-W(\varpi)\|_{H}\leq L_{1}\|w-\varpi\|_{H},~~\forall~w,\varpi\in H.
\end{equation}
Indeed, let $w$, $\varpi\in H$. Since $\vec{\eta}'$ is globally bounded, we calculate
\begin{align*}
\|W(w)-W(\varpi)\|^{2}_{H}
\leq& \sum_{j\in\mathbb{Z}^{3}_{\ast}}
\frac{\varrho_{3}^{2}}{|j|^{9}}\left|\vec{\eta}\left(\frac{|j|^{9/2}\widehat{w}_{j}}{\varrho_{3}}\right)
-\vec{\eta}\left(\frac{|j|^{9/2}\widehat{\varpi}_{j}}{\varrho_{3}}\right)\right|^{2}\nonumber\\
\leq& \sum_{j\in\mathbb{Z}^{3}_{\ast}}|\vec{\eta}'(\xi)|^{2}|\widehat{w}_{j}-\widehat{\varpi}_{j}|^{2}
\leq L^{2}_{1}\|w-\varpi\|^{2}_{H}.
\end{align*}
Then, we calculate
\begin{align*}
F_{3}(W_{1})-F_{3}(W_{2})=& A^{-1/4}\left(B(W_{1},W_{1})+B(W_{1},v)+B(v,W_{1})
-B(W_{2},W_{2})+B(W_{2},v)+B(v,W_{2})\right)\nonumber\\
=& A^{-1/4}\left(B(W_{1}-W_{2},W_{1})+B(W_{2},W_{1}-W_{2})
+B(v,W_{1}-W_{2})+B(W_{1}-W_{2},v)\right),
\end{align*}
where $W_{i}=W(w_{i})$, $w_{i}\in H$, $i=1,2$. Note that the fact that $W: H\rightarrow H$ is globally Lipschitz with Lipschitz constant $L_{1}$ (see (\ref{5.56})), we can estimate
\begin{align*}
\|A^{-1/4}B(W_{1}&-W_{2},W_{1})\|_{H}+\|A^{-1/4}B(W_{1}-W_{2},v)\|_{H}
\leq C\|W_{1}-W_{2}\|_{H}\left(\|W_{1}\|_{H^{5/2}}+\|v\|_{H^{5/2}}\right)\nonumber\\
\leq &CL_{1}\|w_{1}-w_{2}\|_{H}\left(\|W_{1}\|_{H^{5/2}}+\|v\|_{H^{5/2}}\right)
\leq \frac{L}{2}\|w_{1}-w_{2}\|_{H},
\end{align*}
where the constant $L$ depends on $\alpha$, $L_{1}$ and $\|v\|_{H^{2}}$, here we have used the following estimate
\begin{equation*}
\|B(M,N)\|_{H}\leq C\|\nabla M\|_{L^{\infty}}\|N\|_{H}\leq C'\|M\|_{H^{3}}\|N\|_{H},~\forall~M\in H^{3},~N\in H,
\end{equation*}
where again we have used Sobolev embedding $H^{r}\hookrightarrow L^{\frac{6}{3-2r}}$ in $\mathbb{T}^{3}$ for any $0< r < 3/2$. In addition,
\begin{align*}
\|A^{-1/4}B(W_{2},W_{1}-W_{2})\|_{H}+&\|A^{-1/4}B(v,W_{1}-W_{2})\|_{H}
\leq C(\|W_{2}\|_{H^{5/2}}+\|v\|_{H^{5/2}})\|W_{1}-W_{2}\|_{H}\nonumber\\
\leq& CL_{1}(\|W_{2}\|_{H^{5/2}}+\|v\|_{H^{5/2}})\|w_{1}-w_{2}\|_{H}
\leq \frac{L}{2}\|w_{1}-w_{2}\|_{H},
\end{align*}
where we have used \eqref{5.56} and \eqref{5.55}. As a result, we conclude that
\begin{align*}
\|F_{3}(W(w_{1}))-F_{3}(W(w_{2}))\|_{H}\leq L\|w_{1}-w_{2}\|_{H},~~\forall~w_{1},w_{2}\in H.
\end{align*}
\par
\textbf{Step 3.} We give \textbf{claim C} that $W$ is Gateaux differentiable from $H$ to $H$, and its derivative $W'$ has the expression
\begin{equation}\label{5.57}
W'(w)z=
\sum_{j\in\mathbb{Z}^{3}_{\ast}}P^{j}_{\sigma}\vec{\eta}\left(\frac{|j|^{9/2}\widehat{w}_{j}}{\varrho_{3}}\right)\widehat{z}_{j}
e^{ij\cdot x},~~\mbox{for any}~~w,z\in H,
\end{equation}
where $\widehat{z}_{j}=(z,e^{ij\cdot x})$. Moreover, there exists $L_{1}>0$ such that
\begin{equation}\label{5.58}
  \|W'(w)\|_{\mathcal{L}(H,H)}\leq L_{1},
\end{equation}
for all $w\in H$. In addition, for each $z\in H$, the map $w\mapsto W'(w)z$ is continuous from $H$ to $H$.
\begin{proof}[Proof of \textbf{claim C}]
Let $w$, $z\in H$. We are ready to calculate the Gateaux differential of $W$. Note that
\begin{align*}
\frac{W(w+hz)-W(w)}{h}
=& \frac{1}{h}\sum_{j\in\mathbb{Z}^{3}_{\ast}}\frac{\varrho_{3}}{|j|^{9/2}}P^{j}_{\sigma}\left[\vec{\eta}
\left(\frac{|j|^{9/2}(\widehat{w}_{j}+h\widehat{z}_{j})}{\varrho_{3}}\right)-
\vec{\eta}\left(\frac{|j|^{9/2}\widehat{w}_{j}}{\varrho_{3}}\right)\right]e^{ij\cdot x}\nonumber\\
=& \sum_{j\in\mathbb{Z}^{3}_{\ast}}P^{j}_{\sigma}\left[\widehat{z}_{j}\int^{1}_{0}\vec{\eta}'
\left(\frac{|j|^{9/2}(\widehat{w}_{j}+sh\widehat{z}_{j})}{\varrho_{3}}\right)ds\right]e^{ij\cdot x},
\end{align*}
where we have used the fundamental theorem of calculus. We calculate
\begin{align*}
&\left\|\frac{W(w+hz)-W(w)}{h}-W'(w)z\right\|^{2}_{H}\nonumber\\
=& \left\|\sum_{j\in\mathbb{Z}^{3}_{\ast}}P^{j}_{\sigma}\left[\widehat{z}_{j}\int^{1}_{0}\left(
\vec{\eta}'\left(\frac{|j|^{9/2}(\widehat{w}_{j}+s h\widehat{z}_{j})}{\varrho_{3}}\right)
-\vec{\eta}'\left(\frac{|j|^{9/2}\widehat{w}_{j}}{\varrho_{3}}\right)\right)ds\right]e^{ij\cdot x}\right\|_{H}\nonumber\\
\leq& \sum_{j\in\mathbb{Z}^{3}_{\ast}}|\widehat{z}_{j}|^{2}\int^{1}_{0}
\left|\vec{\eta}'\left(\frac{|j|^{9/2}(\widehat{w}_{j}+s h\widehat{z}_{j})}{\varrho_{3}}\right)
-\vec{\eta}'\left(\frac{|j|^{9/2}\widehat{w}_{j}}{\varrho_{3}}\right)\right|^{2}ds\nonumber\\
\leq& \sum_{\stackrel{j\in\mathbb{Z}^{3}}{1\leq|j|\leq N}}|\widehat{z}_{j}|^{2}\int^{1}_{0}
\left|\vec{\eta}'\left(\frac{|j|^{9/2}(\widehat{w}_{j}+s h\widehat{z}_{j})}{\varrho_{3}}\right)
-\vec{\eta}'\left(\frac{|j|^{9/2}\widehat{w}_{j}}{\varrho_{3}}\right)\right|^{2}ds
+C\sum_{\stackrel{j\in\mathbb{Z}^{3}}{|j|> N}}|\widehat{z}_{j}|^{2},
\end{align*}
where we have used the fact that $|\vec{\eta}|^{2}\leq \frac{C}{2}$ due to $|\eta(\xi)|<2$ for all $\xi\in\mathbb{C}$. Let $\delta>0$, since $z\in H$, then
\begin{equation*}
 C\sum_{\stackrel{j\in\mathbb{Z}^{3}}{|j|> N}}|\widehat{z}_{j}|^{2}\leq \frac{\delta}{2},
\end{equation*}
when we choose $N$ large enough. Moreover, since $\vec{\eta}'$ is uniformly continuous, we may select $h$ sufficiently small such that
\begin{equation*}
\sum_{\stackrel{j\in\mathbb{Z}^{3}}{1\leq|j|\leq N}}|\widehat{z}_{j}|^{2}\int^{1}_{0}
\left|\vec{\eta}'\left(\frac{|j|^{9/2}(\widehat{w}_{j}+s h\widehat{z}_{j})}{\varrho_{3}}\right)
-\vec{\eta}'\left(\frac{|j|^{9/2}\widehat{w}_{j}}{\varrho_{3}}\right)\right|^{2}ds
\leq\frac{\delta}{2}.
\end{equation*}
This implies that
\begin{equation*}
\left\|\frac{W(w+hz)-W(w)}{h}-W'(w)z\right\|^{2}_{H}\leq \delta,
\end{equation*}
for sufficiently small $h$, provided $W'(w)z$ is given by (\ref{5.57}). Hence, $W$ is Gateaux differentiable from $H$ to $H$ and $W'$ can be given by the formula (\ref{5.57}).
\par
Furthermore, we see that $\|W'(w)z\|_{H}\leq L_{1}\|z\|_{H}$ for some $L_{1}>0$ via the formula (\ref{5.57}) and the uniform boundedness of $\vec{\eta}'$, which shows that (\ref{5.58}) is valid.
\par
Finally, we shall show that for each $z\in H$ the map $w\mapsto W'(w)z$ is continuous from $H$ to $H$. Indeed, given $\delta>0$, $w$, $\varpi\in H$ satisfies $\|w-\varpi\|_{H}\leq\sigma$. Then due to (\ref{5.57}), we estimate
\begin{align*}
\|(W'(w)-W'(\varpi))z\|^{2}_{H}
=&\sum_{j\in\mathbb{Z}^{3}_{\ast}}
\left|\vec{\eta}'\left(\frac{|j|^{9/2}\widehat{w}_{j}}{\varrho_{3}}\right)
-\vec{\eta}'\left(\frac{|j|^{9/2}\widehat{\varpi}_{j}}{\varrho_{3}}\right)\right|^{2}|\widehat{z}_{j}|^{2}\nonumber\\
\leq& \sum_{\stackrel{j\in\mathbb{Z}^{3}}{1\leq|j|\leq N}}\left|\vec{\eta}'\left(\frac{|j|^{9/2}\widehat{w}_{j}}{\varrho_{3}}\right)
-\vec{\eta}'\left(\frac{|j|^{9/2}\widehat{\varpi}_{j}}{\varrho_{3}}\right)\right|^{2}|\widehat{z}_{j}|^{2}
+C\sum_{\stackrel{j\in\mathbb{Z}^{3}}{|j|\geq N}}|\widehat{z}_{j}|^{2},
\end{align*}
where we have used the fact that $|\vec{\eta}|^{2}\leq \frac{C}{2}$ due to $|\eta(\xi)|<2$ for all $\xi\in\mathbb{C}$. Since $z\in H$, thus $C\sum_{j\in\mathbb{Z}^{3},|j|\geq N}\frac{\varrho_{3}^{2}}{|j|^{3+2\epsilon}}\leq \frac{\delta}{2}$ when we choose $N$ large enough. Also, since $\|w-\varpi\|_{H}\leq\sigma$, we have $|\widehat{w}_{j}-\widehat{\varpi}_{j}|\leq\sigma$ for every $j\in\mathbb{Z}^{3},j\neq 0$, then we can select $\sigma$ sufficiently small such that
$$\sum_{j\in\mathbb{Z}^{3},1\leq|j|\leq N}\left|\vec{\eta}'\left(\frac{|j|^{9/2}\widehat{w}_{j}}{\varrho_{3}}\right)
-\vec{\eta}'\left(\frac{|j|^{9/2}\widehat{\varpi}_{j}}{\varrho_{3}}\right)\right|^{2}|\widehat{z}_{j}|^{2}\leq \frac{\delta}{2},$$
according to the fact that $\vec{\eta}'$ is uniformly continuous and $N$ is fixed whenever it is chose above. Therefore, we conclude that $ \|(W'(w)-W'(\varpi))z\|^{2}_{H}\leq \delta$ as $\|w-\varpi\|_{H}\leq\sigma$.
\end{proof}
\emph{Continue to the proof of Theorem \ref{thm5.11}}. By the \textbf{Step 3}, analogously to the proof of (\ref{5.57}) and the chain rule, we know that $F_{3}(W(\cdot)): H\rightarrow H$ is Gateaux differential and the Gateaux derivative of $F_{3}(W)$ can be written as
\begin{equation}\label{5.59}
F'_{3}(W(w))z
=A^{-1/4}P_{\sigma}\left(((W'(w)z)\cdot \nabla)(\bar{W}(w)+\bar{v})+((W(w)+v)\cdot\nabla)(\bar{W}'(w)z)\right),~w,z\in H,
\end{equation}
which completes the proof.
\end{proof}
\par
Similarly to the proof of Theorem \ref{thm5.11}, we know that $F'_{3}(W(w))$ is a bounded linear operator in $H$ satisfying
\begin{equation*}
\|F'_{3}(W(w))\|_{\mathcal{L}(H,H)}\leq L~~~~\text{for all}~w\in H.
\end{equation*}
\subsubsection{Verification of the spatial averaging condition}\label{sec5.3.3}
\noindent

In this part, we will show the validity of the SAC for our problem \eqref{5.53}.
\begin{proposition}[Mallet-Paret and Sell \cite{M-PS88}]\label{pro5.12}
Let $b>0$. There exist arbitrarily large $\lambda$ and $k>C\log\lambda$ for some $C>0$ independent of $\lambda$ such that if $|j|^{2},|\ell|^{2}\in[\lambda-k,\lambda+k]$ with $j\neq\ell\in\mathbb{Z}^{3}$, then $|j-\ell|\geq b$.
\end{proposition}
Next, we will apply this proposition to verify SAC for the nonlinearity $F(W(w))$ given in \eqref{5.52}.
\begin{theorem}\label{thm5.13}
Let $\delta>0$ and $F'_{3}(W(w))$ be the operator given by $(\ref{5.59})$. Then there exist $N\in \mathbb{N}$ and $k\leq \lambda_{N}$ such that
\begin{equation}\label{5.60}
\|\mathcal{R}_{k,N}F'_{3}(W(w))\mathcal{R}_{k,N}z\|_{H}\leq \|z\|_{H}
\end{equation}
for any $w$, $z\in H$.
\end{theorem}
\begin{proof}
Let $U\in (L^{2}(\mathbb{T}^{3}))^{3}$ with zero mean, and define the multiplication operator $(T_{U}z)(x):=U(x)z(x)$ for any $z\in (L^{2}(\mathbb{T}^{3}))^{3}$ and call that $U$ is the generator of it. In order to estimate \eqref{5.60}, we let $b>0$ and denote
\begin{equation*}
U_{>b}:=\sum_{|j|^{2}>b}\widehat{U}_{j}e^{ij\cdot x},~~\mbox{and}~~U_{<b}:=\sum_{1\leq|j|^{2}\leq b}\widehat{U}_{j}e^{ij\cdot x}.
\end{equation*}
Thus $U=U_{>b}+U_{<b}$. Then, due to Proposition \ref{pro5.12} and $\langle U \rangle=0$, we know that there exist $N$ and $k$ such that
\begin{equation}\label{5.61}
\mathcal{R}_{k,N}\circ T_{U}\circ \mathcal{R}_{k,N}=\langle U\rangle\mathcal{R}_{k,N}+\mathcal{R}_{k,N}\circ T_{U_{>b}}\circ \mathcal{R}_{k,N}=\mathcal{R}_{k,N}\circ T_{U_{>b}}\circ \mathcal{R}_{k,N}.
\end{equation}
Indeed, by Weyl's law, we deduce that there exist arbitrarily large $\lambda_{N}$ and $k$ such that Proposition \ref{pro5.12} holds, it means that
\begin{align*}
\mathcal{R}_{k,N}\circ T_{U_{<b}}\circ \mathcal{R}_{k,N}z
=\mathcal{R}_{k,N}(U_{<b}\mathcal{R}_{k,N}z)
=\sum_{\lambda_{N}-k\leq|j|^{2}\leq\lambda_{N}+k}
\left(\sum_{\stackrel{\lambda_{N}-k\leq|\ell|^{2}\leq\lambda_{N}+k}{1\leq|j-\ell|\leq b}}\widehat{U}_{j-\ell}\widehat{z}_{\ell}
\right)e^{ij\cdot x}=0.
\end{align*}
One can see, e.g., \cite{M-PS88,K18} and \cite{GG18} for more details.
\par
Now we begin to estimate \eqref{5.60}, for simplicity, we denote by
\begin{equation*}
\hbar=\hbar(w)z:=W'(w)z
\end{equation*}
and then owing to \eqref{5.59},
\begin{align}\label{5.62}
F'_{3}(W(w))z
=A^{-\frac{1}{4}}\left(B(\hbar,W(w)+v)+B(W(w)+v,\hbar)\right).
\end{align}

Although we can not obtain $\langle F'_{3}(W)\rangle=0$ directly, we note that $F'_{3}(W)$ is a sum of multiplication operators which possess the properties that the mean of their generators are equal to zero. Indeed, by virtue of the specific form of $F'_{3}(W)z$ given in \eqref{5.62}, we define
\begin{align*}
\mathcal{J}(z):=\left(B(\hbar,W(w)+v)+B(W(w)+v,\hbar)\right),
\end{align*}
then by the definition of $B(\cdot,\cdot)$, we can write
\begin{align}\label{5.63}
\mathcal{J}(z)
=&\sum^{3}_{m,n=1}\left[\hbar^{m}\partial_{m}\left(W^{n}+v^{n}\right)
+\left(W^{m}+v^{m}\right)\partial_{m}\hbar^{n}\right].
\end{align}
Moreover,  since the mean values of $w(w)$ and $v$ are zero, then $\langle w^{m}\rangle=\langle v^{m}\rangle=0,~m=1,2,3$. In addition, for any $n\in\{1,2,3\}$ we can calculate directly that
\begin{align}\label{5.64}
\langle \partial_{1}W^{n}\rangle
=&\frac{1}{(2\pi)^{3}}\int_{\mathbb{T}^{3}}\partial_{1}W^{n}(x)dx\nonumber\\
=&\frac{1}{(2\pi)^{3}}\int_{\mathbb{T}^{2}}\left[W^{n}(\pi,x^{2},x^{3})-W^{n}(-\pi,x^{2},x^{3})\right]d(x^{2}\times x^{3})\nonumber\\
=&0,
\end{align}
and analogously,
\begin{equation}\label{5.65}
\left\{
  \begin{aligned}
    &\langle\partial_{m} W^{n} \rangle=0~~\text{for any}~m\in\{1,2\},~n\in\{1,2,3\},\\
    &\left\langle\partial_{m} v^{n}\right\rangle= 0~~\text{for any}~m,~n\in\{1,2,3\}.
  \end{aligned}
\right.
\end{equation}

Therefore, thanks to \eqref{5.61}-\eqref{5.65}, the equality \eqref{5.62} can be improved as the following form
\begin{align*}
\mathcal{R}_{k,N}(\mathcal{J}(\mathcal{R}_{k,N}z))
=&\mathcal{R}_{k,N}B(\mathcal{R}_{k,N}\hbar,W+v)+\mathcal{R}_{k,N}B(W+v,\mathcal{R}_{k,N}\hbar)\nonumber\\
=&\mathcal{R}_{k,N}B(\mathcal{R}_{k,N}\hbar,W_{>b}+v_{>b})
+\mathcal{R}_{k,N}B(W_{>b}+v_{>b},\mathcal{R}_{k,N}\hbar).
\end{align*}
Recalling that $W'(w)$ is a bounded operator in $H$, we deduce that
\begin{align}\label{5.66}
&\left\|\mathcal{R}_{k,N}(F'_{3}(W)\mathcal{R}_{k,N}z)\right\|_{H}\nonumber\\
=&\sup_{\stackrel{\varphi\in H}{\|\varphi\|_{H}=1}}\left|(\mathcal{R}_{k,N}(F'(W)\mathcal{R}_{k,N}z),\varphi)\right|\nonumber\\
=&\sup_{\stackrel{\varphi\in H}{\|\varphi\|_{H}=1}}\left|(\mathcal{R}_{k,N}(\mathcal{J}(\mathcal{R}_{k,N}z)),A^{-\frac{1}{4}}\varphi)\right|\nonumber\\
=&\sup_{\stackrel{\varphi\in H}{\|\varphi\|_{H}=1}}\left|(\mathcal{R}_{k,N}A^{-\frac{1}{4}}B(\mathcal{R}_{k,N}\hbar,W_{>b}
+v_{>b})+\mathcal{R}_{k,N}A^{-\frac{1}{4}}B(W_{>b}+v_{>b},\mathcal{R}_{k,N}\hbar),\varphi)\right|\nonumber\\
\leq& C\|W_{>b}+v_{>b}\|_{H^{5/2}}\|\mathcal{R}_{k,N}\hbar\|\nonumber\\
\leq& C(\alpha,L_{1})\|W_{>b}+v_{>b}\|_{H^{5/2}}\|z\|_{H}.
\end{align}
Furthermore, by the interpolation and the regular result (\ref{5.54}) of the operator $W$, we have
\begin{align}\label{5.67}
\|W_{>b}\|_{H^{5/2}}\leq& C \|W_{>b}\|^{1-r}\|W_{>b}\|^{r}_{H^{3-\epsilon}}
\leq C\left(d^{-(3-\epsilon)}\right)^{1-r}
\|W_{>b}\|^{1-r}_{H^{3-\epsilon}}\|W_{>b}\|^{r}_{H^{3-\epsilon}}\nonumber\\
=& d^{-(1-r)(3-\epsilon)}\|W_{>b}\|_{H^{3-\epsilon}}
\leq \tilde{C}d^{-(1-r)(3-\epsilon)},
\end{align}
where $r=\frac{5}{2(3-\epsilon)}$, $0<\epsilon<1/2$. Furthermore, $v_{>b}$ is compact in $H^{5/2}$ owning to the fact that $v$ is bounded in $H^{5/2}$, thus
\begin{equation}\label{5.68}
\|v_{>b}\|_{H^{5/2}}\leq \frac{\delta}{2},
\end{equation}
when $b$ is large enough. Therefore, substituting (\ref{5.67}) and (\ref{5.68}) into (\ref{5.66}) yields that
\begin{align*}
\big\|\mathcal{R}_{k,N}(F'_{3}(W)\mathcal{R}_{k,N}z)\big\|_{H}
\leq \tilde{C}\left(d^{-(1-r)(3-\epsilon)}+\|v_{>b}\|_{H^{5/2}}\right)\big\|z\big\|_{H}
\leq C\left(d^{-(1-r)(3-\epsilon)}+\frac{\delta}{2}\right)\big\|z\big\|_{H}.
\end{align*}
Then, the SAC is verified by taking $b$ large enough, that is,
\begin{align*}
\|\mathcal{R}_{k,N}(F'_{3}(w)\mathcal{R}_{k,N}z)\|_{H}\leq \delta\|z\|_{H},~~\mbox{for any}~~w,z\in H.
\end{align*}
\end{proof}

\subsubsection{Proof of Theorem \ref{thm1.4}}\label{sec5.3.4}

To complete the proof Theorem \ref{thm1.4}, we can employ the following supporting results, which is a corollary of Theorem \ref{thm1.3} (the case of $\alpha=1/4$), that is
\begin{corollary}\label{cor5.9}
Let $u_{1}$ and $u_{2}$ be two solutions of \eqref{5.53} in $H$. Set $v=u_{1}-u_{2}$. Assume that $F(\cdot): H\rightarrow H$ is Gateaux differentiable with $\|F'(u)\|_{\mathcal{L}(H,H)}\leq L$ for any $u\in H$ and for some $L\geq 1$. Suppose there exist $N\in\mathbb{N}$ and $k\in[\hbar \log\lambda_{N},\frac{\lambda_{N}}{2})$ for some $\hbar\in(0,1/2]$ such that $\lambda_{N}\geq e^{60L^{2}/\hbar}$ with $1\leq\lambda_{N+1}-\lambda_{N}\leq 2L$, and the SAC holds:
\begin{equation*}
\|R_{k,N}F'(u)R_{k,N}v\|_{H}\leq\delta\|v\|_{H},~~\mbox{for all}~~u\in H,
\end{equation*}
for some $\delta\leq\frac{1}{30}$. Then the following strong cone condition is valid,
\begin{equation*}
\frac{d}{dt}V(t)+(\lambda^{5/4}_{N+1}+\lambda^{5/4}_{N})V(t)
\leq-\frac{5\lambda^{1/4}_{N}}{16}\|v(t)\|_{H}^{2},~~\mbox{for all}~~t\geq 0.
\end{equation*}
\end{corollary}
\begin{remark}\label{rem5.10}
For the problem \eqref{5.53}, we have the following elementary inequality
\begin{align*}
a^{5}-b^{5}=&(a-b)(a^{4}+a^{3}b+a^{2}b^{2}+ab^{3}+b^{4})\nonumber\\
\geq& \frac{1}{6}(a-b)(a+b)^{4}=\frac{1}{6}(a^{2}-b^{2})(a+b)^{3}
\geq \frac{1}{6}(a^{4}-b^{4})(a+b),~\forall~0<a<b,
\end{align*}
where we have used the fact that $(a+b)^{4}=a^{4}+4a^{3}b+6a^{2}b^{2}+4ab^{3}+b^{4}\leq 6(a^{4}+a^{3}b+a^{2}b^{2}+ab^{3}+b^{4})$. This has more advantage than the elementary inequality \eqref{6.13} to construct an IM for \eqref{5.53}, although \eqref{6.13} is enough to verify the existence of an IM.
\end{remark}
\begin{proof}[\textbf{Proof of Theorem \ref{thm1.4}}]
Combining Corollary \ref{cor5.9}, Proposition \ref{pro4.4} and Theorem \ref{thm4.5} with Theorem \ref{thm5.13}, we can conclude that the ``prepared'' equation \eqref{5.53} possesses an $N$-dimensional inertial manifold $\mathcal{M}$ in the sense of Definition \ref{def2.4} which contains the global attractor $\mathscr{A}$, that is, we complete the proof of Theorem \ref{thm1.4}.
\end{proof}

\section{Appendix}\label{sec6}

\subsection{The proofs of the abstract model}\label{sec6.1}
\begin{proof}[Proof of Theorem \ref{thm4.1}]
Let $u_{1}$ and $u_{2}$ be two solutions of \eqref{4.1} in $\mathbb{H}$ and set $v=u_{1}-u_{2}$, then one has
\begin{equation}\label{6.1}
\partial_{t}v+\mathcal{A}^{1+\alpha}v+\mathcal{A}^{\alpha}\left(\mathcal{F}(u_{1})-\mathcal{F}(u_{2})\right)=0.
\end{equation}
Take the scalar product of (\ref{6.1}) with $\mathcal{A}^{-\alpha}p$ and $\mathcal{A}^{-\alpha}q$ respectively:
\begin{equation}\label{6.2}
\left\{
  \begin{aligned}
    &\frac{1}{2}\frac{d}{dt}\|p\|_{\mathbb{H}^{-\alpha}}^{2}+\|\mathcal{A}^{\frac{1}{2}}p\|^{2}
    +(\mathcal{F}(u_{1})-\mathcal{F}(u_{2}),p)=0,\\
    &\frac{1}{2}\frac{d}{dt}\|q\|_{\mathbb{H}^{-\alpha}}^{2}+\|\mathcal{A}^{\frac{1}{2}}q\|^{2}
    +(\mathcal{F}(u_{1})-\mathcal{F}(u_{2}),q)=0.
  \end{aligned}
\right.
\end{equation}
Define $\mathcal{V}(t)=\|q\|_{\mathbb{H}^{-\alpha}}^{2}-\|p\|_{\mathbb{H}^{-\alpha}}^{2}$. By subtracting the two equations in (\ref{6.2}), we have
\begin{equation}\label{6.3}
\frac{1}{2}\frac{d}{dt}\mathcal{V}(t)
+(\|\mathcal{A}^{\frac{1}{2}}q\|^{2}-\|\mathcal{A}^{\frac{1}{2}}p\|^{2})=(\mathcal{F}(u_{1})-\mathcal{F}(u_{2}),p-q)\leq L\|v\|^{2},
\end{equation}
here we have used the fact that $\mathcal{F}:\mathbb{H}\rightarrow \mathbb{H}$ is Lipschitz with Lipschitz constant $L$. Combining (\ref{6.3}) with the definition of $\mathcal{V}(t)$, we obtain
\begin{equation}\label{6.4}
\frac{1}{2}\frac{d}{dt}\mathcal{V}(t)+\gamma\mathcal{V}(t)
\leq-\left(\|\mathcal{A}^{\frac{1}{2}}q\|^{2}-\gamma\|q\|_{\mathbb{H}^{-\alpha}}^{2}\right)
-\left(\gamma\|p\|_{\mathbb{H}^{-\alpha}}^{2}\|\mathcal{A}^{\frac{1}{2}}p\|^{2}\right)+ L\|v\|^{2}.
\end{equation}
Thanks to $\gamma
=\frac{\lambda_{N}^{\alpha}\lambda_{N+1}^{\alpha}(\lambda_{N+1}+\lambda_{N})}{\lambda^{\alpha}_{N+1}+\lambda^{\alpha}_{N}}$, we estimate
\begin{equation}\label{6.5}
\|\mathcal{A}^{\frac{1}{2}}q\|^{2}-\gamma\|q\|_{\mathbb{H}^{-\alpha}}^{2}
\geq \left(\lambda_{N+1}-\gamma\lambda^{-\alpha}_{N+1}\right)\|q\|^{2}
=\frac{\lambda^{1+\alpha}_{N+1}-\lambda^{1+\alpha}_{N}}{\lambda^{\alpha}_{N+1}+\lambda^{\alpha}_{N}}\|q\|^{2}.
\end{equation}
Analogously,
\begin{equation}\label{6.6}
\gamma\|p\|_{\mathbb{H}^{-\alpha}}^{2}-\|\mathcal{A}^{\frac{1}{2}}p\|^{2}
\geq\left(\gamma\lambda^{-\alpha}_{N}-\lambda_{N}\right)\|p\|^{2}
=\frac{\lambda^{1+\alpha}_{N+1}-\lambda^{1+\alpha}_{N}}{\lambda^{\alpha}_{N+1}+\lambda^{\alpha}_{N}}\|p\|^{2}.
\end{equation}
Substituting (\ref{6.5}) and (\ref{6.6}) into (\ref{6.4}), we deduce that
\begin{equation}\label{6.7}
\frac{1}{2}\frac{d}{dt}\mathcal{V}(t)+\gamma\mathcal{V}(t)\leq-\mu\|v\|^{2}\leq-\mu\|v\|_{\mathbb{H}^{-\alpha}},
\end{equation}
where $\mu=\frac{\lambda^{1+\alpha}_{N+1}-\lambda^{1+\alpha}_{N}}{\lambda^{\alpha}_{N+1}+\lambda^{\alpha}_{N}}-L>0$ due to we assume the spectral gap condition (\ref{4.2}) holds. Notice that we can take the embedding constant of $\mathbb{H}\hookrightarrow \mathbb{H}^{-\alpha}$ is 1 in the estimate \eqref{6.7}.
\end{proof}
\par
The proof of Theorem \ref{thm4.2} will be divided into several lemmas:
\begin{lemma}\label{lem6.1}
Let $u_{1}$ and $u_{2}$ be two solutions of \eqref{4.1} and set $v=u_{1}-u_{2}$. Assume that $\mathcal{F}(\cdot): \mathbb{H}\rightarrow \mathbb{H}$ is Gateaux differentiable. Then
\begin{align}\label{6.8}
\frac{d}{dt}V(t)+2\gamma V(t)=:I_{V,p,q}-I_{p}-I_{q}+I_{\mathcal{F}},
\end{align}
where $\gamma=\frac{\lambda^{1+\alpha}_{N+1}+\lambda^{1+\alpha}_{N}}{2}$, $V(t):=V_{N}(u)=\|Q_{N}u\|^{2}-\|P_{N}u\|^{2}$ and
\begin{equation*}
\left\{
  \begin{aligned}
    &I_{V,p,q}=\left[\gamma V(t)-(\|\mathcal{A}^{\frac{1+\alpha}{2}}q\|^{2}-\|\mathcal{A}^{\frac{1+\alpha}{2}}p\|^{2})\right],\\
    &I_{p}=\left(\gamma\|p\|^{2}-\|\mathcal{A}^{\frac{1+\alpha}{2}}p\|^{2}\right),\\
    &I_{q}=\left(\|\mathcal{A}^{\frac{1+\alpha}{2}}q\|^{2}-\gamma\|q\|^{2}\right),\\
    &I_{\mathcal{F}}=2\int^{1}_{0}\left(\mathcal{F}'(su_{1}+(1-s)u_{2})v,\mathcal{A}^{\alpha}p-\mathcal{A}^{\alpha}q\right)ds.
  \end{aligned}
\right.
\end{equation*}
\end{lemma}
\begin{proof}
Let $u_{1}$ and $u_{2}$ be two solutions of \eqref{4.1} and set $v=u_{1}-u_{2}$, we have
\begin{equation}\label{6.9}
\partial_{t}v+\mathcal{A}^{1+\alpha}v+\mathcal{A}^{\alpha}[\mathcal{F}(u_{1})-\mathcal{F}(u_{2})]=0.
\end{equation}
Take the scalar product of (\ref{6.9}) with $p$ and $q$ respectively:
\begin{equation}\label{6.10}
\left\{
  \begin{aligned}
    &\frac{1}{2}\frac{d}{dt}\|p\|^{2}+\|\mathcal{A}^{\frac{1+\alpha}{2}}p\|^{2}
    +(\mathcal{F}(u_{1})-\mathcal{F}(u_{2}),\mathcal{A}^{\alpha}p)=0,\\
    &\frac{1}{2}\frac{d}{dt}\|q\|^{2}+\|\mathcal{A}^{\frac{1+\alpha}{2}}q\|^{2}
    +(\mathcal{F}(u_{1})-\mathcal{F}(u_{2}),\mathcal{A}^{\alpha}q)=0.
  \end{aligned}
\right.
\end{equation}
Then
\begin{equation}\label{6.11}
\frac{d}{dt}V(t)
=-2\left(\|\mathcal{A}^{\frac{1+\alpha}{2}}q\|^{2}-\|\mathcal{A}^{\frac{1+\alpha}{2}}p\|^{2}\right)
+2(\mathcal{F}(u_{1})-\mathcal{F}(u_{2}),\mathcal{A}^{\alpha}p-\mathcal{A}^{\alpha}q)
\end{equation}
The fundamental theorem of calculus for the Gateaux derivative gives that
\begin{equation}\label{6.12}
\mathcal{F}(u_{1})-\mathcal{F}(u_{2})=\int^{1}_{0}\mathcal{F}'(su_{1}+(1-s)u_{2})vds.
\end{equation}
Consequently, due to \eqref{6.11} and \eqref{6.12}, and by setting $\gamma:=\frac{\lambda^{1+\alpha}_{N+1}+\lambda^{1+\alpha}_{N}}{2}$, we see that the assertion in Lemma \ref{lem6.1} is valid.
\end{proof}
Next, we deal with each terms on the right-hand side of \eqref{6.8} by the following elementary inequality
\begin{equation}\label{6.13}
a^{1+\alpha}-b^{1+\alpha}=(1+\alpha)\int^{a}_{b}x^{\alpha}dx=(1+\alpha)\xi^{\alpha}(a-b)
\geq(1+\alpha)(a-b)\cdot b^{\alpha},
\end{equation}
for any real number $a\geq b\geq0$, where $\xi\in [b,a]$.
\begin{lemma}\label{lem6.2}
Let the assumptions of Theorem \ref{thm1.3} hold. Then $I_{V,p,q}$ satisfies the following estimate
\begin{equation}\label{6.14}
I_{V,p,q}\leq-\frac{(1+\alpha)}{2}\lambda^{\alpha}_{N}\|v\|^{2}.
\end{equation}
\end{lemma}
\begin{proof}
Note that
$\|v\|^{2}=\|p\|^{2}+\|q\|^{2},~\|\mathcal{A}^{\frac{1+\alpha}{2}}q\|^{2}\geq\lambda^{1+\alpha}_{N+1}\|q\|^{2}$ and $\|\mathcal{A}^{\frac{1+\alpha}{2}}p\|^{2}\leq\lambda^{1+\alpha}_{N}\|p\|^{2}$. Then
\begin{align}\label{6.15}
I_{V,p,q}\leq&\frac{\lambda^{1+\alpha}_{N+1}+\lambda^{1+\alpha}_{N}}{2}(\|q\|^{2}-\|p\|^{2})
-\lambda^{1+\alpha}_{N+1}\|q\|^{2}+\lambda^{1+\alpha}_{N}\|p\|^{2}\nonumber\\
=&-\frac{\lambda^{1+\alpha}_{N+1}-\lambda^{1+\alpha}_{N}}{2}(\|p\|^{2}+\|q\|^{2}).
\end{align}
According to the elementary inequality (\ref{6.13}), one has
\begin{equation}\label{6.16}
\frac{\lambda_{N+1}^{1+\alpha}-\lambda_{N}^{1+\alpha}}{2}
\geq\frac{(1+\alpha)}{2}(\lambda_{N+1}-\lambda_{N})\cdot \lambda_{N}^{\alpha}\geq\frac{(1+\alpha)}{2}\lambda^{\alpha}_{N},
\end{equation}
owing to the assumption $\lambda_{N+1}-\lambda_{N}\geq 1$. Applying (\ref{6.16}) to (\ref{6.15}) yields the desired result.
\end{proof}
\begin{lemma}\label{lem6.3}
Let the assumptions of Theorem \ref{thm1.3} hold. Then $I_{p}$ can be estimated as follows:
\begin{align}\label{6.17}
I_{p}\geq\frac{(1+\alpha)k}{2}\lambda_{N}^{\alpha}\|P_{k,N}v\|^{2}+\frac{\lambda_{N}^{\alpha}}{2}\|R_{k,N}p\|^{2}.
\end{align}
\end{lemma}
\begin{proof}
Note that $p=P_{k,N}v+R_{k,N}p=P_{k,N}v+R_{k,N}(P_{N}v)$ and $P_{k,N}\perp R_{k,N}\circ P_{N}$. Thus, due to
\begin{equation*}
\|P_{k,N}\mathcal{A}^{\frac{1+\alpha}{2}}v\|^{2}\leq(\lambda_{N}-k)^{1+\alpha}\|P_{k,N}v\|^{2}~~\mbox{and}~~ \|R_{k,N}\mathcal{A}^{1+\alpha}p\|^{2}\leq\lambda_{N}^{1+\alpha}\|R_{k,N}p\|^{2},
\end{equation*}
we estimate $I_{p}$ as follows:
\begin{align}\label{6.18}
I_{p}=&\frac{\lambda_{N+1}^{1+\alpha}+\lambda_{N}^{1+\alpha}}{2}\left(\|P_{k,N}v\|^{2}+\|R_{k,N}p\|^{2}\right)
-\left(\|P_{k,N}\mathcal{A}^{\frac{1+\alpha}{2}}v\|^{2}+\|R_{k,N}\mathcal{A}^{\frac{1+\alpha}{2}}p\|^{2}\right)\nonumber\\
\geq&\left[\frac{\lambda_{N+1}^{1+\alpha}+\lambda_{N}^{1+\alpha}}{2}-(\lambda_{N}-k)^{1+\alpha}\right]\|P_{k,N}v\|^{2}
+\frac{\lambda_{N+1}^{1+\alpha}-\lambda_{N}^{1+\alpha}}{2}\|R_{k,N}p\|^{2}.
\end{align}
Recalling that $\lambda_{N+1}\geq\lambda_{N}>2k$ and using the elementary inequality (\ref{6.13}), we estimate
\begin{align}\label{6.19}
&\frac{\lambda_{N+1}^{1+\alpha}+\lambda_{N}^{1+\alpha}}{2}-(\lambda_{N}-k)^{1+\alpha}\nonumber\\
\geq&\lambda_{N}^{1+\alpha}-(\lambda_{N}-k)^{1+\alpha}\geq(1+\alpha)(\lambda_{N}-\lambda_{N}+k)(\lambda_{N}-k)^{\alpha}
\geq\frac{(1+\alpha)k}{2}\lambda_{N}^{\alpha}.
\end{align}
Thus, applying (\ref{6.19}) and (\ref{6.16}) to (\ref{6.18}) deduces that
\begin{align}\label{6.20}
I_{p}\geq& \frac{(1+\alpha)k}{2}\lambda_{N}^{\alpha}\|P_{k,N}v\|^{2}+\frac{(1+\alpha)\lambda_{N}^{\alpha}}{2}\|R_{k,N}p\|^{2},
\end{align}
which gives \eqref{6.17}.
\end{proof}
\begin{lemma}\label{lem6.4}
Let the assumptions of Theorem \ref{thm1.3} hold. Then the term $I_{q}$ satisfies the following estimates:
\begin{align}\label{6.21}
I_{q}
=&\left(\|R_{k,N}\mathcal{A}^{\frac{1+\alpha}{2}}q\|^{2}-\gamma\|R_{k,N}q\|^{2}\right)
+\left(\|Q_{k,N}\mathcal{A}^{\frac{1+\alpha}{2}}v\|^{2}-\gamma\|Q_{k,N}v\|^{2}\right)\nonumber\\
=&:I_{q,1}+I_{q,2}
\end{align}
with
\begin{equation}\label{6.22}
I_{q,1}\geq\frac{(1+\alpha)}{2}\lambda_{N}^{\alpha}\|R_{k,N}q\|^{2}
\end{equation}
and
\begin{align}\label{6.23}
I_{q,2}
\geq(1+\alpha)(k-2L)\lambda_{N+1}^{\alpha}\|Q_{k,N}v\|^{2}
\end{align}
or
\begin{equation}\label{6.24}
I_{q,2}\geq\frac{\hbar\log\lambda_{N}}{\lambda_{N}^{\alpha}}\|Q_{k,N}\mathcal{A}^{\alpha}v\|^{2}.
\end{equation}
\end{lemma}
\begin{proof}
Indeed, since $q=Q_{k,N}v+R_{k,N}q=Q_{k,N}v+R_{k,N}(Q_{N}v)$ and $Q_{k,N}\perp R_{k,N}\circ Q_{N}$, we can calculate $I_{q}$ as follows:
\begin{align}\label{6.25}
I_{q}
=&\left(\|R_{k,N}\mathcal{A}^{\frac{1+\alpha}{2}}q\|^{2}-\gamma\|R_{k,N}q\|^{2}\right)
+\left(\|Q_{k,N}\mathcal{A}^{\frac{1+\alpha}{2}}v\|^{2}-\gamma\|Q_{k,N}v\|^{2}\right)\nonumber\\
=&:I_{q,1}+I_{q,2}.
\end{align}
By using (\ref{6.16}) and due to $\|R_{k,N}\mathcal{A}^{\frac{1+\alpha}{2}}q\|^{2}\geq\lambda_{N+1}^{1+\alpha}\|R_{k,N}q\|^{2}$, we get
\begin{equation}\label{6.26}
I_{q,1}\geq\frac{\lambda_{N+1}^{1+\alpha}-\lambda_{N}^{1+\alpha}}{2}\|R_{k,N}q\|^{2}
\geq\frac{(1+\alpha)}{2}\lambda_{N}^{\alpha}\|R_{k,N}q\|^{2}.
\end{equation}
Notice that $\|Q_{k,N}\mathcal{A}^{\frac{1+\alpha}{2}}v\|^{2}\geq(\lambda_{N}+k)^{1+\alpha}\|Q_{k,N}v\|^{2}$. Then using the elementary inequality (\ref{6.13}) and recalling that $\gamma=\frac{\lambda^{1+\alpha}_{N+1}+\lambda^{1+\alpha}_{N}}{2}\leq\lambda^{1+\alpha}_{N+1}$, we obtain
\begin{align}\label{6.27}
I_{q,2}\geq&\left[(\lambda_{N}+k)^{1+\alpha}-\lambda_{N+1}^{1+\alpha}\right]\|Q_{k,N}v\|^{2}\nonumber\\
\geq&(1+\alpha)(\lambda_{N}+k-\lambda_{N+1})\lambda_{N+1}^{\alpha}\|Q_{k,N}v\|^{2}\nonumber\\
\geq&(1+\alpha)(k-2L)\lambda_{N+1}^{\alpha}\|Q_{k,N}v\|^{2},
\end{align}
on account of the assumption $\lambda_{N+1}-\lambda_{N}\leq2L$, where we have used the fact that $\lambda_{N}+k\geq\lambda_{N}+2L\geq\lambda_{N+1}$.
\par
However, in order to control the nonlinearity, we may find another estimate for $I_{q,2}$ different from (\ref{6.27}). Indeed,
\begin{align}\label{6.28}
I_{q,2}=&\frac{\hbar\log\lambda_{N}}{\lambda_{N}}\|Q_{k,N}\mathcal{A}^{\frac{1+\alpha}{2}}v\|^{2}
+\left[\left(1-\frac{\hbar\log\lambda_{N}}{\lambda_{N}}\right)\|Q_{k,N}\mathcal{A}^{\frac{1+\alpha}{2}}v\|^{2}
-\gamma\|Q_{k,N}v\|^{2}\right]\nonumber\\
\geq&\frac{\hbar\log\lambda_{N}}{\lambda_{N}^{\alpha}}\|Q_{k,N}\mathcal{A}^{\alpha}v\|^{2}
+\underbrace{\left[\left(1-\frac{\hbar\log\lambda_{N}}{\lambda_{N}}\right)\|Q_{k,N}\mathcal{A}^{\frac{1+\alpha}{2}}v\|^{2}
-\gamma\|Q_{k,N}v\|^{2}\right]}_{J_{q}}.
\end{align}
Next, we shall show that there exists $N\in\mathbb{N}$ such that
\begin{align}\label{6.29}
J_{q}\geq 0.
\end{align}
Indeed, we can evaluate
\begin{align}\label{4.30}
J_{q}\geq& \left(1-\frac{\hbar\log\lambda_{N}}{\lambda_{N}}\right)(\lambda_{N}+k)^{1+\alpha}\|Q_{k,N}v\|^{2}
-\lambda_{N+1}^{1+\alpha}\|Q_{k,N}v\|^{2}\nonumber\\
\geq& \left[\left(1-\frac{\hbar\log\lambda_{N}}{\lambda_{N}}\right)(\lambda_{N}+\hbar\log\lambda_{N})^{1+\alpha}
-(\lambda_{N}+2L)^{1+\alpha}\right]\|Q_{k,N}v\|^{2},
\end{align}
for which we assume that $\lambda_{N}>k\geq\hbar\log\lambda_{N}$ for some $\hbar\in(0,1]$ and $\lambda_{N+1}\leq\lambda_{N}+2L$.
Due to the elementary inequality (\ref{6.13}), we estimate the coefficient of the last term in (\ref{4.30}) as follows:
\begin{align}\label{4.31}
&\left(1-\frac{\hbar\log\lambda_{N}}{\lambda_{N}}\right)\left(\lambda_{N}+\hbar\log\lambda_{N}\right)^{1+\alpha}
-(\lambda_{N}+2L)^{1+\alpha}\nonumber\\
=&\left[\left(1-\frac{\hbar\log\lambda_{N}}{\lambda_{N}}\right)^{\frac{1}{1+\alpha}}(\lambda_{N}+\hbar\log\lambda_{N})\right]^{1+\alpha}
-(\lambda_{N}+2L)^{1+\alpha}\nonumber\\
\geq&(1+\alpha)\left[\left(1-\frac{\hbar\log\lambda_{N}}{\lambda_{N}}\right)^{\frac{1}{1+\alpha}}(\lambda_{N}+\hbar\log\lambda_{N})
-(\lambda_{N}+2L)\right]\cdot(\lambda_{N}+2L)^{\alpha}.
\end{align}

Now, we claim that
\begin{equation}\label{4.32}
\left(1-\frac{\hbar\log\lambda_{N}}{\lambda_{N}}\right)^{\frac{1}{1+\alpha}}(\lambda_{N}+\hbar\log\lambda_{N})
-(\lambda_{N}+2L)\geq 0
\end{equation}
for some $N\in\mathbb{N}$ large enough. Indeed, $\left(1-\frac{\hbar\log\lambda_{N}}{\lambda_{N}}\right)^{\frac{1}{1+\alpha}}(\lambda_{N}+\hbar\log\lambda_{N})$ and $(\lambda_{N}+2L)$ are both nonnegative, it is sufficient to calculate
\begin{equation}\label{4.33}
\left(1-\frac{\hbar\log\lambda_{N}}{\lambda_{N}}\right)^{\frac{2}{1+\alpha}}(\lambda_{N}+\hbar\log\lambda_{N})^{2}
-(\lambda_{N}+2L)^{2}\geq 0.
\end{equation}
Due to $\alpha\in(0,1)$, we know that
\begin{align}\label{4.34}
&\left(1-\frac{\hbar\log\lambda_{N}}{\lambda_{N}}\right)^{\frac{2}{1+\alpha}}(\lambda_{N}+\hbar\log\lambda_{N})^{2}
-(\lambda_{N}+2L)^{2}\nonumber\\
\geq&\left(1-\frac{\hbar\log\lambda_{N}}{\lambda_{N}}\right)\left(\lambda_{N}^{2}+2\hbar\lambda_{N}\log\lambda_{N}
+(\hbar\log\lambda_{N})^{2}\right)-(\lambda_{N}^{2}+2L\lambda_{N}+4L^{2})\nonumber\\
=&\lambda_{N}^{2}+2\hbar\lambda_{N}\log\lambda_{N}+(\hbar\log\lambda_{N})^{2}-\hbar\lambda_{N}\log\lambda_{N}
-2(\hbar\log\lambda_{N})^{2}-\frac{(\hbar\log\lambda_{N})^{3}}{\lambda_{N}}\nonumber\\
&-(\lambda_{N}^{2}+2L\lambda_{N}+4L^{2})\nonumber\\
\geq&\hbar\lambda_{N}\log\lambda_{N}-2(\hbar\log\lambda_{N})^{2}-2L\lambda_{N}-4L^{2}.
\end{align}
If we set $\mathbb{G}(x)=\hbar x\log x-2(\hbar\log x)^{2}-2Lx-4L^{2}$, then $\mathbb{G}(x)>0$ for any $x\geq e^{16L^{2}/\hbar}$. Thus
\begin{equation}\label{4.35}
\hbar\lambda_{N}\log\lambda_{N}-2(\hbar\log\lambda_{N})^{2}-2L\lambda_{N}-4L^{2}\geq 0,
\end{equation}
since $\lambda_{N}\geq e^{60L^{2}/\hbar}$, where $L\geq 1$. Combining (\ref{4.34}) and (\ref{4.35}) with (\ref{4.33}), we know the claim (\ref{4.32}) is valid. Then the estimates (\ref{4.31})-(\ref{4.33}) indicate that the inequality (\ref{6.29}) is valid, and thus due to (\ref{6.28}), we obtain
\begin{equation}\label{4.36}
I_{q,2}\geq\frac{\hbar\log\lambda_{N}}{\lambda_{N}^{\alpha}}\|Q_{k,N}\mathcal{A}^{\alpha}v\|^{2}.
\end{equation}
\end{proof}
\begin{remark}\label{rem6.5}
By Lemma \ref{lem6.4}, we can conclude that
\begin{equation}\label{4.37}
I_{q}\geq\frac{\hbar\log\lambda_{N}}{2\lambda_{N}^{\alpha}}\|Q_{k,N}\mathcal{A}^{\alpha}v\|^{2}
+\frac{(1+\alpha)(k-2L)}{2}\lambda_{N+1}^{\alpha}\|Q_{k,N}v\|^{2}+\frac{\lambda_{N}^{\alpha}}{2}\|R_{k,N}q\|^{2}.
\end{equation}
Indeed, Substituting the estimates (\ref{6.12}), (\ref{6.27}) and (\ref{4.36}) into the equality (\ref{6.25}), we can see that \eqref{4.37} holds.
\end{remark}
\begin{lemma}\label{lem6.6}
Let the assumptions of Theorem \ref{thm1.3} hold. Then
\begin{align}\label{4.38}
\frac{I_{\mathcal{F}}}{2}\leq&|(\mathcal{F}'(u)v,\mathcal{A}^{\alpha}p-\mathcal{A}^{\alpha}q)|\nonumber\\
\leq& \left[\delta+\frac{2\delta^{2}}{\hbar\log\lambda_{N}}+2^{1+2\alpha}\delta^{2}
+\frac{1}{30}\right]\lambda_{N}^{\alpha}\|v\|^{2}
+\frac{\hbar\log\lambda_{N}}{4\lambda_{N}^{\alpha}}\|Q_{k,N}\mathcal{A}^{\alpha}v\|^{2}\nonumber\\
&+\left(1+2^{1+2\alpha}\right)L^{2}\lambda_{N}^{\alpha}\|Q_{k,N}v\|^{2}
+\left(6+2^{1+2\alpha}\right)L^{2}\lambda_{N}^{\alpha}\|P_{k,N}v\|^{2}\nonumber\\
&+\frac{\lambda_{N}^{\alpha}}{4}\|R_{k,N}q\|^{2}+\frac{\lambda_{N}^{\alpha}}{4}\|R_{k,N}p\|^{2},
\end{align}
for any $u\in \mathbb{H}$.
\end{lemma}
\begin{proof}
Indeed, we can calculate $( \mathcal{F}'(u)v,\mathcal{A}^{\alpha}p-\mathcal{A}^{\alpha}q)$ for any $u\in \mathbb{H}$ as follows.
\begin{align}\label{4.39}
&(\mathcal{F}'(u)v,\mathcal{A}^{\alpha}p-\mathcal{A}^{\alpha}q)\nonumber\\
=&(R_{k,N}\mathcal{F}'(u)v,\mathcal{A}^{\alpha}p-\mathcal{A}^{\alpha}q)+(P_{k,N}\mathcal{F}'(u)v,\mathcal{A}^{\alpha}p-\mathcal{A}^{\alpha}q)
+(Q_{k,N}\mathcal{F}'(u)v,\mathcal{A}^{\alpha}p-\mathcal{A}^{\alpha}q)\nonumber\\
=&(R_{k,N}\mathcal{F}'(u)R_{k,N}v,\mathcal{A}^{\alpha}p-\mathcal{A}^{\alpha}q)
+(R_{k,N}\mathcal{F}'(u)P_{k,N}v,\mathcal{A}^{\alpha}p-\mathcal{A}^{\alpha}q)\nonumber\\
&+(R_{k,N}\mathcal{F}'(u)Q_{k,N}v,\mathcal{A}^{\alpha}p-\mathcal{A}^{\alpha}q)+(P_{k,N}\mathcal{F}'(u)v,\mathcal{A}^{\alpha}p-\mathcal{A}^{\alpha}q)
+(Q_{k,N}\mathcal{F}'(u)v,\mathcal{A}^{\alpha}p-\mathcal{A}^{\alpha}q)\nonumber\\
=&\underbrace{(R_{k,N}\mathcal{F}'(u)R_{k,N}v,\mathcal{A}^{\alpha}p-\mathcal{A}^{\alpha}q)}_{J_{\mathcal{F},1}}
+\underbrace{(\mathcal{F}'(u)P_{k,N}v,R_{k,N}\mathcal{A}^{\alpha}p)}_{J_{\mathcal{F},2}}\nonumber\\
&\underbrace{-(\mathcal{F}'(u)P_{k,N}v,R_{k,N}\mathcal{A}^{\alpha}q)}_{J_{\mathcal{F},3}}
+\underbrace{(\mathcal{F}'(u)Q_{k,N}v,R_{k,N}\mathcal{A}^{\alpha}p)}_{J_{\mathcal{F},4}}
\underbrace{-(\mathcal{F}'(u)Q_{k,N}v,R_{k,N}\mathcal{A}^{\alpha}q)}_{J_{\mathcal{F},5}}\nonumber\\
&+\underbrace{(\mathcal{F}'(u)v,P_{k,N}\mathcal{A}^{\alpha}v)}_{J_{\mathcal{F},6}}
\underbrace{-(\mathcal{F}'(u)v,Q_{k,N}\mathcal{A}^{\alpha}v)}_{J_{\mathcal{F},7}}.
\end{align}
We deal with each term on the right-hand side of (\ref{4.39}). Firstly, since $\|\mathcal{F}'(u)\|_{\mathcal{L}(\mathbb{H},\mathbb{H})}\leq L$, one has
\begin{align}\label{6.40}
|J_{\mathcal{F},2}|+|J_{\mathcal{F},4}|\leq&L(\|P_{k,N}v\|+\|Q_{k,N}v\|)\|R_{k,N}\mathcal{A}^{\alpha}p\|\nonumber\\
\leq& L\lambda_{N}^{\alpha}(\|P_{k,N}v\|+\|Q_{k,N}v\|)\|R_{k,N}p\|\nonumber\\
\leq&L^{2}\lambda_{N}^{\alpha}(\|P_{k,N}v\|^{2}+\|Q_{k,N}v\|^{2})+\frac{\lambda_{N}^{\alpha}}{4}\|R_{k,N}p\|^{2}.
\end{align}
Analogously,
\begin{align}\label{6.41}
|J_{\mathcal{F},3}|+|J_{\mathcal{F},5}|\leq&L(\|P_{k,N}v\|+\|Q_{k,N}v\|)\|R_{k,N}\mathcal{A}^{\alpha}q\|\nonumber\\
\leq& L(\lambda_{N}+k)^{\alpha}(\|P_{k,N}v\|+\|Q_{k,N}v\|)\|R_{k,N}q\|\nonumber\\
\leq& L(2\lambda_{N})^{\alpha}(\|P_{k,N}v\|+\|Q_{k,N}v\|)\|R_{k,N}q\|\nonumber\\
\leq&2^{1+2\alpha}L^{2}\lambda_{N}^{\alpha}(\|P_{k,N}v\|^{2}+\|Q_{k,N}v\|^{2})+\frac{\lambda_{N}^{\alpha}}{8}\|R_{k,N}q\|^{2},
\end{align}
where we have used the assumption $k<\lambda_{N}$.
\par
In addition,
\begin{align}\label{6.42}
|J_{\mathcal{F},6}|\leq& L\|v\|\|P_{k,N}\mathcal{A}^{\alpha}v\|\leq L(\lambda_{N}-k)^{\alpha}\|v\|\|P_{k,N}v\|\nonumber\\
\leq& L\lambda_{N}^{\alpha}\|v\|\|P_{k,N}v\|
\leq 5L^{2}\lambda_{N}^{\alpha}\|P_{k,N}v\|^{2}+\frac{\lambda_{N}^{\alpha}}{20}\|v\|^{2}
\end{align}
and
\begin{align}\label{6.43}
|J_{\mathcal{F},7}|\leq L\|v\|\|Q_{k,N}\mathcal{A}^{\alpha}v\|
\leq \frac{2L^{2}\lambda_{N}^{\alpha}}{\hbar\log\lambda_{N}}\|v\|^{2}
+\frac{\hbar\log\lambda_{N}}{8\lambda_{N}^{\alpha}}\|Q_{k,N}\mathcal{A}^{\alpha}v\|^{2}.
\end{align}
\par
Furthermore, by the SAC (\ref{1.13}) and the fact that $q=Q_{k,N}v+R_{k,N}q$, we obtain
\begin{align}\label{6.44}
|J_{\mathcal{F},1}|\leq& \delta\|v\|\|\mathcal{A}^{\alpha}p-\mathcal{A}^{\alpha}q\|
\leq \delta\|v\|(\|\mathcal{A}^{\alpha}p\|+\|\mathcal{A}^{\alpha}q\|)\nonumber\\
\leq& \delta\lambda_{N}^{\alpha}\|v\|\|p\|+\delta\|v\|\|Q_{k,N}\mathcal{A}^{\alpha}q\|+\delta\|v\|\|R_{k,N}\mathcal{A}^{\alpha}q\|\nonumber\\
\leq& \delta\lambda_{N}^{\alpha}\|v\|^{2}+\delta\|v\|\|Q_{k,N}\mathcal{A}^{\alpha}q\|
+\delta(2\lambda_{N})^{\alpha}\|v\|\|R_{k,N}q\|\nonumber\\
\leq& \delta\lambda_{N}^{\alpha}\|v\|^{2}+\frac{2\delta^{2}\lambda_{N}^{\alpha}}{\hbar\log\lambda_{N}}\|v\|^{2}
+\frac{\hbar\log\lambda_{N}}{8\lambda_{N}^{\alpha}}\|Q_{k,N}\mathcal{A}^{\alpha}v\|^{2}
+2^{1+2\alpha}\delta^{2}\lambda_{N}^{\alpha}\|v\|^{2}+\frac{\lambda_{N}^{\alpha}}{8}\|R_{k,N}q\|^{2}\nonumber\\
=&\left[\delta+\frac{2\delta^{2}}{\hbar\log\lambda_{N}}+2^{1+2\alpha}\delta^{2}\right]\lambda_{N}^{\alpha}\|v\|^{2}
+\frac{\hbar\log\lambda_{N}}{8\lambda_{N}^{\alpha}}\|Q_{k,N}\mathcal{A}^{\alpha}v\|^{2}+\frac{\lambda_{N}^{\alpha}}{8}\|R_{k,N}q\|^{2}.
\end{align}
\par
Finally, applying these estimates $J_{\mathcal{F},i}~(i=1,2,\cdots,7)$ to the equality (\ref{4.39}), and concerning the fact that $2L^{2}/\hbar\log\lambda_{N}\leq 1/30$, we deduce that
\begin{align}\label{6.45}
|(\mathcal{F}'(u)v,\mathcal{A}^{\alpha}p-\mathcal{A}^{\alpha}q)|
\leq& \left[\delta+\frac{2\delta^{2}}{\hbar\log\lambda_{N}}+2^{1+2\alpha}\delta^{2}
+\frac{1}{30}\right]\lambda_{N}^{\alpha}\|v\|^{2}
+\frac{\hbar\log\lambda_{N}}{4\lambda_{N}^{\alpha}}\|Q_{k,N}\mathcal{A}^{\alpha}v\|^{2}\nonumber\\
&+\left(1+2^{1+2\alpha}\right)L^{2}\lambda_{N}^{\alpha}\|Q_{k,N}v\|^{2}
+\left(6+2^{1+2\alpha}\right)L^{2}\lambda_{N}^{\alpha}\|P_{k,N}v\|^{2}\nonumber\\
&+\frac{\lambda_{N}^{\alpha}}{4}\|R_{k,N}q\|^{2}+\frac{\lambda_{N}^{\alpha}}{4}\|R_{k,N}p\|^{2},
\end{align}
for any $u\in \mathbb{H}$.
\end{proof}
\begin{proof}[\textbf{Proof of Theorem \ref{thm4.2}}]
Applying the Lemmas \ref{lem6.1}-\ref{lem6.6}, we can verify that Theorem \ref{thm1.3} is valid. Indeed, we now substitute (\ref{6.16}), (\ref{6.20}), (\ref{4.37}) and (\ref{6.45}) into the equality (\ref{6.8}). It follows that
\begin{align}\label{6.46}
\frac{d}{dt}V(t)+2\gamma V(t)\leq& -\left[\frac{1+\alpha}{2}-2\delta-\frac{4(\delta^{2}+L^{2})}{\hbar\log\lambda_{N}}
-2^{2(1+\alpha)}\delta^{2}-\frac{1}{15}\right]\cdot\lambda_{N}^{\alpha}\|v\|^{2}\nonumber\\
&-\left[\frac{(1+\alpha)(k-2L)}{2}-2\left(1+2^{1+2\alpha}\right)L^{2}\right]\lambda_{N}^{\alpha}\|Q_{k,N}v\|^{2}\nonumber\\
&-\left[\frac{(1+\alpha)k}{2}-2\left(6+2^{1+2\alpha}\right)L^{2}\right]\lambda_{N}^{\alpha}\|P_{k,N}v\|^{2}.
\end{align}
By the assumption $k\geq\hbar\log\lambda_{N}$ and $\lambda_{N}\geq e^{60L^{2}/\hbar}$, we know that $k\geq 60L^{2}$ ($L\geq 1$), this implies that
\begin{align}\label{6.47}
\frac{(1+\alpha)(k-2L)}{2}-2\left(1+2^{1+2\alpha}\right)L^{2}>0~(\mbox{as}~k\geq 42L^{2}),
\end{align}
\begin{align}\label{6.48}
\frac{(1+\alpha)k}{2}-2(6+2^{1+2\alpha})L^{2}\geq 0~(\mbox{as}~k\geq 56L^{2}).
\end{align}
Since $\delta\leq\frac{1}{30}$ and $\hbar\log\lambda_{N}\geq 60L^{2}$ with $L\geq 1$, then clearly
\begin{equation}\label{6.49}
\frac{1+\alpha}{2}-2\delta-\frac{4(\delta^{2}+L^{2})}{\hbar\log\lambda_{N}}-2^{2(1+\alpha)}\delta^{2}-\frac{1}{15}\geq \frac{1+\alpha}{4}.
\end{equation}
As a consequence, by (\ref{6.46})-(\ref{6.49}), we conclude that
\begin{equation*}
\frac{d}{dt}V(t)+2\gamma V(t)\leq-\frac{(1+\alpha)\lambda_{N}^{\alpha}}{4}\|v\|^{2},~~\mbox{for all}~~t\geq 0,
\end{equation*}
where $\gamma=\frac{\lambda_{N+1}^{1+\alpha}+\lambda_{N}^{1+\alpha}}{2}$ as before. This completes the proof.
\end{proof}

\begin{proof}[The proof of Proposition \ref{pro4.4}]
Assume $V(v(t_{0}))=0$ for some $t_{0}\geq0$, then by (\ref{2.9}), we have
\begin{equation*}
\frac{d}{dt}V(v(t_{0}))\leq-\mu\|v(t_{0})\|^{2}\leq 0.
\end{equation*}
Then, $V(t)$ is non-increasing at $t=t_{0}$, and thus $V(v(t))\leq V(v(t_{0}))\leq 0$ for any $t\geq t_{0}$.
\par
Next, we show the squeezing property for the problem \eqref{4.1}. Let $u_{1}, u_{2}$ be two weak solutions of \eqref{4.1}. Set $v=u_{1}-u_{2}$, then (\ref{6.9}) holds. Moreover, since $\mathcal{F}$ is globally Lipschitz with Lipschitz constant $L$ and thanks to \eqref{6.9}, we have
\begin{align*}
\frac{1}{2}\frac{d}{dt}\|v(t)\|^{2}+\|\mathcal{A}^{(1+\alpha)/2}v(t)\|^{2}
\leq |(\mathcal{F}(u_{1})-\mathcal{F}(u_{2}),\mathcal{A}^{\alpha})|
\leq \frac{L^{2}}{2}\|v(t)\|^{2}+\frac{1}{2}\|\mathcal{A}^{\alpha}v(t)\|^{2}.
\end{align*}
Concerning $\alpha<1$ and the embedding inequality $\|\mathcal{A}^{\alpha}v(t)\|\leq \|\mathcal{A}^{(1+\alpha)/2}v(t)\|$, we see that
\begin{equation}\label{6.50}
\frac{d}{dt}\|v(t)\|^{2}\leq L^{2}\|v(t)\|^{2}.
\end{equation}
Define
\begin{equation*}
V_{\epsilon}(v(t))=V(v(t))+\epsilon\|v(t)\|^{2}.
\end{equation*}
Multiplying (\ref{6.50}) by $\epsilon$ and using the definition of strong cone condition, and due to (\ref{2.9}), we have
\begin{equation*}
\frac{d}{dt}V_{\epsilon}(V(t))+\gamma V_{\epsilon}(v(t))
\leq [\epsilon(\gamma+L^{2})-\mu]\|v(t)\|^{2}\leq0,
\end{equation*}
by choosing $\epsilon$ small enough, which follows that
\begin{equation}\label{6.51}
V_{\epsilon}(v(t))\leq Ce^{-\nu t}V_{\epsilon}(v(0)),
\end{equation}
for any $t\geq 0$ such that $V(t)\geq 0$, where $C$ depends only on $\nu_{0}$.
\par
If $V(T_{0})>0$ for some $T_{0}\geq 0$, then by the cone invariance, we have $V(t)\geq 0$ for all $t\in[0, T_{0}]$. Also, note that $V(v(t))=\|q\|^{2}-\|p\|^{2} \leq \|q\|^{2}+\|p\|^{2}=\|v\|^{2}$. Therefore, due to (\ref{6.51}),
\begin{align*}
\epsilon\|v(t)\|^{2}\leq \epsilon\|v(t)\|^{2}+V(v(t))\leq& e^{-\nu t}(\epsilon\|v(0)\|^{2}+V(v(0)))\\
\leq& C(1+\epsilon)e^{-\nu t}\|v(0)\|^{2},
\end{align*}
for all $t\in[0, T_{0}]$. Then, $\|v(t)\|^{2}\leq(1+\frac{1}{\epsilon})e^{-\nu t}\|v(0)\|^{2}$ for all $t\in [0, T_{0}]$.
\end{proof}

\begin{proof}[Proof of Theorem \ref{thm4.5}]
The proof can be found in \cite{Z14} for the reaction diffusion equation and 3D Cahn-Hilliard equation in \cite{KZ15}. For the problem \eqref{4.1}, we adopt the proof of Zelik in \cite{Z14}.
\par
\textbf{Step 1}. Let us consider the following boundary value problem:
\begin{equation}\label{6.52}
\begin{cases}
\partial_{t}u+\mathcal{A}^{5/4}u+\mathcal{A}^{1/4}\mathcal{F}(u)=g(x),~(t,x)\in \mathbb{R}_{+}\times\mathbb{T}^{3},\\
P_{N}u(0)=u^{+}_{0},~~Q_{N}u(-T)=0.
\end{cases}
\end{equation}
We claim that it has a unique solution for any $T>0$ and $u^{+}_{0}\in P_{N}\mathbb{H}$. Indeed, define the map $G_{T}: P_{N}\mathbb{H}\rightarrow P_{N}\mathbb{H}$ as follows:
\begin{equation}\label{6.53}
G_{T}(v)=P_{N}S(T)v,~~v\in P_{N}\mathbb{H},
\end{equation}
where $S(t)$ is the solution semigroup of problem \eqref{4.1}. Obviously, $G_{T}$ is continuous. For any $v^{1}, v^{2}\in P_{N}\mathbb{H}$ and associated trajectories $u_{i}(t)=S(t)v^{i}$, due to $Q_{N}u_{1}(-T)=Q_{N}u_{2}(-T)$ and the cone invariance, we have
\begin{equation*}
Q_{N}w\in K^{+}~\mbox{for all}~t\in [-T,0],~\mbox{where}~w(t):=u_{1}(t)-u_{2}(t),
\end{equation*}
and thus
\begin{align*}
\frac{1}{2}\frac{d}{dt}\|P_{N}w\|^{2}=&-\langle P_{N}\mathcal{A}^{5/4}w,P_{N}w\rangle+\langle \mathcal{F}(u_{1})-\mathcal{F}(u_{2}),A^{1/4}P_{N}w\rangle\nonumber\\
\geq& -\lambda^{5/8}_{N}\|P_{N}w\|^{2}-\lambda^{1/4}_{N}L(\|P_{N}w\|+\|Q_{N}w\|)\|P_{N}w\|\nonumber\\
\geq& -\lambda^{1/4}_{N}\left(\lambda^{3/8}_{N}+2L\right)\|P_{N}w\|^{2}.
\end{align*}
Integrating this inequality, we get
\begin{align*}
\|v^{1}-v^{2}\|\leq& e^{\lambda^{1/4}_{N}\left(\lambda^{3/8}_{N}+2L\right)T}\|G_{T}(v^{1})-G_{T}(v^{2})\|.
\end{align*}
Thus, the map $G_{T}: P_{N}\mathbb{H}\rightarrow P_{N}\mathbb{H}$ is injective and inverse is Lipschitz on its domain. Since $ P_{N}\mathbb{H}\sim \mathbb{R}$, by using the Brouwer theorem on the invariance of domain \cite{B1912} and noting that $G_{T}(P_{N}\mathbb{H})$ has empty boundary, we conclude that $G_{T}$ is a homeomorphism from $P_{N}\mathbb{H}$ to $P_{N}\mathbb{H}$, and therefore $G_{T}(v)=u^{+}_{0}$ is uniquely solvable for all $u^{+}_{0}\in P_{N}\mathbb{H}$. It remains to note that $u(t)=S(t+T)G^{-1}_{T}(u^{+}_{0})$ solves $(\ref{6.52})$.
\par
\textbf{Sept 2}. Let $u_{T}(t):=u_{T,u^{+}_{0}}(t)$ be the solution of the boundary value problem $(\ref{6.52})$. We claim that the limit
\begin{equation*}
u_{u^{+}_{0}}(t):=\lim\limits_{T\rightarrow \infty}u_{T,u^{+}_{0}}(t)
\end{equation*}
exists for all $t\in (-\infty,0]$ and solves problem (\ref{6.52}) with $T=\infty$. Indeed, let $w(t):=u_{T_{1}, u^{+}_{0}}(t)-u_{T_{2}, u^{+}_{0}}(t)$. Since $P_{N}w(0)=0$, we have $w(0)\notin K^{+}$. Then, by the cone invariance, we obtain that
\begin{equation}\label{6.54}
w(t)\not\in K^{+},~~\forall~~t\in[-\tilde{T},0]~~\mbox{and}~~\tilde{T}:=\min\{T_{1},T_{2}\}.
 \end{equation}
Thus, according to the squeezing property and (\ref{6.54}), we have
\begin{align}\label{6.55}
\left\|u_{T_{1}, u^{+}_{0}}(t)-u_{T_{2}, u^{+}_{0}}(t)\right\|
\leq& Ce^{-\gamma(t+\tilde{T})}\left\|u_{T_{1}, u^{+}_{0}}(-\tilde{T})-u_{T_{2}, u^{+}_{0}}(-\tilde{T})\right\|\nonumber\\
\leq& 2Ce^{-\gamma(t+\tilde{T})}\|Q_{N}w(-\tilde{T})\|\nonumber\\
\leq& 2Ce^{-\gamma(t+\tilde{T})}\|Q_{N}u_{\bar{T}}(-\tilde{T})\|,
\end{align}
for all $t\in[-\tilde{T},0]$, where $\bar{T}=\max\{T_{1},T_{2}\}$ and the last inequality holds for which we have $\|Q_{N}u_{\tilde{T}}(-\tilde{T})=0$. We now use the fact that $\mathcal{A}^{-3/8}\mathcal{F}(\cdot)$ is globally bounded in $\mathbb{H}$ to find a uniform bound for $\|Q_{N}u_{\bar{T}}(-\tilde{T})\|$. For any solution $u_{T}(t)$ of (\ref{6.52}), we obtain
\begin{align*}
\frac{1}{2}\frac{d}{dt}\|Q_{N}u_{T}(t)\|^{2}\leq&-\|\mathcal{A}^{5/8}Q_{N}u_{T}(t)\|^{2}+\langle \mathcal{F}(u_{T}(t)),\mathcal{A}^{1/4}Q_{N}u_{T}(t)\rangle+\langle g,Q_{N}u_{T}(t)\rangle\\
\leq&-\|\mathcal{A}^{5/8}Q_{N}u_{T}(t)\|^{2}+\|\mathcal{A}^{-3/8}\mathcal{F}(u_{T}(t))\|^{2}
+\frac{1}{2}\|A^{5/8}Q_{N}u_{T}(t)\|^{2}+\|\mathcal{A}^{-3/8}g\|^{2}\nonumber\\
\leq&-\frac{1}{2}\|\mathcal{A}^{5/8}Q_{N}u_{T}(t)\|^{2}+\widetilde{C},
\end{align*}
for all $t\in[-T,0]$, according to the assumption that $\|\mathcal{A}^{-3/8}\mathcal{F}(u_{T}(t))\|^{2}\leq C$ and $g\in \mathbb{H}$. As a consequence, by Gronwall-type Lemma, we get
\begin{equation*}
\|Q_{N}u_{T}(t)\|^{2}\leq e^{-(t+T)/2}\|Q_{N}u_{T}(-T)\|^{2}+\widetilde{C},~~t\geq-T.
\end{equation*}
In particular, since the solution of (\ref{6.52}) starts from $Q_{N}u_{T}(-T)=0$, we conclude that
\begin{equation*}
\|Q_{N}(t)u_{T,u^{+}_{0}}(t)\|^{2}\leq \widetilde{C},~~t\geq-T
\end{equation*}
for all $T>0$ and $u^{+}_{0}\in P_{N}\mathbb{H}$. It follows that
\begin{equation*}
\|Q_{N}u_{\bar{T}}(-\tilde{T})\|\leq \widetilde{C}.
\end{equation*}
Therefore, estimate (\ref{6.55}) shows that $u_{T,u^{+}_{0}}(t)$ is a cauchy sequence in $C_{loc}((-\infty,0);\mathbb{H})$. Thus the limit (\ref{6.53}) exists in $C_{loc}((-\infty,0);\mathbb{H})$. This shows that $u_{u^{+}_{0}}(t)$ is a backward solution of (\ref{6.52}) and it has a unique extension $u_{u^{+}_{0}}$ which solves $\partial_{t}u+\mathcal{A}^{5/4}u+\mathcal{A}^{1/4}\mathcal{F}(u)=g$ for all $t\geq 0$.
\par
\textbf{Step 3}. Define a set $\mathcal{N}\subset C_{loc}(\mathbb{R};\mathbb{H})$ as the set of all solutions of (\ref{6.52}) obtained in \textbf{Step 2} by using  the limiting process (\ref{6.53}). By the construction, this set $\mathcal{N}$ is strictly invariant
\begin{equation*}
S(t)\mathcal{N}=\mathcal{N},~(T(h)u)(t):=u(t+h),~h\in\mathbb{R}.
\end{equation*}
For any two trajectories $u_{1}(t),u_{2}(t)\in \mathcal{N}$, we have
\begin{equation*}
u_{1}(t)-u_{2}(t)\in K^{+},~~\mbox{for all}~~t\in \mathbb{R},
\end{equation*}
that is,
\begin{equation}\label{6.56}
\|Q_{N}(u_{1}(t)-u_{2}(t))\|\leq\|P_{N}(u_{1}(t)-u_{2}(t))\|,~~\mbox{for all}~~t\in \mathbb{R},
\end{equation}
The inequality (\ref{6.56}) can easily be shown by using the approximations $u_{T,u^{+}_{0,1}}$ and $u_{T,u^{+}_{0,2}}$ due to $Q_{N}u_{T,u^{+}_{0,1}}(-T)=Q_{N}u_{T,u^{+}_{0,2}}(-T)=0$, and passing to the limits as $T\rightarrow \infty$. Now, we define the map $\Phi: P_{N}\mathbb{H}\rightarrow Q_{N}\mathbb{H}$ by
\begin{equation*}
\Phi(u^{+}_{0}):=Q_{N}u(0),~u\in \mathcal{N},~P_{N}u(0)=u^{+}_{0},~t\geq0,
\end{equation*}
Then by (\ref{6.56}), we have
\begin{equation*}
\|\Phi(u^{+}_{0,1})-\Phi(u^{+}_{0,2}))\|\leq\|u^{+}_{0,1}-u^{+}_{0,2}\|,~~\mbox{for all}~~t\in \mathbb{R},
\end{equation*}
which shows that $\Phi: P_{N}\mathbb{H}\rightarrow Q_{N}\mathbb{H}$ is well defined and Lipschitz continuous with Lipschitz constant 1. Next, we give the Lipschitz manifold $\mathcal{M}$ as follows
\begin{equation*}
\mathcal{M}:=\{u^{+}_{0}+\Phi(u^{+}_{0}),~u^{+}_{0}\in P_{N}\mathbb{H}\}.
\end{equation*}
Then $\mathcal{M}$ is invariant thanks to the invariance of $\mathcal{N}$. As a result, the desired invariant manifold has been constructed.
\par
\textbf{Step 4}. To complete the proof, we need to check the exponential tracking. Indeed, let $u(t),~t\geq0$, be a forward trajectory of (\ref{6.52}). Let $T>0$ and $u_{T}\in \mathcal{N}$ be the solution of (\ref{6.52}) belonging to the invariant manifold such that
\begin{equation*}
P_{N}u(T)=P_{N}u_{T}(T).
\end{equation*}
Then, obviously, $u(T)-u_{T}(T)\not\in K^{+}$, and consequently $u(t)-u_{T}(t)\not\in K^{+}$ for all $t\in [0, T]$, we also know that $Q_{N}u_{T}(0)$ is uniformly bounded with respect to $T$. Moreover, since
\begin{equation*}
\|P_{N}(u(0)-u_{T}(0))\|\leq\|Q_{N}(u(0)-u_{T}(0))\|,
\end{equation*}
we see that the sequence $u_{T}(0)$ is uniformly bounded as $T\rightarrow \infty$. On the other hand, due to the squeezing property, we have
\begin{equation*}
\|u(t)-u_{T}(t)\|\leq Ce^{-\vartheta t}\|u(0)-u_{T}(0)\|,~t\in[0,T].
\end{equation*}
Since $P_{N}H$ is finite dimensional and $u_{T}(0)$ is bounded, we may assume without loss of generality that
\begin{equation}\label{6.57}
 u_{T}(0)\rightarrow \tilde{u}(0)
\end{equation}
and then the corresponding trajectory $\tilde{u}(t)\in \mathcal{N}$ and satisfies (\ref{6.57}) for all $t\geq 0$. Thus, the theorem is proved.
\end{proof}

\subsection{Some notes about the spatial averaging method}\label{sec6.2}

\begin{theorem}\label{thm6.7}
Let $u_{1}$ and $u_{2}$ be two solutions of \eqref{4.1} in $\mathbb{H}$ and set $v=u_{1}-u_{2}$. Assume that $\mathcal{F}(\cdot): \mathbb{H}\rightarrow \mathbb{H}$ is Gateaux differentiable with $\|\mathcal{F}'(u)\|_{\mathcal{L}(\mathbb{H},\mathbb{H})}\leq L$ for any $u\in \mathbb{H}$, and for some $L\geq 1$. Suppose there exist $N\in\mathbb{N}$ and $k\in[\hbar \log\lambda_{N},\frac{\lambda_{N}}{2})$ for some $\hbar\in(0,1/2]$ such that $\lambda_{N}\geq e^{60L^{2}/\hbar}$ with $1\leq\lambda_{N+1}-\lambda_{N}\leq 2L$, and the SAC holds:
\begin{equation*}
\|R_{k,N}\mathcal{F}'(u)R_{k,N}v-a(u)R_{k,N}v\|\leq\delta\|v\|,~~\mbox{for all}~~u\in \mathbb{H},
\end{equation*}
for some $\delta\leq\frac{1}{30}$, where $a(u)\in \mathbb{R}$ is a scalar depending continuously on $u\in \mathbb{H}$. Then the following strong cone condition is valid,
\begin{equation*}
\frac{d}{dt}V(t)+\left(\lambda^{1+\alpha}_{N+1}+\lambda^{1+\alpha}_{N}-L-\delta\right)V(t)
\leq-\frac{(1+\alpha)\lambda^{\alpha}_{N}}{4}\|v(t)\|^{2},~~\mbox{for all}~~t\geq 0.
\end{equation*}
\end{theorem}
\begin{proof}
We just need to make a small modification to the proof of Theorem \ref{thm1.3}. Indeed, we recall from \eqref{4.39} that
\begin{equation*}
(\mathcal{F}'(u)v,\mathcal{A}^{\alpha}p-\mathcal{A}^{\alpha}q)
=J_{\mathcal{F},1}+J_{\mathcal{F},2}+J_{\mathcal{F},3}+J_{\mathcal{F},4}+J_{\mathcal{F},5}+J_{\mathcal{F},6}+J_{\mathcal{F},7},
\end{equation*}
where $J_{\mathcal{F},i}$ $(i=1,\cdots, 7)$ are defined by \eqref{4.39}. Now, we just give a different estimate on $J_{\mathcal{F},1}$:
\begin{align*}
|J_{\mathcal{F},1}|=&|\left(R_{k,N}\mathcal{F}'(u)R_{k,N}v,\mathcal{A}^{\alpha}p-\mathcal{A}^{\alpha}q\right)|\nonumber\\
\leq& |\left(R_{k,N}\mathcal{F}'(u)R_{k,N}v-a(u)R_{k,N}v,\mathcal{A}^{\alpha}p-\mathcal{A}^{\alpha}q\right)|
+|\left(a(u)R_{k,N}v,\mathcal{A}^{\alpha}p-\mathcal{A}^{\alpha}q\right)|\nonumber\\
\leq& \|R_{k,N}\mathcal{F}'(u)R_{k,N}v-a(u)R_{k,N}v\|\|\mathcal{A}^{\alpha}p-\mathcal{A}^{\alpha}q\|
+|\left(a(u)R_{k,N}v,\mathcal{A}^{\alpha}p-\mathcal{A}^{\alpha}q\right)|\nonumber\\
\leq& \delta\lambda^{\alpha}_{N}\|v\|+a(u)V(t)+|a(u)|\lambda_{N}^{\alpha}(\|P_{k,N}v\|^{2}+\|Q_{k,N}v\|^{2}),
\end{align*}
Then, applying the Lemmas \ref{lem6.1}-\ref{lem6.6} again, we can obtain
\begin{align*}
\frac{d}{dt}V(t)+(2\gamma-a(u)) V(t)\leq& -\left[\frac{1+\alpha}{2}-2\delta-\frac{4(\delta^{2}+L^{2})}{\hbar\log\lambda_{N}}
-2^{2(1+\alpha)}\delta^{2}-\frac{1}{10}\right]\cdot\lambda_{N}^{\alpha}\|v\|^{2}\nonumber\\
&-\left[\frac{(1+\alpha)(k-2L)}{2}-2\left(1+2^{1+2\alpha}\right)L^{2}-(L+\delta)\right]\lambda_{N}^{\alpha}\|Q_{k,N}v\|^{2}\nonumber\\
&-\left[\frac{(1+\alpha)k}{2}-2\left(6+2^{1+2\alpha}\right)L^{2}-(L+\delta)\right]\lambda_{N}^{\alpha}\|P_{k,N}v\|^{2},
\end{align*}
thanks to the obvious fact that $|a(u)|\leq L+\delta$. It follows that
\begin{align*}
\frac{d}{dt}V(t)+\left[2\gamma-(L+\delta)\right] V(t)\leq-\frac{(1+\alpha)\lambda_{N}^{\alpha}}{4}\|v\|^{2},~~\mbox{for all}~~t\geq 0,
\end{align*}
when we take $N\in \mathbb{N}$ such that $\lambda_{N}\geq e^{60L^{2}/\hbar}$ (which implies that $k\geq 60 L^{2}$), where $\gamma=\frac{\lambda_{N+1}^{1+\alpha}+\lambda_{N}^{1+\alpha}}{2}$ as before.
\end{proof}
\section*{Acknowledgement}
The authors are thankful to Prof. Sergey Zelik and Dr. Anna Kostianko for fruitful discussions and sharing their insights and ideas.


\begin{thebibliography}{10}


\bibitem{A63}
J.P. Aubin,
\newblock Une th\'{e}or\`{e} de compacit\'{e},
\newblock {\em C.R. Acad. Paris} 256 (1963) 5042--5044.

\vspace{-0.25 cm}

\bibitem{BFR80}
J. Bardina, J. Ferziger, W. Reynolds,
\newblock Improved subgrid scale models for large eddy simulation,
\newblock {\em in Proceedings of the 13th AIAA Conference on Fluid and Plasma Dynamics}, 1980.
\vspace{-0.25 cm}

\bibitem{B1912}
L.E.J. Brouwer,
\newblock {Zur Invarianz des N-dimensionalen gebiets,
\newblock {\em Math. Ann.} 72 (1) (1912) 55-56.
\vspace{-0.25 cm}

\bibitem{CHOT05}
A. Cheskidov, D.D. Holm, E. Olson, E.S. Titi,
\newblock On a Leray-$\alpha$ model of turbulence,
\newblock {\em R. Soc. A, Math. Phys, Eng. Sci.} 461 (2005) 629--649.

\vspace{-0.25 cm}
\bibitem{CD00}
J.W. Cholewa, T. Dlotko,
\newblock Global Attractors in Abstract Parabolic Problems,
\newblock {\em Cambridge University Press, Cambridge}, 2000.
\vspace{-0.25 cm}

\bibitem{CD18}
J.W. Cholewa, T. Dlotko,
\newblock Fractional Navier-Stokes equations,
\newblock {\em Discrete Contin. Dyn. Syst. B} 23 (8) (2018) 2967--2988.
\vspace{-0.25 cm}

\bibitem{CL02}
I. Chueshov, I. Lasiecka,
\newblock Inertial manifolds for von Karman plate equations,
\newblock {\em Appl. Math. Optim.} 46 (2002) 179--207.
\vspace{-0.25 cm}

\bibitem{C15}
I. Chueshov,
\newblock Dynamics of quasi-stable dissipative systems,
\newblock {\em Springer, New York}, 2015.

\vspace{-0.25 cm}
\bibitem{CFNT89S}
P. Constantin, C. Foias, B. Nicolaenko, R. Temam,
\newblock Spectral barriers and inertial manifolds for dissipative partial differential equations,
\newblock {\em J. Dyn. Differ. Equations} 1 (1) (1989) 45-73.

\vspace{-0.25 cm}
\bibitem{CFNT89}
P. Constantin, C. Foias, B. Nicolaenko, R. Temam,
\newblock Inertial Manifolds for Dissipative Partial Differential Equations (Applied Mathematical Sciences, no. 70),
\newblock {\em Springer-Verlag, New York}, 1989.

\vspace{-0.25 cm}

\bibitem{D65}
Y. Dubinski\"{\i},
\newblock Weak convergence in nonlinear elliptic and parabolic equations,
\newblock {\em Mat. Sb. (N.S.)} 67 (109) (1965) 609--642.
\vspace{-0.25 cm}

\bibitem{FNST88}
C. Foias, B. Nicolaenko, G.R. Sell, R. Temam,
\newblock Inertial manifolds for the Kuramoto-Sivashinsky equation and an estimate of their lowest dimension,
\newblock {\em J. Math. Pures Appl.} 67 (3) (1988) 197-226.
\vspace{-0.25 cm}

\bibitem{FST85}
C. Foias, G.R. Sell, R. Temam,
\newblock Vari\'{e}t\'{e}s inertielles des \'{e}quations diff\'{e}rentielles dissipatives.
 (Inertial manifolds for dissipative differential equations),
\newblock {\em C. R. Acad. Sci. Ser.} I 301 (1985) 139--141.
\vspace{-0.25 cm}

\bibitem{FST88}
C. Foias, G.R. Sell, R. Temam,
\newblock Inertial manifolds for nonlinear evolutionary equations,
\newblock {\em J. Differential Equations} 73 (2) (1988) 309--353.
\vspace{-0.25 cm}

\bibitem{GG18}
C.G. Gal, Y.Q. Guo,
\newblock Inertial manifolds for the hyperviscous Navier-Stokes equations,
\newblock {\em J. Differential Equations} 265 (9) (2018) 4335--4374.
\vspace{-0.25 cm}

\bibitem{GC05}
A.Y. Goritskii, V.V. Chepyzhov,
\newblock Dichotomy property of solutions of quasilinear equations in problems on inertial manifolds,
\newblock {\em Sb. Math.} 196 (4) (2005) 485--511.
\vspace{-0.25 cm}

\bibitem{G08}
L. Grafakos,
\newblock Classical Fourier Analysis,
\newblock {\em Vol 2, New York, Springer}, 2008.

\vspace{-0.25 cm}
\bibitem{HGT15}
M.A. Hamed, Y.Q. Guo, E.S. Titi,
\newblock Inertial manifolds for certain sub-grid scale $\alpha$-models of turbulence,
\newblock {\em SIAM J. Appl. Dyn. Syst.} 14 (3) (2015) 1308--1325.
\vspace{-0.25 cm}

\bibitem{HLT10}
M. Holst, E. Lunasin, G. Tsogtgerel,
\newblock Analysis of a general family of regularized Navier-Stokes and MHD models,
\newblock {\em J. Nonlinear Sci.} 20 (5) (2010) 523--567.

\vspace{-0.25 cm}
\bibitem{ILT06}
A.A. Ilyin, E.M. Lunasin, E.S. Titi,
\newblock A modified-Leray-$\alpha$ subgrid scale model of turbulence,
\newblock {\em Nonlinearity} 19 (4) (2006) 879--897.
\vspace{-0.25 cm}

\bibitem{K18}
A. Kostianko,
\newblock Inertial manifolds for the 3D modified-Leray-$\alpha$ model with periodic boundary conditions,
\newblock {\em J. Dyn. Differ. Equations} 30 (1) (2018) 1--24.
\vspace{-0.25 cm}

\bibitem{KZ15}
A. Kostianko, S. Zelik,
\newblock Inertial manifolds for the 3D Cahn-Hilliard equations with periodic boundary conditions,
\newblock {\em Commun. Pure Appl. Anal.} 14 (5) (2015) 2069--2094.
\vspace{-0.25 cm}

\bibitem{K92}
M. Kwak,
\newblock Finite-dimensional inertial forms for the 2D Navier-Stokes equations,
\newblock {\em Indiana Univ. Math. J.} 41  (1992) 927--982.
\vspace{-0.25 cm}

\bibitem{LL06} W. Layton, R. Lewandowski,
\newblock On a well-posed turbulence model,
\newblock {\em Dicrete Contin. Dyn. Syst. B} 6 (1) (2006), 111--128.
\vspace{-0.25 cm}

\bibitem{LS19}
X. Li, C. Sun,
\newblock Inertial manifolds for the 3D modified-Leray-$\alpha$ model,
\newblock {\em J. Differential Equations (2019)}, https://doi.org/10.1016/j.jde.2019.09.001.
\vspace{-0.25 cm}

\bibitem{L69}
J.L. Lions,
\newblock Quelques m\'{e}thodes des probl\`{e}mes aux limites non lin\'{e}aires},
\newblock {\em Doud, Paris}, 1969.
\vspace{-0.25 cm}

\bibitem{M-PS88}
\newblock J. Mallet-Paret, G.R. Sell,
\newblock Inertial manifolds for reaction diffusion equations in higher space dimensions,
\newblock {\em J. Am. Math. Soc.} 1 (4) (1988) 805--866.
\vspace{-0.25 cm}

\bibitem{M91}
M. Miklavcic,
\newblock A sharp condition for existence of an inertial manifold,
\newblock {\em J. Dyn. Differ. Equations} 3 (3) (1991) 437--456.
\vspace{-0.25 cm}

\bibitem{MZ08}
A. Miranville, S. Zelik,
\newblock Attractors for dissipative partial differential equations in bounded and unbounded domains,
\newblock {\em in: Handbook of Differential Equations: Evolutionary Equations, vol. IV, Elsevier/North-Holland, Amsterdam}, 2008.
\vspace{-0.25 cm}

\bibitem{R82}
I. Richards,
\newblock On the gaps between numbers which are sums of two squares,
\newblock {\em Adv. Math.} 46 (1982) 1--2.
\vspace{-0.25 cm}

\bibitem{R01}
J. Robinson,
\newblock Infinite-dimensional Dynamical Systems,
\newblock {\em Camabridge University Press}, 2001.
\vspace{-0.25 cm}

\bibitem{R94}
A. V. Romanov,
\newblock Sharp estimates for the dimension of inertial manifolds for nonlinear parabolic equations,
\newblock {\em Izv. Math.} 43 (1) (1994) 31--47.
\vspace{-0.25 cm}

\bibitem{S10}
C. Sun,
\newblock Asymptotic regularity for some dissipative equations,
\newblock {\em J. Differential Equations} 248 (2) (2010) 342--362.
\vspace{-0.25 cm}

\bibitem{T89}
R. Temam,
\newblock Do inertial manifolds apply to turbulence?,
\newblock {\em Physica D} 37 (1) (1989) 146--152.
\vspace{-0.25 cm}

\bibitem{T95}
R. Temam,
\newblock Navier-Stokes Equations and Nonlinear Functional Analysis,
\newblock {\em Vol 66, Siam}, 1995.
\vspace{-0.25 cm}

\bibitem{T97}
R. Temam,
\newblock Infinite-Dimensional Dynamical System in Mechanics and Physics,
\newblock {\em second edition, Applied Mathematical Sciences, vol 68, Springer-Verlag, New York}, 1997.
\vspace{-0.25 cm}

\bibitem{TW93}
R. Temam, S. Wang,
\newblock Inertial forms of Navier-Stokes equations on the sphere,
\newblock {\em J. Funct. Anal.} 117 (1) (1993) 215--242.
\vspace{-0.25 cm}

\bibitem{Z14}
S. Zelik,
\newblock Inertial manifolds and finite-dimensional reduction for dissipative PDEs,
\newblock {\em Proc. Royal Soc. Edinburgh} 144 (6) (2014) 1245--1327.

\end{thebibliography}
\end{document}